\documentclass[12pt,english]{amsart}
\usepackage[T1]{fontenc}
\usepackage[latin9]{inputenc}
\usepackage{geometry}
\usepackage{babel}
\usepackage{amsthm}
\usepackage{amsbsy}
\usepackage{amstext}
\usepackage{amssymb}
\usepackage{graphicx}
\usepackage{esint}
\usepackage[unicode=true,pdfusetitle,
 bookmarks=true,bookmarksnumbered=false,bookmarksopen=false,
 breaklinks=false,pdfborder={0 0 0},backref=false,colorlinks=false]
 {hyperref}

\makeatletter

\newcommand{\noun}[1]{\textsc{#1}}

\numberwithin{equation}{section}
\numberwithin{figure}{section}
\theoremstyle{plain}
\newtheorem{thm}{\protect\theoremname}
  \theoremstyle{definition}
  \newtheorem{defn}[thm]{\protect\definitionname}
  \theoremstyle{plain}
  \newtheorem{cor}[thm]{\protect\corollaryname}
  \theoremstyle{plain}
  \newtheorem{lem}[thm]{\protect\lemmaname}
  \theoremstyle{plain}
  \newtheorem{prop}[thm]{\protect\propositionname}

\usepackage{psfrag}
\usepackage{url}
\usepackage[T1]{fontenc}
 
\usepackage{etoolbox}
\patchcmd{\subsubsection}{\normalfont}{\bfseries}{}{}
\patchcmd{\subsubsection}{\itshape}{}{}{}

\@ifundefined{showcaptionsetup}{}{%
 \PassOptionsToPackage{caption=false}{subfig}}
\usepackage{subfig}
\makeatother

  \providecommand{\corollaryname}{Corollary}
  \providecommand{\definitionname}{Definition}
  \providecommand{\lemmaname}{Lemma}
  \providecommand{\propositionname}{Proposition}
\providecommand{\theoremname}{Theorem}

\begin{document}

\title{Magnetic Schr\"odinger operators and Ma\~n\'e's critical value}

\author{Peter Herbrich}
\begin{abstract}
We study periodic magnetic Schr\"odinger operators on covers of closed
manifolds in relation to Ma\~n\'e's critical energy values of the
corresponding classical Hamiltonian systems. In particular, we show
that if the covering transformation group is amenable, then the bottom
of the spectrum is bounded from above by Ma\~n\'e's critical energy
value. We also determine the spectra for various homogeneous spaces
with left-invariant magnetic fields.
\end{abstract}
\maketitle
\global\long\def\multipliedBy{\,}

\global\long\def\baseManifold{M}

\global\long\def\metric{g}

\global\long\def\exteriorDifferential{d}

\global\long\def\volumeForm{\exteriorDifferential vol}

\global\long\def\volume#1{\mathrm{vol}(#1)}

\global\long\def\smoothFunctions#1{C^{\infty}(#1)}
 \global\long\def\smoothCompactlySupportedFunctions#1{C_{0}^{\infty}(#1)}

\global\long\def\smoothOneForms#1{\Omega^{1}(#1)}
 \global\long\def\smoothCompactlySupportedOneForms#1{\Omega_{0}^{1}(#1)}

\global\long\def\smoothTwoForms#1{\Omega^{2}(#1)}

\global\long\def\magneticField{\boldsymbol{\beta}}

\global\long\def\magneticFieldStrength{\boldsymbol{B}}

\global\long\def\magneticPotential{\boldsymbol{\alpha}}

\global\long\def\electricPotential{\boldsymbol{V}}

\global\long\def\magneticDifferential#1{\exteriorDifferential_{#1}}

\global\long\def\cover#1{\widehat{#1}}

\global\long\def\coverTwo#1{\widetilde{#1}}

\global\long\def\coveringTransformationGroup{G}

\global\long\def\groupElement{\gamma}

\global\long\def\abelianCover#1#2{\widehat{#1}_{#2}^{\mathrm{abel}}}

\global\long\def\universalCover#1#2{\widehat{#1}_{#2}^{\mathrm{univ}}}

\global\long\def\abelianCoverWithoutHat#1#2{#1_{#2}^{\mathrm{abel}}}

\global\long\def\universalCoverWithoutHat#1#2{#1_{#2}^{\mathrm{univ}}}

\global\long\def\manesCriticalValue{\boldsymbol{c}}

\global\long\def\eigenvalue{\lambda}

\global\long\def\groundStateEnergy{\lambda_{0}}

\global\long\def\groundStateEnergyOfContinuousSpectrum{\lambda_{0,\mathrm{cont}}}

\global\long\def\homomorphismGroup{\mathrm{Hom}}

\global\long\def\homologyGroup{\mathrm{H}}

\global\long\def\norm#1{|#1|}

\global\long\def\normThatAdapts#1{\left|#1\right|}

\global\long\def\innerProduct#1#2{\langle#1,#2\rangle}

\global\long\def\innerProductThatAdapts#1#2{\left\langle #1,#2\right\rangle }

\global\long\def\normComingFromInnerProduct#1{\Vert#1\Vert}

\global\long\def\normComingFromInnerProductThatAdapts#1{\left\Vert #1\right\Vert }

\global\long\def\adjoint#1{#1^{*}}

\global\long\def\closure#1{\overline{#1}}

\global\long\def\domain#1{\mathrm{Dom}(#1)}

\global\long\def\targetC{}

\global\long\def\conjugate#1{\overline{#1}}

\global\long\def\derivative#1{\mathcal{D}_{#1}}

\global\long\def\schroedingerOperator{\mathcal{H}}

\global\long\def\spectrum#1{\mathrm{spec}(#1)}

\global\long\def\essentialSpectrum#1{\mathrm{spec}_{\mathrm{ess}}(#1)}

\global\long\def\purePointSpectrum#1{\mathrm{spec}_{\mathrm{pp}}(#1)}

\global\long\def\continuousSpectrum#1{\mathrm{spec}_{\mathrm{cont}}(#1)}

\global\long\def\solvableGeometry{\mathrm{Sol}}

\global\long\def\momentumVariable#1#2{\mathrm{\xi}_{#1}^{#2}}

\global\long\def\re{\mathrm{Re}}

\global\long\def\im{\mathrm{Im}}

\global\long\def\sectionOfBundle#1{\Gamma(#1)}

\global\long\def\heisenbergGroup{\mathrm{Nil}}

\section{Introduction \label{First_Page_of_MSOaMCV}}

Ever since quantum mechanics was formulated by Heisenberg and Schr\"odinger
in the 1920s, it has been one of the main objectives of mathematical
physics to understand its relations with classical mechanics. We exploit
one instance of this interplay and relate Ma\~n\'e's critical energy
value of classical electromagnetic Hamiltonians with the ground state
energy of the associated magnetic Schr\"odinger operators.

More precisely, let $\baseManifold$ be a connected closed Riemannian
manifold that is equipped with an electric potential and a magnetic
field given by a smooth function $\electricPotential\in\smoothFunctions{\baseManifold,\mathbb{R}}$
and a closed $2$-form $\magneticField\in\smoothTwoForms{\baseManifold,\mathbb{R}}$,
respectively. Any regular cover $\cover{\pi}\colon\cover{\baseManifold}\to\baseManifold$,
for which $\cover{\magneticField}=\cover{\pi}^{*}\magneticField$
has a magnetic potential $\cover{\magneticPotential}\in\smoothOneForms{\cover{\baseManifold},\mathbb{R}}$
satisfying $\cover{\magneticField}=d\cover{\magneticPotential}$,
gives rise to a Hamiltonian $H_{\cover{\magneticPotential},\cover{\electricPotential}}\colon T^{*}\cover{\baseManifold}\to\mathbb{R}$
defined as $H_{\cover{\magneticPotential},\cover{\electricPotential}}(x,p)=\frac{1}{2}\left|p+\cover{\magneticPotential}\right|_{x}^{2}+\cover{\electricPotential}(x)$,
where $\cover{\electricPotential}=\cover{\pi}^{*}\electricPotential=\electricPotential\circ\cover{\pi}$.
Ma\~n\'e's critical value of the corresponding Lagrangian $L_{\cover{\magneticPotential},\cover{\electricPotential}}\colon T\cover{\baseManifold}\to\mathbb{R}$
is given by~\cite{Mane1997,ContrerasIturriagaPaternainPaternain1998,BurnsPaternain2002}
\[
\manesCriticalValue(L_{\cover{\magneticPotential},\cover{\electricPotential}})=\inf_{f\in\smoothFunctions{\cover{\baseManifold},\mathbb{R}}}\,\sup_{x\in\cover{\baseManifold}}\left(\frac{1}{2}\left|\cover{\magneticPotential}+df\right|_{x}^{2}+\cover{\electricPotential}(x)\right).
\]
For the sake of convenience, all covers are implicitly assumed to
be connected. In Section~\ref{sec:MCV_of_magnetic_Hamiltonians},
we generalize results from~\cite{PaternainPaternain1997,FathiMaderna2007}
concerning the exact case in which $\cover{\magneticPotential}=\cover{\pi}^{*}\magneticPotential$
for some $\magneticPotential\in\smoothOneForms{\baseManifold,\mathbb{R}}$.
We prove that any regular cover $\coverTwo{\baseManifold}$ of $\baseManifold$
whose covering transformation group $\coveringTransformationGroup$
is amenable satisfies 
\begin{equation}
\manesCriticalValue(L_{\coverTwo{\magneticPotential},\coverTwo{\electricPotential}})=\manesCriticalValue(L_{\cover{\magneticPotential},\cover{\electricPotential}})=\min_{[\omega]\in\homologyGroup^{1}(\baseManifold,\mathbb{R})\colon\cover{\omega}\textrm{ is exact}}\manesCriticalValue(L_{\magneticPotential-\omega,\electricPotential}),\label{eq:MCV_on_abelian_covers}
\end{equation}
where $L_{\coverTwo{\magneticPotential},\coverTwo{\electricPotential}}$,
$L_{\cover{\magneticPotential},\cover{\electricPotential}}$ and $\cover{\omega}$
denote the lifts of $L_{\magneticPotential,\electricPotential}$ and
$\omega$ to $\coverTwo{\baseManifold}$ and to a subcover $\cover{\baseManifold}$
of $\coverTwo{\baseManifold}$ whose covering transformation group
is isomorphic to the abelianization $\coveringTransformationGroup/[\coveringTransformationGroup,\coveringTransformationGroup]$.

In Section~\ref{sec:MSO}, we study the quantum analogue of magnetic
Hamiltonians of the form $H_{\cover{\magneticPotential},\cover{\electricPotential}}$,
that is, magnetic Schrödinger operators $\schroedingerOperator_{\cover{\magneticPotential},\cover{\electricPotential}}$
initially defined on $\smoothCompactlySupportedFunctions{\cover{\baseManifold},\mathbb{C}}$
as 
\[
\schroedingerOperator_{\cover{\magneticPotential},\cover{\electricPotential}}u=\frac{1}{2}\Delta u-i\innerProduct{\exteriorDifferential u}{\cover{\magneticPotential}}+\left(\frac{1}{2}\multipliedBy\exteriorDifferential^{*}\cover{\magneticPotential}+\frac{1}{2}\norm{\cover{\magneticPotential}}^{2}+\cover{\electricPotential}\right)u.
\]
It is known~\cite{Shubin2001,BravermanMilatovichShubin2002}, that
for arbitrary $\cover{\magneticPotential}$ and periodic $\cover{\electricPotential}=\cover{\pi}^{*}\electricPotential$
as above, the closure $\closure{\schroedingerOperator_{\cover{\magneticPotential},\cover{\electricPotential}}}$
is a self-adjoint operator in $L^{2}(\cover{\baseManifold},\mathbb{C})$.
The main object of study is the ground state energy defined as
\[
\groundStateEnergy(\cover{\magneticPotential},\cover{\electricPotential})=\inf\multipliedBy\spectrum{\closure{\schroedingerOperator_{\cover{\magneticPotential},\cover{\electricPotential}}}}.
\]
We generalize various results which are either known for $\cover{\magneticPotential}=0$
or $\cover{\baseManifold}=\mathbb{R}^{n}$. The exact case allows
for a detailed spectral analysis in terms of twisted operators and
representations of the covering transformation group of $\cover{\baseManifold}$~\cite{Sunada1989,KobayashiOnoSunada1989}.
We moreover prove theorems about amenable and abelian covers, which
are motivated by corresponding results for the discrete analogue of
$\schroedingerOperator_{\cover{\magneticPotential},\cover{\electricPotential}}$
on periodic graphs~\cite{HiguchiShirai1999,HiguchiShirai2001}. In
particular, abelian covers $\cover{\pi}\colon\cover{\baseManifold}\to\baseManifold$
satisfy 
\begin{equation}
\groundStateEnergy(\cover{\magneticPotential},\cover{\electricPotential})=\min_{[\omega]\in2\pi\multipliedBy\homologyGroup^{1}(\baseManifold,\cover{\baseManifold},\mathbb{Z})}\groundStateEnergy(\magneticPotential-\omega,\electricPotential),\label{eq:GSE_on_abelian_covers}
\end{equation}
where
\[
\homologyGroup^{1}(\baseManifold,\cover{\baseManifold},\mathbb{Z})=\left\{ [\omega]\in\homologyGroup^{1}(\baseManifold,\mathbb{R})\,\Big|\,\int_{\gamma}\cover{\omega}\in\mathbb{Z}\textrm{ for any closed curve }\gamma\textrm{ in }\cover{\baseManifold}\right\} .
\]
The structural resemblance of (\ref{eq:MCV_on_abelian_covers}) and
(\ref{eq:GSE_on_abelian_covers}) motivated the main result in Section~\ref{sec:MCV_and_the_GSE},
which generalizes~\cite[Theorem B]{Paternain2001} to non-compact
amenable covers as follows.

\theoremstyle{theorem} \newtheorem*{GSE_smaller_MCV_Theorem}{Theorem \ref{thm:GSE_smaller_MCV_for_amenable_CTG}}
\begin{GSE_smaller_MCV_Theorem}

Let $\cover{\baseManifold}$ be a regular cover of $\baseManifold$
with amenable covering transformation group. If $\cover{\magneticPotential}\in\smoothOneForms{\cover{\baseManifold},\mathbb{R}}$
and $\cover{\electricPotential}\in\smoothFunctions{\cover{\baseManifold},\mathbb{R}}$
are potentials with $\inf\,\cover{\electricPotential}>-\infty$, then
the associated Lagrangian $L_{\cover{\magneticPotential},\cover{\electricPotential}}$
and the associated magnetic Schr\"odinger operator $\schroedingerOperator_{\cover{\magneticPotential},\cover{\electricPotential}}$
satisfy
\[
\groundStateEnergy(\cover{\magneticPotential},\cover{\electricPotential})\leq\manesCriticalValue(L_{\cover{\magneticPotential},\cover{\electricPotential}}).
\]

\end{GSE_smaller_MCV_Theorem}

In the last section, we explicitly determine $\spectrum{\closure{\schroedingerOperator_{\cover{\magneticPotential},\cover{\electricPotential}}}}$
on various covers $\cover{\baseManifold}$ of compact homogeneous
spaces $\baseManifold=\Lambda\backslash\Gamma$, which facilitates
comparisons with the corresponding classical data. In each case, $\Gamma$
is a Lie group that is equipped with a left-invariant metric and a
left-invariant magnetic field $\magneticField\in\smoothTwoForms{\Gamma,\mathbb{R}}$,
and $\Lambda\subset\Gamma$ is a cocompact lattice. The following
exact examples have not been studied before and exhibit unexpected
phenomena:
\begin{itemize}
\item $PSL(2,\mathbb{R})$ has a left-invariant magnetic potential $\cover{\magneticPotential}$,
such that $\groundStateEnergy(\magneticFieldStrength\cover{\magneticPotential},0)>\manesCriticalValue(L_{\magneticFieldStrength\cover{\magneticPotential},0})$
near $\magneticFieldStrength=0$, and the mapping $\magneticFieldStrength\mapsto\groundStateEnergy(\magneticFieldStrength\cover{\magneticPotential},0)$
has $2$ non-trivial local minima. The corresponding classical dynamics
are well-understood~\cite{CieliebakFrauenfelderPaternain2010}.
\item The compact quotient $\Lambda\backslash\heisenbergGroup$ of the Heisenberg
group $\heisenbergGroup$ by the lattice $\Lambda$ of integer matrices
has a magnetic potential $\magneticPotential$, such that the mapping
$\magneticFieldStrength\mapsto\groundStateEnergy(\magneticFieldStrength\cover{\magneticPotential}^{[\Lambda,\Lambda]\backslash\heisenbergGroup},0)$
has countably many local minima, where $\cover{\magneticPotential}^{[\Lambda,\Lambda]\backslash\heisenbergGroup}$
denotes the lift of $\magneticPotential$ to the maximal abelian cover
$[\Lambda,\Lambda]\backslash\heisenbergGroup$. This is the first
example in which the ground state energy is an unbounded function
of the strength of the magnetic field with infinitely many local minima.
\end{itemize}
Irrespective of amenability of $\pi_{1}(\baseManifold)$, all exact
examples in Section~\ref{sec:Hom_Ex_3_Body_Problem} satisfy
\[
\groundStateEnergy(\magneticFieldStrength\cover{\magneticPotential},0)=\groundStateEnergy(0,0)+\manesCriticalValue(L_{\magneticFieldStrength\cover{\magneticPotential},0})\qquad\textnormal{near }\magneticFieldStrength=0
\]
on the respective universal covers and various intermediate covers.
The planar restricted $3$-body problem is touched in the final subsection.
It lies beyond the scope of the theory that will be developed in the
next sections, and hints at possible further development.

\newpage{}

\section{Ma\~n\'e's critical value of magnetic Hamiltonians \label{sec:MCV_of_magnetic_Hamiltonians}}

Let $\baseManifold$ be a connected closed manifold with smooth Riemannian
metric $\metric$ and associated norms on $T\baseManifold$ and $T^{*}\baseManifold$
denoted by $\norm v_{x}$ and $\norm p_{x}$ for $v\in T_{x}\baseManifold$
and $p\in T_{x}^{*}\baseManifold$, respectively. Any pair consisting
of a smooth function $\electricPotential\in\smoothFunctions{\baseManifold,\mathbb{R}}$
and a closed $2$-form $\magneticField\in\smoothTwoForms{\baseManifold,\mathbb{R}}$
can be viewed as an electromagnetic field acting on charged particles
whose motion is confined to $\baseManifold$. More precisely, let
$\pi\colon T^{*}\baseManifold\to\baseManifold$ be the canonical projection
and let $\omega_{0}=-\exteriorDifferential\lambda$ be the canonical
symplectic form of $T^{*}M$ with Liouville $1$-form $\lambda$.
The triple $(\metric,\magneticField,\electricPotential)$ gives rise
to the Hamiltonian system $(T^{*}M,\omega_{\magneticField},H_{\electricPotential})$
with Hamiltonian $H_{\electricPotential}(x,p)=\frac{1}{2}\norm p_{x}^{2}+\electricPotential(x)$
and twisted symplectic structure $\omega_{\magneticField}=\omega_{0}+\pi^{*}\magneticField$.
The metric $\metric$ induces the canonical bundle isomorphism $\flat\colon T\baseManifold\to T^{*}\baseManifold$,
which in turn gives rise to the dual system $(TM,\flat^{*}\omega_{\magneticField},\flat^{*}H_{\electricPotential})$.
The corresponding Hamiltonian flow on $T\baseManifold$ is called
the electromagnetic flow\emph{ }of $(\metric,\magneticField,\electricPotential)$
since its orbits coincide with the trajectories of a particle of unit
mass and charge under the influence of the conservative force $\nabla\electricPotential$
and the Lorentz force $F\colon T\baseManifold\to T\baseManifold$
defined by~\cite{BurnsPaternain2002}
\[
\magneticField_{x}(v,w)=\metric_{x}(F_{x}(v),w)\qquad\textnormal{for all }x\in\baseManifold\textnormal{ and }v,w\in T_{x}\baseManifold.
\]
The flow of $(g,\magneticField,0)$ is called magnetic flow or twisted
geodesic flow~\cite{BurnsPaternain2002,Paternain2009} since the
triple $(g,0,0)$ gives rise to the geodesic flow of $\metric$.

In the following, let $\cover{\pi}\colon\cover{\baseManifold}\to\baseManifold$
be a regular cover such that $\cover{\magneticField}=\cover{\pi}^{*}\magneticField$
is exact with magnetic potential $\cover{\magneticPotential}\in\smoothOneForms{\cover{\baseManifold},\mathbb{R}}$,
that is, $d\cover{\magneticPotential}=\cover{\magneticField}$. Whenever
there exists a magnetic potential, one can remove the twist in the
lift $\cover{\omega_{\magneticField}}$ of the symplectic structure
$\omega_{\magneticField}$ as follows. Let $\cover{\electricPotential}$,
$\cover{\metric}$, $\cover{\omega_{0}}$ and $H_{\cover{\electricPotential}}$
denote the lifts of $\electricPotential$, $\metric$, $\omega_{0}$
and $H_{\electricPotential}$, respectively. The flow of the system
$(T^{*}\cover M,\cover{\omega_{\magneticField}},H_{\cover{\electricPotential}})$
is equivalent to the flow of the system $(T^{*}\cover M,\cover{\omega_{0}},H_{\cover{\magneticPotential},\cover{\electricPotential}})$
with Hamiltonian
\[
H_{\cover{\magneticPotential},\cover{\electricPotential}}(x,p)=H_{\cover{\electricPotential}}(x,p+\cover{\magneticPotential})=\frac{1}{2}\left|p+\cover{\magneticPotential}\right|_{x}^{2}+\cover{\electricPotential}(x).
\]
An equivalence is given by the mapping $(x,p)\mapsto(x,p+\cover{\magneticPotential}_{x})$.
Ma\~n\'e's critical value of the Hamiltonian $H_{\cover{\magneticPotential},\cover{\electricPotential}}$
is defined as~\cite{CieliebakFrauenfelderPaternain2010}
\begin{equation}
\manesCriticalValue(H_{\cover{\magneticPotential},\cover{\electricPotential}})=\inf_{\cover{\magneticPotential}'\in\smoothOneForms{\cover{\baseManifold},\mathbb{R}}\colon d\cover{\magneticPotential}'=\cover{\magneticField}}\,\sup_{x\in\cover{\baseManifold}}H_{\cover{\electricPotential}}(x,\cover{\magneticPotential}_{x}')=\inf_{\cover{\omega}\in\smoothOneForms{\cover{\baseManifold},\mathbb{R}}\colon\exteriorDifferential\cover{\omega}=0}\,\sup_{x\in\cover{\baseManifold}}H_{\cover{\magneticPotential},\cover{\electricPotential}}(x,\cover{\omega}_{x}).\label{eq:MCV_of_Cover}
\end{equation}

The existence of magnetic potentials allows for a description of $\manesCriticalValue(H_{\cover{\magneticPotential},\cover{\electricPotential}})$
in terms of the Legendre transform $L_{\cover{\magneticPotential},\cover{\electricPotential}}\colon T\cover{\baseManifold}\to\mathbb{R}$,
that is, in terms of the Lagrangian given by 
\begin{equation}
L_{\cover{\magneticPotential},\cover{\electricPotential}}(x,v)=\frac{1}{2}\norm v_{x}^{2}-\cover{\magneticPotential}_{x}(v)-\cover{\electricPotential}(x).\label{eq:Lagrangian}
\end{equation}
The solutions of the Euler-Lagrange equations
\[
\frac{d}{dt}\frac{\partial L_{\cover{\magneticPotential},\cover{\electricPotential}}}{\partial v}(x,v)=\frac{\partial L_{\cover{\magneticPotential},\cover{\electricPotential}}}{\partial x}(x,v)
\]
give rise to the Euler-Lagrange flow, whose orbits are known to coincide
with the orbits of the electromagnetic flow of $(\cover{\metric},\cover{\magneticField},\cover{\electricPotential})$.
We let $A_{\cover{\magneticPotential},\cover{\electricPotential}}$
denote the action of $L_{\cover{\magneticPotential},\cover{\electricPotential}}$
on the space of absolutely continuous curves $\gamma\colon[a,b]\rightarrow\cover{\baseManifold}$
given as 
\[
A_{\cover{\magneticPotential},\cover{\electricPotential}}(\gamma)=\int_{a}^{b}L_{\cover{\magneticPotential},\cover{\electricPotential}}\left(\gamma(t),\dot{\gamma}(t)\right)dt.
\]
Ma\~n\'e~\cite{Mane1997} defined the critical value of the Lagrangian
$L_{\cover{\magneticPotential},\cover{\electricPotential}}$ as 
\[
\manesCriticalValue(L_{\cover{\magneticPotential},\cover{\electricPotential}})=\inf\multipliedBy\left\{ k\in\mathbb{R}\cup\{\infty\}\,\big|\, A_{\cover{\magneticPotential},\cover{\electricPotential}+k}(\gamma)\geq0\textrm{ for any closed curve }\gamma\right\} .
\]
Burns and Paternain~\cite{BurnsPaternain2002} gave the following
Hamiltonian description of $\manesCriticalValue(L_{\cover{\magneticPotential},\cover{\electricPotential}})$,
which is a generalization of~\cite[Theorem A]{ContrerasIturriagaPaternainPaternain1998}
\begin{eqnarray}
\manesCriticalValue(L_{\cover{\magneticPotential},\cover{\electricPotential}}) & = & \inf_{f\in\smoothFunctions{\cover{\baseManifold},\mathbb{R}}}\,\sup_{x\in\cover{\baseManifold}}H_{\cover{\magneticPotential},\cover{\electricPotential}}(x,d_{x}f)\label{eq:MCV_in_Ham_Terms}\\
 & = & \inf\left\{ k\in\mathbb{R}\cup\{\infty\}\,\big|\,\text{there exists }f\in\smoothFunctions{\cover{\baseManifold},\mathbb{R}}\colon H_{\cover{\magneticPotential},\cover{\electricPotential}}(df)<k\right\} .\nonumber 
\end{eqnarray}
In other words, $\manesCriticalValue(L_{\cover{\magneticPotential},\cover{\electricPotential}})$
is the infimum of values $k\in\mathbb{R}\cup\{\infty\}$ for which
$H_{\cover{\magneticPotential},\cover{\electricPotential}}^{-1}(-\infty,k)$
contains an exact Lagrangian graph. Note that replacing $\cover{\magneticPotential}$
by $\cover{\magneticPotential}+df$ for some $f\in\smoothFunctions{\cover{\baseManifold},\mathbb{R}}$
does not effect $\manesCriticalValue(L_{\cover{\magneticPotential},\cover{\electricPotential}})$.
Hence, another magnetic potential $\cover{\magneticPotential}'$ may
yield a different critical value only if $\cover{\magneticPotential}-\cover{\magneticPotential}'$
corresponds to a non-zero cohomology class in $\homologyGroup^{1}(\cover{\baseManifold},\mathbb{R})$.
A comparison of (\ref{eq:MCV_of_Cover}) and (\ref{eq:MCV_in_Ham_Terms})
leads to 
\begin{equation}
\manesCriticalValue(H_{\cover{\magneticPotential},\cover{\electricPotential}})=\inf_{\left[\cover{\omega}\right]\in\homologyGroup^{1}(\cover M,\mathbb{R})}\manesCriticalValue(L_{\cover{\magneticPotential}-\cover{\omega},\cover{\electricPotential}}).\label{eq:MCV_in_Lag_Terms}
\end{equation}
If $L_{\coverTwo{\magneticPotential},\coverTwo{\electricPotential}}$
is the lift of $L_{\cover{\magneticPotential},\cover{\electricPotential}}$
to a regular cover $\coverTwo{\baseManifold}$ of $\cover{\baseManifold}$,
then 
\begin{equation}
\manesCriticalValue(L_{\coverTwo{\magneticPotential},\coverTwo{\electricPotential}})\leq\manesCriticalValue(L_{\cover{\magneticPotential},\cover{\electricPotential}})\qquad\textnormal{and}\qquad\manesCriticalValue(H_{\coverTwo{\magneticPotential},\coverTwo{\electricPotential}})\leq\manesCriticalValue(H_{\cover{\magneticPotential},\cover{\electricPotential}}),\label{eq:MCV_of_Lift_smaller}
\end{equation}
with equality if $\coverTwo{\baseManifold}$ is a finite cover of
$\cover{\baseManifold}$. In the so-called exact case in which $\magneticField=\exteriorDifferential\magneticPotential$
for some $\magneticPotential\in\smoothOneForms{\baseManifold,\mathbb{R}}$,
we have a Hamiltonian $H_{\magneticPotential,\electricPotential}$
and a Lagrangian $L_{\magneticPotential,\electricPotential}$ on $\baseManifold$.
Ma\~n\'e~\cite{Mane1997} coined the phrase strict critical value
for $\manesCriticalValue(H_{\magneticPotential,\electricPotential})$,
and denoted it by $\manesCriticalValue_{0}(L_{\magneticPotential,\electricPotential})$.
He related $\manesCriticalValue_{0}(L_{\magneticPotential,\electricPotential})$
to Mather's action functional$\alpha\colon\homologyGroup^{1}(\baseManifold,\mathbb{R})\to\mathbb{R}$
given by~\cite{Mather1991} 
\[
\alpha\left([\omega]\right)=-\min\left\{ \int L_{\magneticPotential+\omega,\electricPotential}d\mu\,\Big|\,\mu\in\mathcal{M}(L_{\magneticPotential,\electricPotential})\right\} ,
\]
where $\mathcal{M}(L_{\magneticPotential,\electricPotential})$ denotes
the set of probabilities on the Borel $\sigma$-algebra of $T\baseManifold$
that have compact support and are invariant under the Euler-Lagrange
flow. Ma\~n\'e~\cite{Mane1997} proved that 
\[
\alpha\left([\omega]\right)=\manesCriticalValue(L_{\magneticPotential+\omega,\electricPotential}).
\]
Since $\alpha$ is convex and superlinear~\cite[Theorem 1]{Mather1991},
one obtains
\[
\manesCriticalValue_{0}(L_{\magneticPotential,\electricPotential})=\manesCriticalValue(H_{\magneticPotential,\electricPotential})=\min_{\left[\omega\right]\in\homologyGroup^{1}(M,\mathbb{R})}\alpha\left([\omega]\right).
\]

We let $\universalCover{\pi}{}\colon\universalCover{\baseManifold}{}\to M$
and $\abelianCover{\pi}{}\colon\abelianCover{\baseManifold}{}\to\baseManifold$
denote the universal and the maximal abelian cover of $M$, respectively.
Their covering transformation groups are the fundamental group $\pi_{1}(\baseManifold)$
and the first homology group $\homologyGroup_{1}(M,\mathbb{Z})$.
More precisely, $\abelianCover{\baseManifold}{}$ is defined as the
regular cover of $\baseManifold$ whose fundamental group is the kernel
of the Hurewicz homomorphism $\pi_{1}(\baseManifold)\to\homologyGroup_{1}(\baseManifold,\mathbb{Z})$,
that is, the commutator subgroup $[\pi_{1}(\baseManifold),\pi_{1}(\baseManifold)]$.
Therefore, any regular cover with an abelian covering transformation
group is covered by $\abelianCover{\baseManifold}{}$.

In the monopole case in which $[\magneticField]\neq0\in\homologyGroup^{2}(\baseManifold,\mathbb{R})$,
one can define $\universalCoverWithoutHat L{\cover{\magneticPotential},\cover{\electricPotential}}\colon T\universalCover{\baseManifold}{}\to\mathbb{R}$,
respectively $\abelianCoverWithoutHat L{\cover{\magneticPotential},\cover{\electricPotential}}\colon T\abelianCover{\baseManifold}{}\to\mathbb{R}$,
as in~(\ref{eq:Lagrangian}) only if the lift of $\magneticField$
to $\universalCover{\baseManifold}{}$, respectively to $\abelianCover{\baseManifold}{}$,
is exact. If this is the case, we obtain $\manesCriticalValue(\universalCoverWithoutHat H{\cover{\magneticPotential},\cover{\electricPotential}})=\manesCriticalValue(\universalCoverWithoutHat L{\cover{\magneticPotential},\cover{\electricPotential}})$
directly from~(\ref{eq:MCV_in_Lag_Terms}) since $\homologyGroup^{1}(\universalCover{\baseManifold}{},\mathbb{R})$
is trivial. In the exact case, it is known that $\manesCriticalValue(\abelianCoverWithoutHat L{\cover{\magneticPotential},\cover{\electricPotential}})=\manesCriticalValue_{0}(L_{\magneticPotential,\electricPotential})$~\cite{PaternainPaternain1997}.
Recently, Fathi and Maderna~\cite{FathiMaderna2007} proved that
if, in addition, $\pi_{1}(\baseManifold)$ is amenable, then 
\begin{equation}
\manesCriticalValue(\universalCoverWithoutHat L{\cover{\magneticPotential},\cover{\electricPotential}})=\manesCriticalValue(\abelianCoverWithoutHat L{\cover{\magneticPotential},\cover{\electricPotential}})=\manesCriticalValue_{0}(L_{\magneticPotential,\electricPotential}),\label{eq:MCV_of_Univ_Abel_Cover}
\end{equation}
which also implies $\manesCriticalValue(\abelianCoverWithoutHat H{\cover{\magneticPotential},\cover{\electricPotential}})=\manesCriticalValue(\abelianCoverWithoutHat L{\cover{\magneticPotential},\cover{\electricPotential}})$.
For the reader's convenience, we recall the notion of amenability.
\begin{defn}
A discrete group $\coveringTransformationGroup$ is called amenable,
if there exists a continuous functional $m$ on the space $L^{\infty}(\coveringTransformationGroup,\mathbb{R})$
of bounded real-valued functions on $\coveringTransformationGroup$
such that\end{defn}
\begin{enumerate}
\item $m(1_{\coveringTransformationGroup})=1,$ where $1_{\coveringTransformationGroup}(\groupElement)=1$
for all $\groupElement\in\coveringTransformationGroup$,
\item if $f\geq0,$ then $m(f)\geq0$, and
\item $m(\groupElement\multipliedBy f)=m(f)$ for each $\groupElement\in\coveringTransformationGroup$
and $f\in L^{\infty}(\coveringTransformationGroup,\mathbb{R})$, where
$(\groupElement\multipliedBy f)(\groupElement')=f(\groupElement^{-1}\groupElement')$.
\end{enumerate}
For instance, groups with subexponential growth and finite extensions
of solvable groups are amenable whereas groups containing a free subgroup
on two generators are not amenable. The proof of~(\ref{eq:MCV_of_Univ_Abel_Cover})
given in~\cite{FathiMaderna2007} is based on an equivariant version
of the weak KAM theorem and can be extended as follows. Let $\coverTwo{\baseManifold}$
be a regular cover of $\baseManifold$ with amenable covering transformation
group $\coveringTransformationGroup$. Let $\cover{\baseManifold}$
denote a subcover whose covering transformation group $\cover{\coveringTransformationGroup}$
is isomorphic to the abelianization $\coveringTransformationGroup/[\coveringTransformationGroup,\coveringTransformationGroup]$.
Note that $\cover{\baseManifold}$ is covered by $\abelianCover{\baseManifold}{}$,
which entails a surjective homomorphism $\Phi\colon\homologyGroup_{1}(\baseManifold,\mathbb{Z})\twoheadrightarrow\cover{\coveringTransformationGroup}$,
whose kernel is the group of covering transformations of $\abelianCover{\baseManifold}{}$
with trivial projections to $\cover{\baseManifold}$. The transpose
of $\Phi$ can be extended to an injective linear map from $\homomorphismGroup(\cover{\coveringTransformationGroup},\mathbb{Z})\otimes\mathbb{R}\simeq\homomorphismGroup(\cover{\coveringTransformationGroup},\mathbb{R})\simeq\homomorphismGroup(\coveringTransformationGroup,\mathbb{R})$
into $\homomorphismGroup(\homologyGroup_{1}(\baseManifold,\mathbb{Z}),\mathbb{Z})\otimes\mathbb{R}\simeq\homologyGroup^{1}(\baseManifold,\mathbb{Z})\otimes\mathbb{R}\simeq\homologyGroup^{1}(\baseManifold,\mathbb{R})$,
whereby $\homomorphismGroup(\coveringTransformationGroup,\mathbb{R})$
can be identified with a subspace of $\homologyGroup^{1}(\baseManifold,\mathbb{R})$,
namely,
\begin{equation}
\left\{ [\omega]\in\homologyGroup^{1}(\baseManifold,\mathbb{R})\,\Big|\,\int_{\gamma}\cover{\omega}=0\textrm{ for any closed curve }\gamma\textrm{ in }\cover{\baseManifold}\right\} ,\label{eq:Hom_CTG_to_R_as_One_Forms}
\end{equation}
where $\cover{\omega}$ denotes the lift of a representative of $[\omega]$.
We provide an explicit isomorphism below. A similar argument appears
in~\cite{KotaniSunada2000}, which deals with the long-time asymptotics
of the heat kernel. Note that any $\cover{\omega}\in\smoothOneForms{\cover{\baseManifold},\mathbb{R}}$
that satisfies (\ref{eq:Hom_CTG_to_R_as_One_Forms}) is exact with
primitives of the form $f(x)=f(x_{0})+\int_{x_{0}}^{x}\cover{\omega}$,
where $x_{0}\in\cover{\baseManifold}$ and $f(x_{0})\in\mathbb{R}$
can be chosen arbitrarily. Hence, $[\omega]\in\homomorphismGroup(\coveringTransformationGroup,\mathbb{R})$
if and only if the lift of any representative of $[\omega]$ to $\cover{\baseManifold}$
is exact. The following theorem generalizes the aforementioned results
in~\cite{PaternainPaternain1997,FathiMaderna2007}.
\begin{thm}
\label{thm:MCV_of_Amenable_Covers_and_Abelian_Subcovers_coincide}
Let $\baseManifold$ be a connected closed Riemannian manifold. Let
$\coverTwo{\baseManifold}$ be a regular cover with amenable covering
transformation group $\coveringTransformationGroup$, and let $\cover{\baseManifold}$
be a subcover whose covering transformation group $\cover{\coveringTransformationGroup}$
is isomorphic to $\coveringTransformationGroup/[\coveringTransformationGroup,\coveringTransformationGroup]$.
Then, for any $\magneticPotential\in\smoothOneForms{\baseManifold,\mathbb{R}}$
and $\electricPotential\in\smoothFunctions{\baseManifold,\mathbb{R}}$,
the Lagrangian $L_{\magneticPotential,\electricPotential}\colon T\baseManifold\to\mathbb{R}$
given by
\[
L_{\magneticPotential,\electricPotential}(x,v)=\frac{1}{2}\norm v_{x}^{2}-\magneticPotential_{x}(v)-\electricPotential(x)
\]
and its lifts $L_{\coverTwo{\magneticPotential},\coverTwo{\electricPotential}}$
and $L_{\cover{\magneticPotential},\cover{\electricPotential}}$ to
$\coverTwo{\baseManifold}$ and $\cover{\baseManifold}$ satisfy
\[
\manesCriticalValue(L_{\coverTwo{\magneticPotential},\coverTwo{\electricPotential}})=\manesCriticalValue(L_{\cover{\magneticPotential},\cover{\electricPotential}})=\min_{[\omega]\in\homomorphismGroup(\coveringTransformationGroup,\mathbb{R})}\manesCriticalValue(L_{\magneticPotential-\omega,\electricPotential})=\min_{\omega\in\smoothOneForms{\baseManifold,\mathbb{R}}\colon\cover{\omega}\textrm{ is exact}}\manesCriticalValue(L_{\magneticPotential-\omega,\electricPotential}),
\]
where $\homomorphismGroup(\coveringTransformationGroup,\mathbb{R})$
denotes the vector space\textup{~(\ref{eq:Hom_CTG_to_R_as_One_Forms})},
and $\cover{\omega}$ denotes the lift of $\omega$ to $\cover{\baseManifold}$\textup{.}
\end{thm}
Note that forms with exact lifts are necessarily closed. For abelian
covers $\coverTwo{\baseManifold}=\cover{\baseManifold}$, one obtains
the following generalization of $\manesCriticalValue(\abelianCoverWithoutHat L{\cover{\magneticPotential},\cover{\electricPotential}})=\manesCriticalValue_{0}(L_{\magneticPotential,\electricPotential})$.
\begin{cor}
\label{cor:MCV_of_Abelian_covers}If $\cover{\baseManifold}$ is a
regular cover with abelian covering transformation group, then we
have
\[
\manesCriticalValue(L_{\cover{\magneticPotential},\cover{\electricPotential}})=\min_{\omega\in\smoothOneForms{\baseManifold,\mathbb{R}}\colon\cover{\omega}\textrm{ is exact}}\manesCriticalValue(L_{\magneticPotential-\omega,\electricPotential}).
\]
\end{cor}
\begin{proof}
[Proof of Theorem \ref{thm:MCV_of_Amenable_Covers_and_Abelian_Subcovers_coincide}]For
any $\omega\in\smoothOneForms{\baseManifold,\mathbb{R}}$ with exact
lift $\cover{\omega}\in\smoothOneForms{\cover{\baseManifold},\mathbb{R}}$,
(\ref{eq:MCV_in_Ham_Terms}) and (\ref{eq:MCV_of_Lift_smaller}) imply
that 
\[
\manesCriticalValue(L_{\coverTwo{\magneticPotential},\coverTwo{\electricPotential}})\leq\manesCriticalValue(L_{\cover{\magneticPotential},\cover{\electricPotential}})=\manesCriticalValue(L_{\cover{\magneticPotential}-\cover{\omega},\cover{\electricPotential}})\leq\manesCriticalValue(L_{\magneticPotential-\omega,\electricPotential}),
\]
see also~\cite[Lemma 2.4]{PaternainPaternain1997}. Hence, Theorem~\ref{thm:MCV_of_Amenable_Covers_and_Abelian_Subcovers_coincide}
is proven once we find $\omega\in\smoothOneForms{\baseManifold,\mathbb{R}}$
with $[\omega]\in\homomorphismGroup(\coveringTransformationGroup,\mathbb{R})$
such that $\manesCriticalValue(L_{\magneticPotential-\omega,\electricPotential})=\manesCriticalValue(L_{\coverTwo{\magneticPotential},\coverTwo{\electricPotential}})$.
This is established along the lines of~\cite[Theorem 1.5]{FathiMaderna2007},
more precisely, $\coverTwo{\baseManifold}$, $\cover{\baseManifold}$,
$\coveringTransformationGroup$ and $\cover{\coveringTransformationGroup}$
assume the roles of $\universalCover{\baseManifold}{}$, $\abelianCover{\baseManifold}{}$,
$\pi_{1}(\baseManifold)$ and $\homologyGroup_{1}(\baseManifold,\mathbb{Z})=\pi_{1}(\baseManifold)/[\pi_{1}(\baseManifold),\pi_{1}(\baseManifold)]$,
respectively. For any $\omega$ with $[\omega]\in\homomorphismGroup(\coveringTransformationGroup,\mathbb{R})$,
we can choose $f_{\omega}\in\smoothFunctions{\coverTwo{\baseManifold},\mathbb{R}}$
such that the lift of $\omega$ to $\coverTwo{\baseManifold}$ takes
the form $\coverTwo{\omega}=df_{\omega}$. Since $\coverTwo{\baseManifold}$
is connected, $f_{\omega}$ is determined up to a constant. For any
$\groupElement\in\coveringTransformationGroup$, we consider the function
$\rho_{\omega,\groupElement}=\groupElement^{*}f_{\omega}-f_{\omega}$.
Note that the definition of $\rho_{\omega,\groupElement}$ is independent
of the choice of $f_{\omega}$. Since $\groupElement^{*}\coverTwo{\omega}=\coverTwo{\omega}$
for any $\groupElement\in\coveringTransformationGroup$, the function
$\rho_{\omega,\groupElement}$ has vanishing derivative and is therefore
a constant that we denote by $\rho_{\omega}(\groupElement)$. One
easily verifies that the mapping $\groupElement\mapsto\rho_{\omega}(\groupElement)$
is an element of $\homomorphismGroup(\coveringTransformationGroup,\mathbb{R})$.
Moreover, the mapping $\omega\mapsto\rho_{\omega}$ is linear and
injective, and thus establishes the desired isomorphism. In~\cite[Section 7]{FathiMaderna2007},
it is shown that any $\rho\in\homomorphismGroup(\coveringTransformationGroup,\mathbb{R})$
gives rise to a critical value $\manesCriticalValue(\rho)$ such that
any $\rho_{\omega}$ as above satisfies $\manesCriticalValue(\rho_{\omega})=\manesCriticalValue(L_{\magneticPotential+\omega,\electricPotential})$.
Since $\coveringTransformationGroup$ is amenable, there exists $\rho\in\homomorphismGroup(\coveringTransformationGroup,\mathbb{R})$
such that $\manesCriticalValue(\rho)=\manesCriticalValue(L_{\coverTwo{\magneticPotential},\coverTwo{\electricPotential}})$
by virtue of~\cite[Lemma 7.3 and Theorem 7.4]{FathiMaderna2007}.
We already saw that $\rho=\rho_{\omega}$ for some $\omega$ with
$[\omega]\in\homomorphismGroup(\coveringTransformationGroup,\mathbb{R})$,
which completes the proof.
\end{proof}
Note that the functions $f_{\omega}\in\smoothFunctions{\coverTwo{\baseManifold},\mathbb{R}}$
in the proof of Theorem~\ref{thm:MCV_of_Amenable_Covers_and_Abelian_Subcovers_coincide}
are lifts of smooth functions on the subcover $\cover{\baseManifold}$.
Algebraically, this is reflected in $\homomorphismGroup(\coveringTransformationGroup,\mathbb{R})\simeq\homomorphismGroup(\cover{\coveringTransformationGroup},\mathbb{R})$.
The existence of $\rho_{\omega}\in\homomorphismGroup(\coveringTransformationGroup,\mathbb{R})$
with minimal associated critical value $\manesCriticalValue(\rho_{\omega})$
is a consequence of the convexity and superlinearity of the Mather
function~\cite[Proposition 7.2]{FathiMaderna2007} which is the mapping
$\homomorphismGroup(\coveringTransformationGroup,\mathbb{R})\ni\rho_{\omega}\mapsto\manesCriticalValue(\rho_{\omega})=\manesCriticalValue(L_{\magneticPotential+\omega,\electricPotential})$.
One easily extends Theorem~\ref{thm:MCV_of_Amenable_Covers_and_Abelian_Subcovers_coincide}
to any Lagrangian $L\colon T\baseManifold\to\mathbb{R}$ of class
$C^{2}$ that satisfies the following conditions:
\begin{enumerate}
\item Convexity: For every $x\in\baseManifold$, the restriction of $L$
to $T_{x}\baseManifold$ has positive definite Hessian everywhere.
\item Uniform\emph{ }superlinearity: For every $K\geq0$, there exists $C(K)\in\mathbb{R}$
such that 
\[
L(x,v)\geq K\multipliedBy\norm v_{x}-C(K)\qquad\textnormal{for all }(x,v)\in TM.
\]

\end{enumerate}
Corollary~\ref{cor:MCV_of_Abelian_covers} also allows for a direct
proof along the lines of~\cite[Theorem 1.1]{PaternainPaternain1997}.
More precisely, if $\coveringTransformationGroup$ denotes the abelian
covering transformation group of $\cover{\baseManifold}$, then $\cover{\baseManifold}$,
$\coveringTransformationGroup$, $\coveringTransformationGroup\otimes\mathbb{R}$
and $\homomorphismGroup(\coveringTransformationGroup,\mathbb{R})$
assume the roles of $\abelianCover{\baseManifold}{}$, $\homologyGroup_{1}(\baseManifold,\mathbb{Z})$,
$\homologyGroup_{1}(\baseManifold,\mathbb{R})$ and $\homologyGroup^{1}(\baseManifold,\mathbb{R})$
in~\cite{PaternainPaternain1997}, respectively. In particular, the
curves $\overline{x_{i}}$ appearing in the proof of~\cite[Theorem 1.1]{PaternainPaternain1997}
can be chosen as lifts of curves in $\cover{\baseManifold}$, for
which reason the slopes appearing in~\cite[Theorem 2.3]{PaternainPaternain1997}
can be taken from $\homomorphismGroup(\coveringTransformationGroup,\mathbb{R})$.
The necessary adaptions are routine but lengthly for which reason
we skip them.

\newpage{}

\section{Magnetic Schr\"odinger operators\label{sec:MSO}}

As in Section~\ref{sec:MCV_of_magnetic_Hamiltonians}, let $\baseManifold$
be a connected closed manifold with Riemannian metric $\metric$,
and let $\cover{\baseManifold}$ be a regular cover of $\baseManifold$
equipped with potentials $\cover{\magneticPotential}\in\smoothOneForms{\cover{\baseManifold},\mathbb{R}}$
and $\cover{\electricPotential}\in\smoothFunctions{\cover{\baseManifold},\mathbb{R}}$.
Recall that the system $(T^{*}\cover M,\cover{\omega_{0}},H_{\cover{\magneticPotential},\cover{\electricPotential}})$
with standard symplectic structure $\cover{\omega_{0}}$ on $T^{*}\cover{\baseManifold}$
and Hamiltonian $H_{\cover{\magneticPotential},\cover{\electricPotential}}\colon T^{*}\cover{\baseManifold}\to\mathbb{R}$
given by 
\begin{equation}
H_{\cover{\magneticPotential},\cover{\electricPotential}}(x,p)=\frac{1}{2}\left|p+\cover{\magneticPotential}\right|_{x}^{2}+\cover{\electricPotential}(x)\label{eq:Classical_Hamiltonian}
\end{equation}
describes the classical motion of a charged particle on $\cover{\baseManifold}$
under the influence of the electromagnetic field with magnetic and
electric potentials $\cover{\magneticPotential}$ and $\cover{\electricPotential}$,
respectively. In contrast, non-relativistic quantum mechanics is essentially
the study of self-adjoint, densely-defined operators on Hilbert spaces.
In our case, these are minimal Schr\"odinger operators, that is,
closures of differential operators of Schr\"odinger type that are
initially defined on the space $\smoothCompactlySupportedFunctions{\cover{\baseManifold}}=\smoothCompactlySupportedFunctions{\cover{\baseManifold},\mathbb{C}}$
of compactly-supported, smooth, $\mathbb{C}$-valued functions on
$\cover{\baseManifold}$. In comparison to Laplacians, the study of
magnetic potentials requires to consider complex-valued functions.
Let $L^{2}(\cover{\baseManifold})=L^{2}(\cover{\baseManifold},\volumeForm,\mathbb{C})$
denote the completion of $\smoothCompactlySupportedFunctions{\cover{\baseManifold}}$
with respect to the norm $\normComingFromInnerProduct u=\sqrt{\innerProduct uu}$
coming from the inner product 
\[
\innerProduct uv=\int_{\cover{\baseManifold}}u\multipliedBy\conjugate v\multipliedBy\volumeForm,
\]
where $u,v\in\smoothCompactlySupportedFunctions{\cover{\baseManifold}}$
and $\volumeForm$ denotes the volume form of $\metric$. Following~\cite{Paternain2001},
we use the Dirac quantization rule which says that in order to quantize
(\ref{eq:Classical_Hamiltonian}), we have to replace $p$ by the
operator $\frac{1}{i}\multipliedBy\exteriorDifferential$, where $\exteriorDifferential$
is the exterior differential. Let $\smoothCompactlySupportedOneForms{\cover{\baseManifold}}=\smoothCompactlySupportedOneForms{\cover{\baseManifold},\mathbb{C}}$
denote the space of compactly-supported, smooth, $\mathbb{C}$-valued
$1$-forms on $\cover{\baseManifold}$. This space is equipped with
an inner product given by integration over the fibrewise inner products
on $T^{*}\cover{\baseManifold}$. We denote the completion of $\smoothCompactlySupportedOneForms{\cover{\baseManifold}}$
by $L^{2}(\smoothOneForms{\cover{\baseManifold}})$. The magnetic
differential\emph{ }$\magneticDifferential{\cover{\magneticPotential}}\colon\smoothCompactlySupportedFunctions{\cover{\baseManifold}}\to\smoothCompactlySupportedOneForms{\cover{\baseManifold}}$
is the operator given by
\[
\magneticDifferential{\cover{\magneticPotential}}u=\frac{1}{i}\multipliedBy\exteriorDifferential u+u\multipliedBy\cover{\magneticPotential}.
\]
The associated magnetic Schr\"odinger operator $\schroedingerOperator_{\cover{\magneticPotential},\cover{\electricPotential}}$
with domain $\smoothCompactlySupportedFunctions{\cover{\baseManifold}\targetC}$
is defined as 
\begin{equation}
\schroedingerOperator_{\cover{\magneticPotential},\cover{\electricPotential}}=\frac{1}{2}\multipliedBy\adjoint{\magneticDifferential{\cover{\magneticPotential}}}\multipliedBy\magneticDifferential{\cover{\magneticPotential}}+\cover{\electricPotential},\label{eq:Def_of_MSO}
\end{equation}
where $\adjoint{\magneticDifferential{\cover{\magneticPotential}}}\colon\smoothCompactlySupportedOneForms{\cover{\baseManifold}\targetC}\to\smoothCompactlySupportedFunctions{\cover{\baseManifold}\targetC}$
denotes the formal adjoint of $\magneticDifferential{\cover{\magneticPotential}}$.
Recall that $\adjoint{\magneticDifferential{\cover{\magneticPotential}}}$
is the unique differential operator such that $\innerProduct u{\adjoint{\magneticDifferential{\cover{\magneticPotential}}}\omega}=\innerProduct{\magneticDifferential{\cover{\magneticPotential}}u}{\omega}$
holds for any $u\in\smoothCompactlySupportedFunctions{\cover{\baseManifold}}$
and $\omega\in\smoothCompactlySupportedOneForms{\cover{\baseManifold}}$.
As a differential operator, $\schroedingerOperator_{\cover{\magneticPotential},\cover{\electricPotential}}$
can be expressed in local terms.
\begin{lem}
The magnetic Schr\"odinger operator (\ref{eq:Def_of_MSO}) with domain
$\smoothCompactlySupportedFunctions{\cover{\baseManifold}}$ is given
by
\begin{equation}
\schroedingerOperator_{\cover{\magneticPotential},\cover{\electricPotential}}u=\frac{1}{2}\Delta u-i\innerProduct{\exteriorDifferential u}{\cover{\magneticPotential}}+\left(\frac{1}{2}\multipliedBy\exteriorDifferential^{*}\cover{\magneticPotential}+\frac{1}{2}\norm{\cover{\magneticPotential}}^{2}+\cover{\electricPotential}\right)u,\label{eq:Coordinatefree_Version_of_MSO}
\end{equation}
where $\adjoint{\exteriorDifferential}$ denotes the codifferential
and $\Delta=\adjoint{\exteriorDifferential}\multipliedBy\exteriorDifferential$
is the Laplace-Beltrami operator on $\cover{\baseManifold}$. With
respect to local coordinates $(x_{1},x_{2,}\ldots,x_{n})$, we have
\begin{equation}
\schroedingerOperator_{\cover{\magneticPotential},\cover{\electricPotential}}=\frac{1}{2}\frac{1}{\sqrt{\norm{\metric}}}\multipliedBy\sum_{j,k}\left(\frac{1}{i}\multipliedBy\frac{\partial}{\partial x_{j}}+\conjugate{\cover{\magneticPotential}_{j}}\right)\multipliedBy\metric^{jk}\multipliedBy\sqrt{\norm{\metric}}\multipliedBy\left(\frac{1}{i}\frac{\partial}{\partial x_{k}}+\cover{\magneticPotential}_{k}\right)+\cover{\electricPotential},\label{eq:Coordinate_Version_of_MSO}
\end{equation}
where $\cover{\magneticPotential}(x_{1},x_{2},\ldots,x_{n})=\sum_{k}\cover{\magneticPotential}_{k}(x_{1},x_{2},\ldots,x_{n})\multipliedBy\exteriorDifferential x^{k}$,
$\volumeForm=\sqrt{\norm{\metric}}\multipliedBy\exteriorDifferential x^{1}\ldots\exteriorDifferential x^{n}$
and $\metric^{jk}=\innerProduct{\exteriorDifferential x^{j}}{\exteriorDifferential x^{k}}$
are the entries of the inverse of the local matrix expression of $\metric$.\end{lem}
\begin{proof}
For compact $\cover{\baseManifold}$, \cite{Paternain2001} contains
a proof of (\ref{eq:Coordinatefree_Version_of_MSO}) that can be easily
generalized to our setting. Instead, we verify the coordinate version~(\ref{eq:Coordinate_Version_of_MSO}).
For any $u\in\smoothCompactlySupportedFunctions{\cover{\baseManifold}}$
with support in the given coordinate neighborhood and any $\omega\in\smoothCompactlySupportedOneForms{\cover{\baseManifold}}$
with local expression $\omega=\sum_{k}\omega_{k}\multipliedBy\exteriorDifferential x^{k}$,
partial integration yields
\begin{eqnarray*}
\innerProduct u{\adjoint{\magneticDifferential{\cover{\magneticPotential}}}\omega}=\innerProduct{-i\multipliedBy\exteriorDifferential u+u\multipliedBy\cover{\magneticPotential}}{\omega} & = & \innerProductThatAdapts{\sum_{j}\left(\frac{1}{i}\frac{\partial u}{\partial x_{j}}+u\multipliedBy\cover{\magneticPotential}_{j}\right)\exteriorDifferential x^{j}}{\sum_{k}\omega_{k}\multipliedBy\exteriorDifferential x^{k}}\\
 & = & \int\sum_{j,k}\left(\left(\frac{1}{i}\frac{\partial u}{\partial x_{j}}+u\multipliedBy\cover{\magneticPotential}_{j}\right)\conjugate{\omega_{k}}\multipliedBy\metric^{jk}\multipliedBy\sqrt{\norm{\metric}}\right)\multipliedBy\exteriorDifferential x^{1}\ldots\exteriorDifferential x^{n}\\
 & = & \int u\multipliedBy\sum_{j,k}\left(\left(-\frac{1}{i}\frac{\partial}{\partial x_{j}}+\cover{\magneticPotential}_{j}\right)\metric^{jk}\multipliedBy\sqrt{\norm{\metric}}\multipliedBy\conjugate{\omega_{k}}\right)\multipliedBy\exteriorDifferential x^{1}\ldots\exteriorDifferential x^{n}\\
 & = & \innerProductThatAdapts u{\frac{1}{\sqrt{\norm{\metric}}}\multipliedBy\sum_{j,k}\left(\frac{1}{i}\frac{\partial}{\partial x_{j}}+\conjugate{\cover{\magneticPotential}_{j}}\right)\metric^{jk}\multipliedBy\sqrt{\norm{\metric}}\multipliedBy\omega_{k}}.
\end{eqnarray*}
If $\omega=\magneticDifferential{\cover{\magneticPotential}}v$ for
some $v\in\smoothCompactlySupportedFunctions{\cover{\baseManifold}}$,
then $\omega_{k}=\left(\frac{1}{i}\frac{\partial}{\partial x_{k}}+\cover{\magneticPotential}_{k}\right)v$,
which completes the proof. \end{proof}
\begin{thm}
If $\cover{\electricPotential}$ is semi-bounded from below, meaning
$\inf\,\cover{\electricPotential}>-\infty$, then $\schroedingerOperator_{\cover{\magneticPotential},\cover{\electricPotential}}$
with domain $\smoothCompactlySupportedFunctions{\cover{\baseManifold}}$
is essentially self-adjoint, that is, its closure $\closure{\schroedingerOperator_{\cover{\magneticPotential},\cover{\electricPotential}}}$
is self-adjoint. \label{thm:MSO_are_Essentially_Selfadjoint}\end{thm}
\begin{proof}
As a cover of a closed manifold, $\cover{\baseManifold}$ is complete.
Moreover, $\schroedingerOperator_{\cover{\magneticPotential},\cover{\electricPotential}}$
is semi-bounded from below on $\smoothCompactlySupportedFunctions{\cover{\baseManifold}}$
since for any $u\in\smoothCompactlySupportedFunctions{\cover{\baseManifold}}$,
we have
\begin{equation}
\innerProduct{\schroedingerOperator_{\cover{\magneticPotential},\cover{\electricPotential}}u}u=\frac{1}{2}\innerProduct{\magneticDifferential{\cover{\magneticPotential}}u}{\magneticDifferential{\cover{\magneticPotential}}u}+\innerProduct{u\multipliedBy\cover{\electricPotential}}u\geq\inf\,\cover{\electricPotential}\multipliedBy\innerProduct uu.\label{eq:MSO_semibounded_from_below}
\end{equation}
The claim now follows from~\cite[Theorem 1.1]{Shubin2001} or~\cite[Theorem 2.13]{BravermanMilatovichShubin2002}.
\end{proof}
Ikeba and Kato~\cite{IkebeKato1962} were the first to prove essential
self-adjointness for a wide class of singular magnetic potentials
on $\mathbb{R}^{n}$. Kato~\cite{Kato1972} extended these results
using his famous inequality, which we discuss in Section~\ref{sub:GSE}.
Hess~et~al.~\cite{HessSchraderUhlenbrock1977} and Simon~\cite{Simon1979}
later revealed the functional analytic nature of Kato's inequality
in terms of domination of semigroups. Most proofs of essential self-adjointness
of magnetic Schr\"odinger operators use some sort of Kato inequality,
see also~\cite{LeinfelderSimader1981,Iwatsuka1990,BravermanMilatovichShubin2002,Milatovic2003,Milatovic2004,RozenblumMelgaard2005}.

The study of $\schroedingerOperator_{\cover{\magneticPotential},0}$
dates back to the 1930s when Landau considered the case $\cover{\baseManifold}=\mathbb{R}^{2}$;
however, major progress had not been made till the 1970s \cite{AvronHerbstSimon1978,Iwatsuka1986,Tamura1987,Tamura1988,Tamura1989}.
In this context, we point out the article~\cite{KondratievShubin2002},
in which Kondratiev and Shubin derive necessary and sufficient conditions
for magnetic Schr\"odinger operators to have discrete spectrum. In
particular, they recover a theorem of Avron~et~al.~\cite{AvronHerbstSimon1978}
which says that whenever $\closure{\schroedingerOperator_{0,0}}=\closure{\Delta}$
has discrete spectrum, the same is true for any $\closure{\schroedingerOperator_{\cover{\magneticPotential},0}}$
with $\cover{\magneticPotential}\in\smoothOneForms{\cover{\baseManifold},\mathbb{R}}$.

By virtue of~\cite[Theorem 4.3.1]{Davies1995}, the estimate (\ref{eq:MSO_semibounded_from_below})
implies $\spectrum{\closure{\schroedingerOperator_{\cover{\magneticPotential},\cover{\electricPotential}}}}\subseteq[\inf\,\cover{\electricPotential},\infty)$.
Moreover, $\schroedingerOperator_{\cover{\magneticPotential},\cover{\electricPotential}}$
and $\schroedingerOperator_{-\cover{\magneticPotential},\cover{\electricPotential}}$
are conjugate to each other with respect to complex conjugation on
$\smoothCompactlySupportedFunctions{\cover{\baseManifold}}$ which
yields $\spectrum{\closure{\schroedingerOperator_{\cover{\magneticPotential},\cover{\electricPotential}}}}=\spectrum{\closure{\schroedingerOperator_{-\cover{\magneticPotential},\cover{\electricPotential}}}}$.
If $\cover{\baseManifold}$ is compact, then $\spectrum{\closure{\schroedingerOperator_{\cover{\magneticPotential},\cover{\electricPotential}}}}$
is discrete and $L^{2}(\cover{\baseManifold})$ is the direct sum
of countably many finite-dimensional eigenspaces of $\closure{\schroedingerOperator_{\cover{\magneticPotential},\cover{\electricPotential}}}$
consisting of smooth eigenfunctions as is well-known, see~\cite[Theorem 2.1]{Shigekawa1987}
for instance. If, in addition, $\coverTwo{\baseManifold}$ is a finite
regular cover of $\cover{\baseManifold}$ with lifted potentials $\coverTwo{\magneticPotential}$
and $\coverTwo{\electricPotential}$, then we obtain $\spectrum{\closure{\schroedingerOperator_{\cover{\magneticPotential},\cover{\electricPotential}}}}\subseteq\spectrum{\closure{\schroedingerOperator_{\coverTwo{\magneticPotential},\coverTwo{\electricPotential}}}}$
by lifting eigenfunctions.

In the following, we collect results about magnetic Schr\"odinger
operators that describe periodic electromagnetic fields. Let $\coveringTransformationGroup$
denote the covering transformation group of the regular cover $\cover{\pi}\colon\cover{\baseManifold}\to\baseManifold$.
We assume that $\exteriorDifferential\cover{\magneticPotential}=\cover{\pi}^{*}\magneticField$
for some magnetic field $\magneticField\in\smoothTwoForms{\baseManifold,\mathbb{R}}$,
and that $\cover{\electricPotential}=\electricPotential\circ\cover{\pi}$
for some electric potential $\electricPotential\in\smoothFunctions{\baseManifold,\mathbb{R}}$.
As before, $[\magneticField]=0\in\homologyGroup^{2}(\baseManifold,\mathbb{R})$
is called the exact case, whereas $[\magneticField]\neq0\in\homologyGroup^{2}(\baseManifold,\mathbb{R})$
is called the monopole case. For any $\groupElement\in\coveringTransformationGroup$,
the $1$-form $\groupElement^{*}\cover{\magneticPotential}-\cover{\magneticPotential}$
is closed since $\exteriorDifferential(\groupElement^{*}\cover{\magneticPotential})=\groupElement^{*}\cover{\pi}^{*}\magneticField=\cover{\pi}^{*}\magneticField=\exteriorDifferential\cover{\magneticPotential}$,
and we define the $\cover{\magneticPotential}$-exact subgroup as
\[
\coveringTransformationGroup^{\cover{\magneticPotential}}=\left\{ \groupElement\in\coveringTransformationGroup\,|\,\groupElement^{*}\cover{\magneticPotential}-\cover{\magneticPotential}\textrm{ is exact}\right\} .
\]
Note that if $\cover{\baseManifold}=\universalCover{\baseManifold}{}$,
we have $\coveringTransformationGroup^{\cover{\magneticPotential}}=\coveringTransformationGroup=\pi_{1}(\baseManifold)$.
The subgroup property can be seen as follows. If $\groupElement,\groupElement'\in\coveringTransformationGroup^{\cover{\magneticPotential}}$
such that $\groupElement^{*}\cover{\magneticPotential}=\cover{\magneticPotential}+\exteriorDifferential f_{\groupElement}$
and $\groupElement'^{*}\cover{\magneticPotential}=\cover{\magneticPotential}+\exteriorDifferential f_{\groupElement'}$
where $f_{\groupElement},f_{\groupElement'}\in\smoothFunctions{\cover{\baseManifold},\mathbb{R}}$,
then 
\[
\left(\groupElement'\multipliedBy\groupElement\right)^{*}\cover{\magneticPotential}=\groupElement^{*}\left(\cover{\magneticPotential}+\exteriorDifferential f_{\groupElement'}\right)=\cover{\magneticPotential}+\exteriorDifferential\left(f_{\groupElement}+\groupElement^{*}f_{\groupElement'}\right).
\]

\begin{thm}
If $\coveringTransformationGroup^{\cover{\magneticPotential}}$ is
infinite, then all eigenspaces of $\closure{\schroedingerOperator_{\cover{\magneticPotential},\cover{\electricPotential}}}$
are infinite-dimensional, that is, $\closure{\schroedingerOperator_{\cover{\magneticPotential},\cover{\electricPotential}}}$
has no discrete spectrum.\label{thm:Eigenspaces_are_infinite_dimensional}
\end{thm}
The idea behind Theorem~\ref{thm:Eigenspaces_are_infinite_dimensional}
goes back to the following well-known argument for the special case
$\cover{\magneticPotential}=0$ and $\cover{\electricPotential}=0$.
Since Laplacians commute with isometries, any eigenfunction can be
translated by covering transformations to obtain linearly independent
eigenfunctions in the same eigenspace. In the presence of magnetic
potentials, one has to adapt the translations as follows. If $x_{0}\in\cover{\baseManifold}$
denotes some fixed reference point, then any $\groupElement\in\coveringTransformationGroup^{\cover{\magneticPotential}}$
gives rise to a unique $f_{\groupElement}\in\smoothFunctions{\cover{\baseManifold},\mathbb{R}}$
such that $\groupElement^{*}\cover{\magneticPotential}=\cover{\magneticPotential}+\exteriorDifferential f_{\groupElement}$
and $f_{\groupElement}(x_{0})=0$. Following~\cite{MathaiShubin2002},
we define the associated magnetic translation $T_{\groupElement}\colon L^{2}(\cover{\baseManifold})\to L^{2}(\cover{\baseManifold})$
as the unitary map given by $T_{\groupElement}u=e^{i\multipliedBy f_{\groupElement}}\groupElement^{*}u$.
\begin{lem}
\label{lem:Magnetic_translations_commute_with_MSO}Magnetic translations
map $\domain{\closure{\schroedingerOperator_{\cover{\magneticPotential},\cover{\electricPotential}}}}$
into itself and commute with $\closure{\schroedingerOperator_{\cover{\magneticPotential},\cover{\electricPotential}}}$,
that is, for any $\groupElement\in\coveringTransformationGroup^{\cover{\magneticPotential}}$
and $u\in\domain{\closure{\schroedingerOperator_{\cover{\magneticPotential},\cover{\electricPotential}}}}$,
we have $e^{i\multipliedBy f_{\groupElement}}\groupElement^{*}u\in\domain{\closure{\schroedingerOperator_{\cover{\magneticPotential},\cover{\electricPotential}}}}$
and 
\begin{equation}
\closure{\schroedingerOperator_{\cover{\magneticPotential},\cover{\electricPotential}}}\left(e^{i\multipliedBy f_{\groupElement}}\groupElement^{*}u\right)=e^{i\multipliedBy f_{\groupElement}}\groupElement^{*}\left(\closure{\schroedingerOperator_{\cover{\magneticPotential},\cover{\electricPotential}}}u\right).\label{eq:Action_of_symmetry_on_MSO}
\end{equation}
\end{lem}
\begin{proof}
We use ideas from~\cite{MathaiShubin2002}, where the underlying
twisted group algebra structure is exploited in detail. In order to
see that (\ref{eq:Action_of_symmetry_on_MSO}) holds for any $u\in\domain{\schroedingerOperator_{\cover{\magneticPotential},\cover{\electricPotential}}}=\smoothCompactlySupportedFunctions{\cover{\baseManifold}}$,
it suffices to show that for any $v\in\smoothCompactlySupportedFunctions{\cover{\baseManifold}}$
\begin{equation}
\innerProduct{\schroedingerOperator_{\cover{\magneticPotential},\cover{\electricPotential}}\multipliedBy T_{\groupElement}u}v=\innerProduct{T_{\groupElement}\multipliedBy\schroedingerOperator_{\cover{\magneticPotential},\cover{\electricPotential}}u}v.\label{eq:Magnetic_translations_commute_with_MSO}
\end{equation}
One can prove (\ref{eq:Magnetic_translations_commute_with_MSO}) directly
using (\ref{eq:Coordinatefree_Version_of_MSO}). In order to avoid
this tedious calculation, we define a magnetic translation $T_{\groupElement}$
on $\smoothCompactlySupportedOneForms{\cover{\baseManifold}}$ given
as $T_{\groupElement}\omega=e^{i\multipliedBy f_{\groupElement}}\groupElement^{*}\omega$
for $\omega\in\smoothCompactlySupportedOneForms{\cover{\baseManifold}}$.
We obtain $\magneticDifferential{\cover{\magneticPotential}}\circ T_{\groupElement}=T_{\groupElement}\circ\magneticDifferential{\cover{\magneticPotential}}$
by noting that for any $u\in\smoothCompactlySupportedFunctions{\cover{\baseManifold}}$,
\begin{eqnarray*}
\magneticDifferential{\cover{\magneticPotential}}\multipliedBy T_{\groupElement}u & = & -i\multipliedBy\exteriorDifferential(e^{i\multipliedBy f_{\groupElement}}\groupElement^{*}u)+e^{i\multipliedBy f_{\groupElement}}\groupElement^{*}u\multipliedBy\cover{\magneticPotential}\\
 & = & e^{i\multipliedBy f_{\groupElement}}(-i\multipliedBy\groupElement^{*}\exteriorDifferential u+\groupElement^{*}u(\exteriorDifferential f_{\groupElement}+\cover{\magneticPotential}))=e^{i\multipliedBy f_{\groupElement}}\groupElement^{*}(-i\multipliedBy\exteriorDifferential u+u\multipliedBy\cover{\magneticPotential})=T_{\groupElement}\multipliedBy\magneticDifferential{\cover{\magneticPotential}}u,
\end{eqnarray*}
where we used that $\exteriorDifferential f_{\groupElement}+\cover{\magneticPotential}=\groupElement^{*}\cover{\magneticPotential}$.
Taking adjoints, we obtain $T_{\groupElement}^{-1}\circ\adjoint{\magneticDifferential{\cover{\magneticPotential}}}=\adjoint{\magneticDifferential{\cover{\magneticPotential}}}\circ T_{\groupElement}^{-1}$
on $\smoothCompactlySupportedOneForms{\cover{\baseManifold}}$. Since
$T_{\groupElement}^{-1}\circ\cover{\electricPotential}\circ T_{\groupElement}=\cover{\electricPotential}$
trivially holds on $\smoothCompactlySupportedFunctions{\cover{\baseManifold}}$,
the claim $\schroedingerOperator_{\cover{\magneticPotential},\cover{\electricPotential}}\multipliedBy T_{\groupElement}=T_{\groupElement}\multipliedBy\schroedingerOperator_{\cover{\magneticPotential},\cover{\electricPotential}}$
follows directly from (\ref{eq:Def_of_MSO}). Let now $u\in\domain{\closure{\schroedingerOperator_{\cover{\magneticPotential},\cover{\electricPotential}}}}$
and choose a sequence $(u_{n})_{n\in\mathbb{N}}$ in $\smoothCompactlySupportedFunctions{\cover{\baseManifold}}=\domain{\schroedingerOperator_{\cover{\magneticPotential},\cover{\electricPotential}}}$
such that $\normComingFromInnerProduct{u-u_{n}}\to0$ and $\normComingFromInnerProduct{\closure{\schroedingerOperator_{\cover{\magneticPotential},\cover{\electricPotential}}}u-\schroedingerOperator_{\cover{\magneticPotential},\cover{\electricPotential}}u_{n}}\to0$.
The translated sequence $(T_{\groupElement}u_{n})_{n\in\mathbb{N}}$
also lies in $\smoothCompactlySupportedFunctions{\cover{\baseManifold}}$
and satisfies $\normComingFromInnerProduct{T_{\groupElement}u-T_{\groupElement}u_{n}}=\normComingFromInnerProduct{u-u_{n}}\to0$
as well as 
\[
\normComingFromInnerProductThatAdapts{T_{\groupElement}\multipliedBy\closure{\schroedingerOperator_{\cover{\magneticPotential},\cover{\electricPotential}}}u-\schroedingerOperator_{\cover{\magneticPotential},\cover{\electricPotential}}\multipliedBy T_{\groupElement}u_{n}}=\normComingFromInnerProductThatAdapts{T_{\groupElement}\multipliedBy\closure{\schroedingerOperator_{\cover{\magneticPotential},\cover{\electricPotential}}}u-T_{\groupElement}\multipliedBy\schroedingerOperator_{\cover{\magneticPotential},\cover{\electricPotential}}u_{n}}\to0
\]
since $T_{\groupElement}$ is unitary. Hence, $T_{\groupElement}u\in\domain{\closure{\schroedingerOperator_{\cover{\magneticPotential},\cover{\electricPotential}}}}$
and $\closure{\schroedingerOperator_{\cover{\magneticPotential},\cover{\electricPotential}}}\multipliedBy T_{\groupElement}u=T_{\groupElement}\multipliedBy\closure{\schroedingerOperator_{\cover{\magneticPotential},\cover{\electricPotential}}}u$
as claimed.
\end{proof}

\begin{proof}
[Proof of Theorem \ref{thm:Eigenspaces_are_infinite_dimensional}]We
largely follow~\cite{Sunada1988} and assume that $\closure{\schroedingerOperator_{\cover{\magneticPotential},\cover{\electricPotential}}}$
has an eigenspace with finite orthonormal basis $\{u_{1},u_{2},\ldots,u_{N}\}$.
For each $\groupElement\in\coveringTransformationGroup^{\cover{\magneticPotential}}$
and $j\in\{1,2,\ldots,N\}$, the translate $T_{\groupElement}u_{j}$
is an eigenfunction of $\closure{\schroedingerOperator_{\cover{\magneticPotential},\cover{\electricPotential}}}$
by virtue of Lemma~\ref{lem:Magnetic_translations_commute_with_MSO}.
Hence, 
\[
T_{\groupElement}u_{j}=\sum_{k=1}^{N}U_{jk}(\groupElement)\multipliedBy u_{k}
\]
for some matrix $U(\groupElement)$, which is unitary since
\[
\sum_{k=1}^{N}U_{jk}(\groupElement)\multipliedBy\conjugate{U_{lk}(\groupElement)}=\sum_{k,m=1}^{N}U_{jk}(\groupElement)\multipliedBy\conjugate{U_{lm}}(\groupElement)\multipliedBy\innerProduct{u_{k}}{u_{m}}=\innerProduct{T_{\groupElement}u_{j}}{T_{\groupElement}u_{l}}=\innerProduct{u_{j}}{u_{l}}.
\]
Let $F$ denote a fundamental domain for the action of $\coveringTransformationGroup^{\cover{\magneticPotential}}$
on $\cover{\baseManifold}$. We obtain 
\begin{eqnarray*}
N=\sum_{j=1}^{N}\normComingFromInnerProduct{u_{j}}^{2} & = & \sum_{j=1}^{N}\sum_{\groupElement\in\coveringTransformationGroup^{\cover{\magneticPotential}}}\int_{F}\norm{\groupElement^{*}u_{j}}^{2}=\sum_{\groupElement\in\coveringTransformationGroup^{\cover{\magneticPotential}}}\sum_{j=1}^{N}\int_{F}\norm{T_{\groupElement}\multipliedBy u_{j}}^{2}\\
 & = & \sum_{\groupElement\in\coveringTransformationGroup^{\cover{\magneticPotential}}}\sum_{j=1}^{N}\sum_{k,l=1}^{N}\int_{F}U_{jk}(\groupElement)\multipliedBy u_{k}\multipliedBy\conjugate{U_{jl}(\groupElement)}\multipliedBy\conjugate{u_{l}}=\norm{\coveringTransformationGroup^{\cover{\magneticPotential}}}\multipliedBy\sum_{k=1}^{N}\int_{F}\norm{u_{k}}^{2},
\end{eqnarray*}
which contradicts $\norm{\coveringTransformationGroup^{\cover{\magneticPotential}}}=\infty$.
\end{proof}

\subsection{Ground state energy\label{sub:GSE}}

\global\long\def\unitaryRepresentation{\eta}

Let $\cover{\pi}\colon\cover{\baseManifold}\to\baseManifold$ be a
regular cover equipped with potentials $\cover{\magneticPotential}\in\smoothOneForms{\cover{\baseManifold},\mathbb{R}}$
and $\cover{\electricPotential}=\electricPotential\circ\cover{\pi}$,
where $\electricPotential\in\smoothFunctions{\baseManifold,\mathbb{R}}$.
\begin{defn}
The ground state energy\emph{ }is defined as
\[
\groundStateEnergy(\cover{\magneticPotential},\cover{\electricPotential})=\groundStateEnergy(\closure{\schroedingerOperator_{\cover{\magneticPotential},\cover{\electricPotential}}})=\inf\multipliedBy\spectrum{\closure{\schroedingerOperator_{\cover{\magneticPotential},\cover{\electricPotential}}}}.
\]
Any normalized eigenfunction of $\closure{\schroedingerOperator_{\cover{\magneticPotential},\cover{\electricPotential}}}$
with eigenvalue $\groundStateEnergy(\cover{\magneticPotential},\cover{\electricPotential})$
is called a ground state.
\end{defn}
If $\cover{\baseManifold}$ is non-compact, then $\groundStateEnergy(\cover{\magneticPotential},\cover{\electricPotential})$
can belong to the continuous spectrum $\continuousSpectrum{\closure{\schroedingerOperator_{\cover{\magneticPotential},\cover{\electricPotential}}}}$.
In any case, the variational principle says that $\groundStateEnergy(\cover{\magneticPotential},\cover{\electricPotential})$
is an infimum of Rayleigh quotients 
\[
\groundStateEnergy(\cover{\magneticPotential},\cover{\electricPotential})=\inf_{u\in\smoothCompactlySupportedFunctions{\cover{\baseManifold}\targetC}\backslash\{0\}}\frac{\innerProduct{\schroedingerOperator_{\cover{\magneticPotential},\cover{\electricPotential}}u}u}{\innerProduct uu}=\inf_{u\in\smoothCompactlySupportedFunctions{\cover{\baseManifold}\targetC}\backslash\{0\}}\frac{\int_{\cover{\baseManifold}}\frac{1}{2}\norm{\magneticDifferential{\cover{\magneticPotential}}u}^{2}+\cover{\electricPotential}\multipliedBy\norm u^{2}}{\int_{\cover{\baseManifold}}\norm u^{2}}\geq\min\,\electricPotential.
\]
The reader is referred to~\cite[Section 4]{Davies1995} for details.
If $\electricPotential=0$, we use the shortened notation $\groundStateEnergy(\cover{\magneticPotential})$
for $\groundStateEnergy(\cover{\magneticPotential},0)$. We recall
that a distribution $\nu$ on $\cover{\baseManifold}$ is called positive,
denoted by $\nu\geqslant0$, if $\nu(f)\geq0$ for every non-negative
$f\in\smoothCompactlySupportedFunctions{\cover{\baseManifold},\mathbb{R}}$.
The following version of Kato's inequality is a special case of~\cite[Proposition 5.9 and Corollary 5.10]{BravermanMilatovichShubin2002},
see also \cite[Proposition 2.2]{HessSchraderUhlenbrock1980}.
\begin{prop}
\label{prop:Katos_Inequality}For any $u\in\smoothCompactlySupportedFunctions{\cover{\baseManifold}}$
and $\varepsilon>0$, let $\norm u_{\varepsilon}\in\smoothCompactlySupportedFunctions{\cover{\baseManifold}}$
be defined as
\[
\norm u_{\varepsilon}=\sqrt{\norm u^{2}+\varepsilon^{2}}-\varepsilon.
\]
Then, we have the following inequality of smooth functions on $\cover{\baseManifold}$
\[
\left(\norm u_{\varepsilon}+\varepsilon\right)\multipliedBy\schroedingerOperator_{0,0}\norm u_{\varepsilon}\leq\re\innerProduct{\schroedingerOperator_{\cover{\magneticPotential},0}u}u.
\]
In the limit $\varepsilon\searrow0$, we moreover have the following
inequality of distributions 
\[
\schroedingerOperator_{0,0}\norm u\leqslant\re\innerProduct{\schroedingerOperator_{\cover{\magneticPotential},0}u}{\mathrm{sign}\multipliedBy u},
\]
where 
\[
\mathrm{sign}\multipliedBy u=\begin{cases}
\frac{u(x)}{\norm{u(x)}} & \textrm{for }u(x)\neq0,\\
0 & \textrm{otherwise.}
\end{cases}
\]

\end{prop}
For compact $\cover{\baseManifold}$, let $\homologyGroup^{1}(\cover{\baseManifold},\mathbb{Z})$
denote the first cohomology group with integer coefficients, which
we identify with the lattice group
\[
\left\{ [\omega]\in\homologyGroup^{1}(\cover{\baseManifold},\mathbb{R})\,\Big|\,\int_{\gamma}\omega\in\mathbb{Z}\textrm{ for any closed curve }\gamma\textrm{ in }\cover{\baseManifold}\right\} .
\]
The following theorem describes the diamagnetic effect of magnetic
fields and the gauge invariance group of compact covers.
\begin{thm}
\label{thm:Gauge_Invariance}For any regular cover $\cover{\pi}\colon\cover{\baseManifold}\to\baseManifold$
equipped with potentials $\cover{\magneticPotential}\in\smoothOneForms{\cover{\baseManifold},\mathbb{R}}$
and $\cover{\electricPotential}=\electricPotential\circ\cover{\pi}$,
where $\electricPotential\in\smoothFunctions{\baseManifold,\mathbb{R}}$,
the diamagnetic inequality holds, namely 
\begin{equation}
\groundStateEnergy(0,\cover{\electricPotential})\leq\groundStateEnergy(\cover{\magneticPotential},\cover{\electricPotential}).\label{eq:Diamagnetic_Inequality}
\end{equation}
Moreover, any $\omega=\frac{1}{i}\frac{\exteriorDifferential\varphi}{\varphi}$
with $\varphi\in\smoothFunctions{\cover{\baseManifold},\mathbb{S}^{1}}\subset\smoothFunctions{\cover{\baseManifold}}$
gives rise to a unitary gauge transformation $\mathcal{U}_{\varphi}\colon L^{2}(\cover{\baseManifold})\to L^{2}(\cover{\baseManifold})$
given by $\mathcal{U}_{\varphi}u=\varphi\multipliedBy u$ such that
\begin{equation}
\schroedingerOperator_{\cover{\magneticPotential}+\omega,\cover{\electricPotential}}=\adjoint{\mathcal{U}_{\varphi}}\multipliedBy\schroedingerOperator_{\cover{\magneticPotential},\cover{\electricPotential}}\multipliedBy\mathcal{U}_{\varphi}\quad\textnormal{on }\smoothCompactlySupportedFunctions{\cover{\baseManifold}}.\label{eq:Gauge_Transformation}
\end{equation}
In particular,
\begin{equation}
\groundStateEnergy(\cover{\magneticPotential}+\exteriorDifferential f,\cover{\electricPotential})=\groundStateEnergy(\cover{\magneticPotential},\cover{\electricPotential})\quad\textnormal{for }f\in\smoothFunctions{\cover{\baseManifold},\mathbb{R}}.\label{eq:GSE_of_Mag_Pot_plus_df}
\end{equation}
If $\cover{\baseManifold}$ is a finite and therefore compact regular
cover of $\baseManifold$, then the following are equivalent:
\begin{enumerate}
\item $\groundStateEnergy(\cover{\magneticPotential},\cover{\electricPotential})=\groundStateEnergy(0,\cover{\electricPotential}),$\label{enu:Equivalent_GSE}
\item \textup{$\closure{\schroedingerOperator_{\cover{\magneticPotential},\cover{\electricPotential}}}$
and $\closure{\schroedingerOperator_{0,\cover{\electricPotential}}}$
are unitarily equivalent via a gauge transformation, in particular,\label{enu:Gauge_Equivalence}
\[
\spectrum{\closure{\schroedingerOperator_{\cover{\magneticPotential},\cover{\electricPotential}}}}=\spectrum{\closure{\schroedingerOperator_{0,\cover{\electricPotential}}}},
\]
}
\item $\cover{\magneticPotential}=\frac{1}{i}\frac{\exteriorDifferential\varphi}{\varphi}$
for some $\varphi\in\smoothFunctions{\cover{\baseManifold},\mathbb{S}^{1}}\subset\smoothFunctions{\cover{\baseManifold}}$,\label{enu:Potential_comes_from_Map_to_Circle}
\item $\cover{\magneticPotential}$ is closed and $[\cover{\magneticPotential}]\in2\pi\multipliedBy\homologyGroup^{1}(\cover{\baseManifold},\mathbb{Z})$,
that is, for any closed curve $\gamma$ in $\cover{\baseManifold}$\label{enu:Potential_in_integral_Homology}
\[
\int_{\gamma}\cover{\magneticPotential}\in2\pi\multipliedBy\mathbb{Z}.
\]

\end{enumerate}
\end{thm}
\begin{proof}
The lifted potential $\cover{\electricPotential}$ is bounded from
below by $\min\electricPotential$, which exists by compactness of
$\baseManifold$. Hence, we may assume that $\cover{\electricPotential}\geq0$,
otherwise consider $\schroedingerOperator_{\cover{\magneticPotential},\cover{\electricPotential}}-\min\electricPotential$.
In order to prove inequality~(\ref{eq:Diamagnetic_Inequality}),
one can argue along the lines of~\cite[Theorem 3.3]{HessSchraderUhlenbrock1977}
and use Kato's inequality as stated in Proposition~\ref{prop:Katos_Inequality}
to verify that the semigroup $e^{-t\multipliedBy\closure{\schroedingerOperator_{0,\cover{\electricPotential}}}}$
dominates the semigroup $e^{-t\multipliedBy\closure{\schroedingerOperator_{\cover{\magneticPotential},\cover{\electricPotential}}}}$,
which implies (\ref{eq:Diamagnetic_Inequality}) by virtue of \cite[Corollary 2.13]{HessSchraderUhlenbrock1977}.
The gauge equivalence~(\ref{eq:Gauge_Transformation}) is the content
of~\cite[Proposition 3.2]{Shigekawa1987}, where the local expression~(\ref{eq:Coordinatefree_Version_of_MSO})
is used. In order to provide an alternative proof, we let $\mathcal{W}_{\varphi}\colon L^{2}(\smoothOneForms{\cover{\baseManifold}})\to L^{2}(\smoothOneForms{\cover{\baseManifold}})$
denote the unitary map given by $\mathcal{W}_{\varphi}\omega=\varphi\multipliedBy\omega$.
Note that $\mathcal{U}_{\varphi}(\smoothCompactlySupportedFunctions{\cover{\baseManifold}})=\smoothCompactlySupportedFunctions{\cover{\baseManifold}}$
and $\adjoint{\mathcal{U}_{\varphi}}=\mathcal{U}_{\varphi}^{-1}=\mathcal{U}_{\varphi^{-1}}$
as well as $\mathcal{W}_{\varphi}(\smoothCompactlySupportedOneForms{\cover{\baseManifold}})=\smoothCompactlySupportedOneForms{\cover{\baseManifold}}$
and $\adjoint{\mathcal{W}_{\varphi}}=\mathcal{W}_{\varphi}^{-1}=\mathcal{W}_{\varphi^{-1}}$.
For any $u\in\smoothCompactlySupportedFunctions{\cover{\baseManifold}}$,
we have
\[
\magneticDifferential{\cover{\magneticPotential}}\multipliedBy\mathcal{U}_{\varphi}u=\magneticDifferential{\cover{\magneticPotential}}(\varphi\multipliedBy u)=-i\multipliedBy\varphi\multipliedBy\exteriorDifferential\multipliedBy u+u\left(-i\multipliedBy\exteriorDifferential\varphi+\varphi\multipliedBy\cover{\magneticPotential}\right)=\varphi\multipliedBy\magneticDifferential{\cover{\magneticPotential}+\omega}u=\mathcal{W}_{\varphi}\multipliedBy\magneticDifferential{\cover{\magneticPotential}+\omega}u.
\]
Taking adjoints, we get $\adjoint{\mathcal{U}_{\varphi}}\multipliedBy\adjoint{\magneticDifferential{\cover{\magneticPotential}}}=\adjoint{\magneticDifferential{\cover{\magneticPotential}+\omega}}\multipliedBy\adjoint{\mathcal{W}_{\varphi}}$
on $\smoothCompactlySupportedOneForms{\cover{\baseManifold}}$. In
addition, $\adjoint{\mathcal{U}_{\varphi}}\multipliedBy\cover{\electricPotential}\multipliedBy\mathcal{U}_{\varphi}=\cover{\electricPotential}$
trivially holds on $\smoothCompactlySupportedFunctions{\cover{\baseManifold}}$.
Hence, we get~(\ref{eq:Gauge_Transformation}), which in turn implies
(\ref{eq:GSE_of_Mag_Pot_plus_df}) by setting $\varphi=e^{if}$.

In what follows, let $\cover{\baseManifold}$ be compact. The equivalence
of \ref{enu:Potential_comes_from_Map_to_Circle}. and \ref{enu:Potential_in_integral_Homology}.
is the content of~\cite[Proposition 3.1]{Shigekawa1987}. Since \ref{enu:Potential_comes_from_Map_to_Circle}.
implies \ref{enu:Gauge_Equivalence}. by the first part, and since
\ref{enu:Gauge_Equivalence}. immdiately gives \ref{enu:Equivalent_GSE}.,
it suffices to show that \ref{enu:Equivalent_GSE}. implies \ref{enu:Potential_comes_from_Map_to_Circle}.
We use ideas from~\cite{Helffer1988}. It is well-known that $\closure{\schroedingerOperator_{0,\cover{\electricPotential}}}$
has a smooth positive ground state $u_{0}$, in particular, for any
$u\in\smoothFunctions{\cover{\baseManifold}}$
\begin{equation}
\groundStateEnergy(0,\cover{\electricPotential})\multipliedBy\innerProduct{u_{0}^{-1}\multipliedBy\norm u^{2}}{u_{0}}=\innerProduct{u_{0}^{-1}\multipliedBy\norm u^{2}}{\schroedingerOperator_{0,\cover{\electricPotential}}u_{0}}=\frac{1}{2}\innerProduct{\exteriorDifferential(u_{0}^{-1}\multipliedBy\norm u^{2})}{\exteriorDifferential u_{0}}+\innerProduct{u_{0}^{-1}\multipliedBy\norm u^{2}\multipliedBy\cover{\electricPotential}}{u_{0}}.\label{eq:Inner_product_with_GS}
\end{equation}
This yields another proof of the diamagnetic inequality~(\ref{eq:Diamagnetic_Inequality})
as follows 
\begin{eqnarray}
\normComingFromInnerProduct{u_{0}\multipliedBy\magneticDifferential{\cover{\magneticPotential}}(u_{0}^{-1}\multipliedBy u)}^{2} & = & \normComingFromInnerProduct{i\multipliedBy u_{0}^{-1}\multipliedBy u\multipliedBy\exteriorDifferential u_{0}-i\exteriorDifferential u+u\multipliedBy\cover{\magneticPotential}}^{2}\nonumber \\
 & = & \normComingFromInnerProduct{-i\exteriorDifferential u+u\multipliedBy\cover{\magneticPotential}}^{2}+2\multipliedBy\mathrm{Re}\innerProduct{-i\exteriorDifferential u+u\multipliedBy\cover{\magneticPotential}}{i\multipliedBy u_{0}^{-1}\multipliedBy u\multipliedBy\exteriorDifferential u_{0}}+\normComingFromInnerProduct{u_{0}^{-1}\multipliedBy u\multipliedBy\exteriorDifferential u_{0}}^{2}\nonumber \\
 & = & \normComingFromInnerProduct{\magneticDifferential{\cover{\magneticPotential}}u}^{2}-2\multipliedBy\re\innerProduct{u_{0}^{-1}\multipliedBy\conjugate u\multipliedBy\exteriorDifferential u}{\exteriorDifferential u_{0}}+\innerProduct{u_{0}^{-2}\multipliedBy\norm u^{2}\multipliedBy\exteriorDifferential u_{0}}{\exteriorDifferential u_{0}}\nonumber \\
 & = & \normComingFromInnerProduct{\magneticDifferential{\cover{\magneticPotential}}u}^{2}-\innerProduct{\exteriorDifferential(u_{0}^{-1}\multipliedBy\norm u^{2})}{\exteriorDifferential u_{0}}\nonumber \\
 & = & \normComingFromInnerProduct{\magneticDifferential{\cover{\magneticPotential}}u}^{2}+2\innerProduct{u_{0}^{-1}\multipliedBy\norm u^{2}\multipliedBy\cover{\electricPotential}}{u_{0}}-2\groundStateEnergy(0,\cover{\electricPotential})\multipliedBy\innerProduct{u_{0}^{-1}\multipliedBy\norm u^{2}}{u_{0}},\nonumber \\
 & = & 2\innerProductThatAdapts{\left(\schroedingerOperator_{\cover{\magneticPotential},\cover{\electricPotential}}-\groundStateEnergy(0,\cover{\electricPotential})\right)u}u,\label{eq:Diamagnetic_Inequality_Alternative_Proof}
\end{eqnarray}
where we used $\re\innerProduct{\norm u^{2}\multipliedBy\cover{\magneticPotential}}{i\multipliedBy u_{0}^{-1}\multipliedBy\exteriorDifferential u_{0}}=0$
and (\ref{eq:Inner_product_with_GS}). Recall that $\closure{\schroedingerOperator_{\cover{\magneticPotential},\cover{\electricPotential}}}$
has a smooth normalized ground state~\cite[Theorem 2.1]{Shigekawa1987},
which we denote by $u_{\cover{\magneticPotential}}$. Assuming that
$\groundStateEnergy(0,\cover{\electricPotential})=\groundStateEnergy(\cover{\magneticPotential},\cover{\electricPotential})$,
we show that $u_{\cover{\magneticPotential}}$ is non-vanishing. Due
to~(\ref{eq:Diamagnetic_Inequality_Alternative_Proof}), the function
$\varphi_{\cover{\magneticPotential}}=u_{0}^{-1}\multipliedBy u_{\cover{\magneticPotential}}$
satisfies $\magneticDifferential{\cover{\magneticPotential}}\varphi_{\cover{\magneticPotential}}=0$,
that is,
\begin{equation}
\exteriorDifferential\varphi_{\cover{\magneticPotential}}=-i\multipliedBy\varphi_{\cover{\magneticPotential}}\multipliedBy\cover{\magneticPotential}.\label{eq:Differential_of_qoutient}
\end{equation}
We show that whenever $\varphi_{\cover{\magneticPotential}}(x_{0})=0$
for some $x_{0}\in\cover{\baseManifold}$, then $\varphi_{\cover{\magneticPotential}}$
vanishes in a neighborhood of $x_{0}$, which by continuity of $\varphi_{\cover{\magneticPotential}}$
and connectivity of $\cover{\baseManifold}$ leads to $\varphi_{\cover{\magneticPotential}}=0$
contradicting $\normComingFromInnerProduct{u_{\cover{\magneticPotential}}}=1$.
Using~(\ref{eq:Differential_of_qoutient}), we can find $r_{0}>0$
and $C>0$ such that the exponential map $\exp_{x_{0}}\colon T_{x_{0}}\cover{\baseManifold}\to\cover{\baseManifold}$
restricted to the ball $B(0,r_{0})$ or radius $r_{0}$ is a diffeomorphism
onto its image $B(x_{0},r_{0})=\exp_{x_{0}}(B(0,r_{0}))$, on which
\[
\norm{\exteriorDifferential\varphi_{\cover{\magneticPotential}}}\leq C\multipliedBy\norm{\varphi_{\cover{\magneticPotential}}}
\]
holds pointwise. For $x\in B(x_{0},r_{0})$, let $v_{x}=\exp_{x_{0}}^{-1}x$,
and note that since $\varphi_{\cover{\magneticPotential}}$ is smooth,
\[
\varphi_{\cover{\magneticPotential}}(x)=\varphi_{\cover{\magneticPotential}}(x)-\varphi_{\cover{\magneticPotential}}(x_{0})=\int_{0}^{1}\frac{d}{dt}\left(\varphi_{\cover{\magneticPotential}}\left(\exp_{x_{0}}\left(t\multipliedBy v_{x}\right)\right)\right)dt.
\]
Using the Gauß lemma, we obtain that for any $0<r<r_{0}$
\[
\sup_{x\in B(x_{0},r)}\norm{\varphi_{\cover{\magneticPotential}}(x)}\leq C\multipliedBy r\multipliedBy\sup_{x\in B(x_{0},r)}\norm{\varphi_{\cover{\magneticPotential}}(x)},
\]
and $\varphi_{\cover{\magneticPotential}}$ must vanish on any $B(x_{0},r)$
with $0<r<\min(r_{0},C)$. Hence, $\varphi_{\cover{\magneticPotential}}$
vanishes nowhere, and on any simply-connected coordinate neighborhood
$U\subseteq\cover{\baseManifold}$, we may write $\varphi_{\cover{\magneticPotential}}=R\multipliedBy e^{if}$
for some $R\in\smoothFunctions{U,\mathbb{R}^{+}}$ and $f\in\smoothFunctions{U,\mathbb{R}}$.
Using~(\ref{eq:Differential_of_qoutient}), we obtain 
\[
\cover{\magneticPotential}=i\multipliedBy\varphi_{\cover{\magneticPotential}}^{-1}\multipliedBy\exteriorDifferential\varphi_{\cover{\magneticPotential}}=i\multipliedBy R^{-1}\exteriorDifferential R-df.
\]
Since $\cover{\magneticPotential}$ is real-valued, we have $\cover{\magneticPotential}=-\exteriorDifferential f$
and $\exteriorDifferential R=0$, that is, $\cover{\magneticPotential}$
is closed and $|\varphi_{\cover{\magneticPotential}}|$ is constant
on $\cover{\baseManifold}$. Hence, $\cover{\magneticPotential}=\frac{1}{i}\frac{\exteriorDifferential\varphi}{\varphi}$
for $\varphi=\frac{\conjugate{\varphi_{\cover{\magneticPotential}}}}{|\varphi_{\cover{\magneticPotential}}|}\in\smoothFunctions{\cover{\baseManifold},\mathbb{S}^{1}}$
as claimed.
\end{proof}
Note that we can have $\groundStateEnergy(\cover{\magneticPotential},\cover{\electricPotential})\neq\groundStateEnergy(0,\cover{\electricPotential})$
even if $\exteriorDifferential\cover{\magneticPotential}=0$, which
is known as the Aharonov-Bohm effect.
\begin{prop}
\label{prop:GSE_if_norm_constant}If $\cover{\magneticPotential}$
has constant norm $\norm{\cover{\magneticPotential}}$ on $\cover{\baseManifold}$,
then the function $\mu_{0}\colon\mathbb{R}\to\mathbb{R}$ given by
$\mu_{0}(\magneticFieldStrength)=\groundStateEnergy(\magneticFieldStrength\cover{\magneticPotential},\cover{\electricPotential})-\frac{1}{2}\magneticFieldStrength^{2}\norm{\cover{\magneticPotential}}^{2}$
is concave. In particular, $\mu_{0}$ is continuous, admits monotonically
non-increasing left and right derivatives, and is differentiable at
all but at most countably many points.\end{prop}
\begin{proof}
The operator $\schroedingerOperator_{\magneticFieldStrength\cover{\magneticPotential},\cover{\electricPotential}}-\frac{1}{2}\magneticFieldStrength^{2}\norm{\cover{\magneticPotential}}^{2}$
with domain $\smoothCompactlySupportedFunctions{\cover{\baseManifold}}$
is of the form $\mathcal{K}+\magneticFieldStrength\multipliedBy\mathcal{L}$,
where
\[
\mathcal{K}u=\frac{1}{2}\Delta u+\cover{\electricPotential}\multipliedBy u\qquad\textnormal{and}\qquad\mathcal{L}u=-i\innerProduct{\exteriorDifferential u}{\cover{\magneticPotential}}+\frac{1}{2}\multipliedBy\exteriorDifferential^{*}\cover{\magneticPotential}\multipliedBy u,
\]
see also~\cite[Proposition 4.11]{AvronHerbstSimon1978}. Since $\mathcal{K}+\magneticFieldStrength\multipliedBy\mathcal{L}$
is affine linear in $\magneticFieldStrength$, its ground state $\mu_{0}$
is concave in $\magneticFieldStrength$~\cite[Theorem 3.5.21]{Thirring2002},
more precisely, for $\magneticFieldStrength_{1},\magneticFieldStrength_{2}\in\mathbb{R}$
and $t\in[0,1]$, we have
\begin{eqnarray*}
\mu_{0}(t\magneticFieldStrength_{1}+(1-t)\magneticFieldStrength_{2}) & = & \inf_{\normComingFromInnerProduct u=1}\innerProduct{(t(\mathcal{K}+\magneticFieldStrength_{1}\mathcal{L})+(1-t)(\mathcal{K}+\magneticFieldStrength_{2}\mathcal{L}))u}u\\
 & \geq & \inf_{\normComingFromInnerProduct u=1}t\innerProduct{(\mathcal{K}+\magneticFieldStrength_{1}\mathcal{L})u}u+\inf_{\normComingFromInnerProduct u=1}(1-t)\innerProduct{(\mathcal{K}+\magneticFieldStrength_{2}\mathcal{L})u}u\\
 & = & t\mu_{0}(\magneticFieldStrength_{1})+(1-t)\mu_{0}(\magneticFieldStrength_{2}).
\end{eqnarray*}
\vspace{-10mm}

\end{proof}
In Section~\ref{sub:HG}, we present a non-compact quotient of the
Heisenberg group equipped with a left-invariant magnetic potential,
such that $\groundStateEnergy$ has countably many local minima, at
which it is not differentiable. Such points are referred to as phase
transitions in~\cite{HiguchiShirai1999}, which contains similar
examples for magnetic Schr\"odinger operators on periodic graphs.

\subsection{Exact case}

\global\long\def\character{\chi}

\global\long\def\characterGroup#1{\hat{#1}}

\global\long\def\trivialRepresentation{\boldsymbol{1}}

\global\long\def\hilbertSpace{P}

\global\long\def\orbit#1#2{[#1,#2]}

For the remainder of this section, we consider lifts of potentials
$\magneticPotential\in\smoothOneForms{\baseManifold,\mathbb{R}}$
and $\electricPotential\in\smoothFunctions{\baseManifold,\mathbb{R}}$.
It is well-known that $\schroedingerOperator_{0,\electricPotential}$
has a real-valued non-vanishing ground state. Using standard pertubation
theory, Shigekawa~\cite[Theorem 5.1 and Proposition 4.4]{Shigekawa1987}
obtained the following generalization.
\begin{prop}
(\cite{Shigekawa1987})\label{prop:GSE_simple_for_small_B} There
exists $\varepsilon>0$ such that for all $\magneticFieldStrength\in(-\varepsilon,\varepsilon)$
the ground state energy $\groundStateEnergy(\magneticFieldStrength\multipliedBy\magneticPotential,\electricPotential)$
is a simple eigenvalue of $\closure{\schroedingerOperator_{\magneticFieldStrength\magneticPotential,\electricPotential}}$
with non-vanishing smooth eigenfunction, and we have
\[
\groundStateEnergy(\magneticFieldStrength\multipliedBy\magneticPotential,\electricPotential)=\inf_{\omega\in2\pi\multipliedBy\homologyGroup^{1}(\baseManifold,\mathbb{Z})}\groundStateEnergy\left(0,\frac{1}{2}\norm{\magneticFieldStrength\multipliedBy\magneticPotential-\omega}^{2}+\electricPotential\right).
\]

\end{prop}
Following the proof of~\cite[Theorem 4.3]{Shigekawa1987}, we see
that any real-valued ground state $u_{\magneticPotential}$ of $\closure{\schroedingerOperator_{\magneticPotential,\electricPotential}}$
is also a ground state of $\closure{\schroedingerOperator_{0,\electricPotential+\frac{1}{2}\norm{\magneticPotential}^{2}}}$
since
\begin{eqnarray*}
\groundStateEnergy(\magneticPotential,\electricPotential) & = & \inf_{u\in\smoothFunctions{\baseManifold,\mathbb{R}}\colon\normComingFromInnerProduct u=1}\int_{\baseManifold}\frac{1}{2}\norm{-i\multipliedBy\exteriorDifferential u+u\multipliedBy\magneticPotential}^{2}+\electricPotential\multipliedBy\norm u^{2}\\
 & = & \inf_{u\in\smoothFunctions{\baseManifold,\mathbb{R}}\colon\normComingFromInnerProduct u=1}\int_{\baseManifold}\frac{1}{2}\norm{\exteriorDifferential u}^{2}+\left(\frac{1}{2}\norm{\magneticPotential}^{2}+\electricPotential\right)\multipliedBy\norm u^{2}=\groundStateEnergy(0,\electricPotential+\frac{1}{2}\norm{\magneticPotential}^{2}).
\end{eqnarray*}
In particular, $u_{\magneticPotential}$ has no zeros and the space
of real-valued ground states is at most one-dimensional. In the following,
we let $\electricPotential=0$ and consider the function $\eigenvalue_{0}^{\baseManifold}\colon\mathbb{R}\to\mathbb{R}$
given by $\eigenvalue_{0}^{\baseManifold}(\magneticFieldStrength)=\groundStateEnergy(\magneticFieldStrength\magneticPotential)$.
Using Hodge decomposition of $\smoothOneForms{\baseManifold,\mathbb{R}}$,
we can write $\magneticPotential=\magneticPotential_{\mathrm{cc}}+df_{\magneticPotential}$,
where $\magneticPotential_{\mathrm{cc}}\in\smoothOneForms{\baseManifold,\mathbb{R}}$
is coclosed and $f_{\magneticPotential}$ minimizes $\normComingFromInnerProduct{\magneticPotential-df}$
on $\smoothFunctions{\baseManifold,\mathbb{R}}$.
\begin{lem}
\label{lem:Trivial_bound_for_GSE}
\[
\eigenvalue_{0}^{\baseManifold}(\magneticFieldStrength)\leq\frac{1}{2}\magneticFieldStrength^{2}\min_{f\in\smoothFunctions{\baseManifold,\mathbb{R}}}\frac{\normComingFromInnerProduct{\magneticPotential-df}^{2}}{\volume{\baseManifold}}=\frac{1}{2}\magneticFieldStrength^{2}\frac{\normComingFromInnerProduct{\magneticPotential_{\mathrm{cc}}}^{2}}{\volume{\baseManifold}}.
\]
\end{lem}
\begin{proof}
For $f\in\smoothFunctions{\baseManifold,\mathbb{R}}$, we get $\eigenvalue_{0}^{\baseManifold}(\magneticFieldStrength)=\groundStateEnergy(\magneticFieldStrength(\magneticPotential-df))$
from~(\ref{eq:GSE_of_Mag_Pot_plus_df}). Thus, we can compare with
the Rayleigh quotient of $\schroedingerOperator_{\magneticFieldStrength(\magneticPotential-df),0}$
at the function $1_{\baseManifold}$ which is identically equal to
$1$ on $\baseManifold$ 
\[
\eigenvalue_{0}^{\baseManifold}(\magneticFieldStrength)\leq\frac{\innerProduct{\schroedingerOperator_{\magneticFieldStrength(\magneticPotential-df),0}1_{\baseManifold}}{1_{\baseManifold}}}{\innerProduct{1_{\baseManifold}}{1_{\baseManifold}}}=\frac{1}{2}\frac{\normComingFromInnerProduct{\magneticDifferential{\magneticFieldStrength(\magneticPotential-df)}1_{\baseManifold}}^{2}}{\normComingFromInnerProduct{1_{\baseManifold}}^{2}}=\frac{1}{2}\magneticFieldStrength^{2}\frac{\int_{\baseManifold}\norm{\magneticPotential-df}^{2}}{\volume{\baseManifold}}.
\]
\vspace{-10mm}

\end{proof}
The following proposition is the analogue of~\cite[Proposition C]{HiguchiShirai2001}
for manifolds. A proof can be obtained from the computations in~\cite[Section 4.4]{Paternain2001}.
\begin{prop}
\label{prop:Derivative_of_GSE_on_quotient}The function $\eigenvalue_{0}^{\baseManifold}$
is real analytic near $\magneticFieldStrength=0$ and satisfies
\[
\eigenvalue_{0}^{\baseManifold}\hspace{0.5mm}'(0)=0\qquad\textnormal{and}\qquad\eigenvalue_{0}^{\baseManifold}\hspace{0.5mm}''(0)=\min_{f\in\smoothFunctions{\baseManifold,\mathbb{R}}}\frac{\normComingFromInnerProduct{\magneticPotential-df}^{2}}{\volume{\baseManifold}}=\frac{\normComingFromInnerProduct{\magneticPotential_{\mathrm{cc}}}^{2}}{\volume{\baseManifold}}.
\]

\end{prop}
Since $\closure{\schroedingerOperator_{\magneticPotential-df,\electricPotential}}$
and $\closure{\schroedingerOperator_{\magneticPotential,\electricPotential}}$
are gauge equivalent and therefore have the same spectrum by virtue
of Theorem~\ref{thm:Gauge_Invariance}, some authors prefer to always
work in the so-called Coulomb gauge given by $\adjoint d\magneticPotential=0$,
or equivalently $\magneticPotential=\magneticPotential_{\mathrm{cc}}$,
see also~\cite{RozenblumMelgaard2005}.

\subsubsection{Twisted operators}

Let $\cover{\pi}\colon\cover{\baseManifold}\to\baseManifold$ be
a regular cover with covering transformation group $\coveringTransformationGroup$,
and let $\cover{\magneticPotential}=\cover{\pi}^{*}\magneticPotential$
and $\cover{\electricPotential}=\cover{\pi}^{*}\electricPotential$
be the lifted potentials. According to Theorem~\ref{thm:Eigenspaces_are_infinite_dimensional},
$\spectrum{\closure{\schroedingerOperator_{\cover{\magneticPotential},\cover{\electricPotential}}}}$
is purely discrete or purely essential depending on whether $\volume{\cover{\baseManifold}}<\infty$
or $\volume{\cover{\baseManifold}}=\infty$. Let $\unitaryRepresentation\colon\coveringTransformationGroup\to U(\hilbertSpace)$
be a unitary representation on some separable Hilbert space $\hilbertSpace$.
We let $E_{\unitaryRepresentation}$ denote the associated flat vector
bundle~\cite{Sunada1989}, that is, $E_{\unitaryRepresentation}$
is the quotient space of $\cover{\baseManifold}\times\hilbertSpace$
by the action of $\coveringTransformationGroup$ given by $\groupElement(x,v)=(\groupElement x,\unitaryRepresentation(\groupElement)v)$.
If $\orbit xv$ denotes the $\coveringTransformationGroup$-orbit
of $(x,v)$, then the mapping $\orbit xv\mapsto\cover{\pi}(x)$ yields
a vector bundle projection $E_{\unitaryRepresentation}\to\baseManifold$,
and $E_{\unitaryRepresentation}$ inherits an inner product given
by
\[
\innerProduct{\orbit xv}{\orbit xw}_{\unitaryRepresentation}=\innerProduct vw_{\hilbertSpace}.
\]
Any section $s$ of $E_{\unitaryRepresentation}$ can be identified
with a $\hilbertSpace$-valued function $u_{s}$ on $\cover{\baseManifold}$
defined via
\[
s\left(\cover{\pi}\left(x\right)\right)=\orbit x{u_{s}(x)}.
\]

In other words, the space $\sectionOfBundle{E_{\unitaryRepresentation}}$
of smooth sections of $E_{\unitaryRepresentation}$ can be identified
with 
\begin{equation}
\left\{ u\colon\cover{\baseManifold}\to\hilbertSpace\textrm{ smooth}\,\big|\, u(\groupElement x)=\unitaryRepresentation(\groupElement)\multipliedBy u(x)\textrm{ for any }\groupElement\in\coveringTransformationGroup,x\in\cover{\baseManifold}\right\} ,\label{eq:Invariant_smooth_functions_on_cover}
\end{equation}
where differentability has to be understood with respect to the norm
topology. Let $L^{2}(E_{\unitaryRepresentation})$ denote the completion
of $\sectionOfBundle{E_{\unitaryRepresentation}}$ with respect to
the norm coming from the inner product 
\[
\innerProduct st=\int_{\baseManifold}\innerProduct st_{\unitaryRepresentation},
\]
and similarly for $L^{2}(E_{\unitaryRepresentation}\otimes T^{*}\baseManifold)$.
The magnetic differential $\magneticDifferential{\magneticPotential}\colon\smoothFunctions{\baseManifold}\to\smoothOneForms{\baseManifold}$
can be extended to
\[
\magneticDifferential{\magneticPotential,\unitaryRepresentation}\colon\sectionOfBundle{E_{\unitaryRepresentation}}\to\sectionOfBundle{E_{\unitaryRepresentation}\otimes T^{*}\baseManifold}.
\]
We let $\adjoint{\magneticDifferential{\magneticPotential,\unitaryRepresentation}}$
denote its formal adjoint and define the twisted magnetic Schr\"odinger
operator as
\[
\schroedingerOperator_{\magneticPotential,\electricPotential,\unitaryRepresentation}=\frac{1}{2}\multipliedBy\adjoint{\magneticDifferential{\magneticPotential,\unitaryRepresentation}}\magneticDifferential{\magneticPotential,\unitaryRepresentation}+\electricPotential\colon\sectionOfBundle{E_{\unitaryRepresentation}}\to\sectionOfBundle{E_{\unitaryRepresentation}}.
\]

The operator $\schroedingerOperator_{\magneticPotential,\electricPotential,\unitaryRepresentation}$
is sometimes called Bochner Laplacian and is known to have a unique
self-adjoint extension to $L^{2}(E_{\unitaryRepresentation})$~\cite{HessSchraderUhlenbrock1980,BravermanMilatovichShubin2002}.
Its ground state energy satisfies
\[
\groundStateEnergy(\magneticPotential,\electricPotential,\unitaryRepresentation)=\inf\multipliedBy\spectrum{\closure{\schroedingerOperator_{\magneticPotential,\electricPotential,\unitaryRepresentation}}}=\inf_{s\in\sectionOfBundle{E_{\unitaryRepresentation}}\backslash\{0\}}\frac{\int_{\baseManifold}\frac{1}{2}\normComingFromInnerProduct{\magneticDifferential{\magneticPotential,\unitaryRepresentation}s}_{\unitaryRepresentation}^{2}+\electricPotential\multipliedBy\normComingFromInnerProduct s_{\unitaryRepresentation}^{2}}{\int_{\baseManifold}\normComingFromInnerProduct s_{\unitaryRepresentation}^{2}}.
\]
If $\unitaryRepresentation$ is one-dimensional, then $\schroedingerOperator_{\magneticPotential,\electricPotential,\unitaryRepresentation}$
can be described as follows. Lemma~\ref{lem:Magnetic_translations_commute_with_MSO}
implies that $\schroedingerOperator_{\cover{\magneticPotential},\cover{\electricPotential}}$
commutes with the action of $\coveringTransformationGroup$. Hence,
$\schroedingerOperator_{\cover{\magneticPotential},\cover{\electricPotential}}$
maps the space~(\ref{eq:Invariant_smooth_functions_on_cover}) to
itself, and $\schroedingerOperator_{\magneticPotential,\electricPotential,\unitaryRepresentation}$
corresponds to the restriction of $\schroedingerOperator_{\cover{\magneticPotential},\cover{\electricPotential}}$
to this space. The following lemma is proven exactly as in~\cite{Sunada1989}.
\begin{lem}
If $\unitaryRepresentation$ is the trivial representation of $\coveringTransformationGroup$,
then $(\closure{\schroedingerOperator_{\magneticPotential,\electricPotential,\unitaryRepresentation}},L^{2}(E_{\unitaryRepresentation}))$
and $(\closure{\schroedingerOperator_{\magneticPotential,\electricPotential}},L^{2}(\baseManifold))$
are unitarily equivalent. Similarly, if $\unitaryRepresentation$
is the right regular representation on $L^{2}(\coveringTransformationGroup),$
then $(\closure{\schroedingerOperator_{\magneticPotential,\electricPotential,\unitaryRepresentation}},L^{2}(E_{\unitaryRepresentation}))$
and $(\closure{\schroedingerOperator_{\cover{\magneticPotential},\cover{\electricPotential}}},L^{2}(\cover{\baseManifold}))$
are unitarily equivalent.\label{lem:Unitary_equivalences_between_twisted_operators_and_usual_ones}
\end{lem}

\subsubsection{Amenable covers}

In the absence of magnetic potentials, one can use twisted Schr\"odinger
operators to show that for any electric potential $\electricPotential\in\smoothFunctions{\baseManifold,\mathbb{R}}$
with lift $\cover{\electricPotential}$ to the regular cover $\cover{\baseManifold}$,
we have 
\[
\groundStateEnergy(0,\electricPotential)\leq\groundStateEnergy(0,\cover{\electricPotential}),
\]

and equality holds precisely if the covering transformation group
of $\cover{\baseManifold}$ is amenable~\cite[Proposition 1]{KobayashiOnoSunada1989}.
The special case $\electricPotential=0$ and $\cover{\baseManifold}=\universalCover{\baseManifold}{}$
is known as Brooks' theorem~\cite{Brooks1981}, which states that
$0\in\spectrum{\closure{\universalCoverWithoutHat{\Delta}{}}}$ if
and only if $\pi_{1}(\baseManifold)$ is amenable. We continue these
developments and prove an analogue of~\cite[Theorem 2.1]{HiguchiShirai1999}
for manifolds.
\begin{thm}
\label{thm:Spectral_Inclusion_for_Amenable_Covers}If $\cover{\baseManifold}$
is a regular cover with amenable covering transformation group $\coveringTransformationGroup$,
then 
\[
\spectrum{\closure{\schroedingerOperator_{\magneticPotential,\electricPotential}}}\subseteq\spectrum{\closure{\schroedingerOperator_{\cover{\magneticPotential},\cover{\electricPotential}}}}.
\]

\end{thm}
Instead of working along the lines of~\cite{Sunada1988,Sunada1989},
our proof follows Brooks' original approach. In essence, \cite[Proposition 2]{Brooks1981}~says
that a regular cover $\cover{\baseManifold}$ has an amenable covering
transformation group $\coveringTransformationGroup$ if and only if
for any $\varepsilon>0$ and any fundamental domain $F$ of the action
of $\coveringTransformationGroup$ on $\cover{\baseManifold}$ arising
from a smooth triangulation of $\baseManifold$, there exist $\groupElement_{1},\groupElement_{2},\ldots,\groupElement_{n}\in\coveringTransformationGroup$
such that the compact subdomain $D=\groupElement_{1}F\cup\groupElement_{2}F\cup\ldots\groupElement_{n}F$
satisfies
\[
\frac{\volume{\partial D}}{\volume D}<\varepsilon.
\]
Following~\cite[Section 2]{Brooks1981}, one considers smooth functions
that are supported inside $D$, and that are non-constant only in
a small neighborhood of $\partial D$ to obtain the following.
\begin{prop}
\label{prop:Foelner_Sequence_for_amenable_covers}If $\cover{\baseManifold}$
is a regular cover of the closed manifold $\baseManifold$ with amenable
covering transformation group $\coveringTransformationGroup$, then
for any $u\in\smoothFunctions{\baseManifold}$ and $\varepsilon>0$
there exists $\chi_{u,\varepsilon}\in\smoothCompactlySupportedFunctions{\cover{\baseManifold},\mathbb{R}}$
with 
\begin{equation}
\normComingFromInnerProduct{\chi_{u,\varepsilon}}=\sqrt{\volume{\baseManifold}}\qquad\text{and}\qquad\max\left(\normComingFromInnerProduct{\Delta\chi_{u,\varepsilon}},\normComingFromInnerProduct{d\chi_{u,\varepsilon}}\right)<\varepsilon\label{eq:Bound1_for_almost_constant_L2_function}
\end{equation}
such that the lift $\cover u$ of $u$ to $\cover{\baseManifold}$
satisfies
\begin{equation}
\normComingFromInnerProduct{\chi_{u,\varepsilon}\cover u}\geq(1-\varepsilon)\normComingFromInnerProduct u.\label{eq:Bound2_for_almost_constant_L2_function}
\end{equation}

\end{prop}
Note that~(\ref{eq:Bound1_for_almost_constant_L2_function}) could
be weakened since
\[
\normComingFromInnerProduct{d\chi_{u,\varepsilon}}^{2}=\innerProduct{\chi_{u,\varepsilon}}{\Delta\chi_{u,\varepsilon}}\leq\normComingFromInnerProduct{\chi_{u,\varepsilon}}\normComingFromInnerProduct{\Delta\chi_{u,\varepsilon}}.
\]

\begin{proof}
[Proof of Theorem  \ref{thm:Spectral_Inclusion_for_Amenable_Covers}]It
suffices to show that for any $\lambda\in\spectrum{\closure{\schroedingerOperator_{\magneticPotential,\electricPotential}}}$,
we have 
\[
\inf_{v\in\smoothCompactlySupportedFunctions{\cover{\baseManifold}}\backslash\{0\}}\frac{\normComingFromInnerProduct{(\schroedingerOperator_{\cover{\magneticPotential},\cover{\electricPotential}}-\lambda)v}}{\normComingFromInnerProduct v}=0.
\]
As $\baseManifold$ is compact and $\closure{\schroedingerOperator_{\magneticPotential,\electricPotential}}$
is elliptic, any $\lambda\in\spectrum{\closure{\schroedingerOperator_{\magneticPotential,\electricPotential}}}$
is an eigenvalue of $\schroedingerOperator_{\magneticPotential,\electricPotential}$
with smooth eigenfunction $u$ satisfying $\schroedingerOperator_{\magneticPotential,\electricPotential}u=\lambda\multipliedBy u$.
For arbitrary $\varepsilon>0$, let $\chi_{u,\varepsilon}\in\smoothCompactlySupportedFunctions{\cover{\baseManifold},\mathbb{R}}$
be as in Proposition~\ref{prop:Foelner_Sequence_for_amenable_covers},
and define $v_{\varepsilon}=\chi_{u,\varepsilon}\multipliedBy\cover u\in\smoothCompactlySupportedFunctions{\cover{\baseManifold}}$,
where $\cover u$ denotes the lift of $u$ to $\cover{\baseManifold}$.
Note that $\schroedingerOperator_{\cover{\magneticPotential},\cover{\electricPotential}}\cover u=\lambda\multipliedBy\cover u$.
Since $\Delta v_{\varepsilon}=\chi_{u,\varepsilon}\Delta\cover u-2\innerProduct{d\chi_{u,\varepsilon}}{d\cover u}+\cover u\multipliedBy\Delta\chi_{u,\varepsilon}$,
the local expression~(\ref{eq:Coordinatefree_Version_of_MSO}) for
$\schroedingerOperator_{\cover{\magneticPotential},\cover{\electricPotential}}$
leads to
\[
(\schroedingerOperator_{\cover{\magneticPotential},\cover{\electricPotential}}-\lambda)v_{\varepsilon}=\frac{1}{2}\multipliedBy\cover u\multipliedBy\Delta\chi_{u,\varepsilon}-\innerProduct{d\chi_{u,\varepsilon}}{d\cover u}-i\innerProduct{\cover u\multipliedBy\exteriorDifferential\chi_{u,\varepsilon}}{\cover{\magneticPotential}}.
\]
A similar computation appears in the proof of~\cite[Proposition 3.2]{Shigekawa1987}.
Since $\cover u$ and $\cover{\magneticPotential}$ arise from $u$
and $\magneticPotential$, the claim follows from (\ref{eq:Bound1_for_almost_constant_L2_function})
and (\ref{eq:Bound2_for_almost_constant_L2_function}) via the estimate
\[
\frac{\normComingFromInnerProduct{(\schroedingerOperator_{\cover{\magneticPotential},\cover{\electricPotential}}-\lambda)v_{\varepsilon}}}{\normComingFromInnerProduct{v_{\varepsilon}}}\leq\frac{\varepsilon}{(1-\varepsilon)}\multipliedBy\frac{\frac{1}{2}\normComingFromInnerProduct u_{\infty}+\normComingFromInnerProduct{du}_{\infty}+\normComingFromInnerProduct u_{\infty}\normComingFromInnerProduct{\magneticPotential}_{\infty}}{\normComingFromInnerProduct u}.
\]

\end{proof}

\subsubsection{Abelian covers}

We discuss Bloch-Floquet theory for lifted magnetic Schr\"o\-dinger
operators on abelian covers and extend results that are known for
graphs~\cite{HiguchiShirai1999} to manifolds. Our approach bases
upon the study of lifted Laplacians on abelian covers in~\cite{KotaniSunada2000}
and~\cite[Section 3]{Post2000}, the latter of which gives an extensive
account of the underlying functional analysis.

Let $\cover{\pi}\colon\cover{\baseManifold}\to\baseManifold$ be a
regular cover with abelian covering transformation group $\coveringTransformationGroup$.
Note that $\coveringTransformationGroup\simeq\mathbb{Z}^{r_{0}}\times\mathbb{Z}_{p_{1}}^{r_{1}}\times\ldots\times\mathbb{Z}_{p_{k}}^{r_{k}}$
for some $r_{0},r_{1},\ldots,r_{k}\in\mathbb{N}_{0}$, where $\mathbb{Z}_{p}$
denotes the cyclic group of order $p$. The irreducible unitary representations
of $\coveringTransformationGroup$ are one-dimensional and constitute
the so-called character group $\characterGroup{\coveringTransformationGroup}=\homomorphismGroup(\coveringTransformationGroup,\mathbb{S}^{1})\simeq\mathbb{T}^{r_{0}}\times\mathbb{Z}_{p_{1}}^{r_{1}}\times\ldots\times\mathbb{Z}_{p_{k}}^{r_{k}}$,
which is compact with respect to its canonical topology of pointwise
convergence. Recall that $\cover{\baseManifold}$ is covered by $\abelianCover{\baseManifold}{}$,
which entails a surjective homomorphism $\Phi\colon\homologyGroup_{1}(\baseManifold,\mathbb{Z})\twoheadrightarrow\coveringTransformationGroup$.
Dualizing yields the following injection of compact character groups
\begin{equation}
\characterGroup{\Phi}\colon\characterGroup{\coveringTransformationGroup}\hookrightarrow\characterGroup{\homologyGroup}_{1}(\baseManifold,\mathbb{Z})=\homomorphismGroup(\homologyGroup_{1}(\baseManifold,\mathbb{Z}),\mathbb{S}^{1}).\label{eq:Injection_of_Character_Groups}
\end{equation}
The latter can be identified with the so-called Jacobian torus $\homologyGroup^{1}(\baseManifold,\mathbb{R})/2\pi\multipliedBy\homologyGroup^{1}(\baseManifold,\mathbb{Z})$
via the mapping~\cite[Section 1]{KatsudaSunada1988} 
\begin{equation}
\homologyGroup^{1}(\baseManifold,\mathbb{R})\ni[\omega]\mapsto\character_{[\omega]}\in\characterGroup{\homologyGroup}_{1}(\baseManifold,\mathbb{Z})\qquad\textnormal{given by}\qquad\character_{[\omega]}([\groupElement])=e^{i\int_{\groupElement}\omega},\label{eq:Character_Group_as_Jacobian_Torus}
\end{equation}
where $[\groupElement]$ denotes the homology class of the closed
curve $\groupElement$. Note that if $\groupElement'$ is another
closed curve such that $[\groupElement']=[\groupElement]$, then $\gamma'\circ\gamma^{-1}$
has homotopy class in $[\pi_{1}(\baseManifold),\pi_{1}(\baseManifold)]$
leading to $\int_{\gamma'\circ\gamma^{-1}}\omega=0$. Moreover, $[\omega]\in\homologyGroup^{1}(\baseManifold,\mathbb{R})$
satisfies $\character_{[\omega]}\in\characterGroup{\Phi}(\characterGroup{\coveringTransformationGroup})$
if and only if $\int_{\gamma}\omega\in2\pi\multipliedBy\mathbb{Z}$
for any closed curve $\gamma$ in $\baseManifold$ with $[\gamma]\in\ker\Phi\subseteq\homologyGroup_{1}(\baseManifold,\mathbb{Z})$,
that is, if and only if $[\omega]\in2\pi\multipliedBy\homologyGroup^{1}(\baseManifold,\cover{\baseManifold},\mathbb{Z})$,
where we defined
\begin{equation}
\homologyGroup^{1}(\baseManifold,\cover{\baseManifold},\mathbb{Z})=\left\{ [\omega]\in\homologyGroup^{1}(\baseManifold,\mathbb{R})\,\Big|\,\int_{\gamma}\cover{\omega}\in\mathbb{Z}\textrm{ for any closed curve }\gamma\textrm{ in }\cover{\baseManifold}\right\} .\label{eq:Hom_CTG_to_S_as_One_forms}
\end{equation}

The mapping~(\ref{eq:Character_Group_as_Jacobian_Torus}) also identifies
the tangent space $T_{\boldsymbol{1}}\characterGroup{\coveringTransformationGroup}$
at the trivial representation $\boldsymbol{1}\in\characterGroup{\coveringTransformationGroup}$
with $\homomorphismGroup(\coveringTransformationGroup,\mathbb{R})=\{[\omega]\in\homologyGroup^{1}(\baseManifold,\mathbb{R})\,|\,\cover{\omega}\textrm{ is exact}\}$
as given in~(\ref{eq:Hom_CTG_to_R_as_One_Forms}), see~\cite{Sunada1992a}
and~\cite[Section 2]{KotaniSunada2000}. Essentially the same ideas
work in the case of graphs~\cite[Section 3]{HiguchiShirai1999}.
In the following, we use direct integral decompositions as in~\cite[Section XIII.16]{ReedSimon1978}
and extend~\cite[Proposition 2]{KobayashiOnoSunada1989}.
\begin{thm}
\label{thm:MSO_on_abelian_cover_as_direct_integral}If $\cover{\magneticPotential}$
and $\cover{\electricPotential}$ are lifts of potentials $\magneticPotential\in\smoothOneForms{\baseManifold,\mathbb{R}}$
and $\electricPotential\in\smoothFunctions{\baseManifold,\mathbb{R}}$
to $\cover{\baseManifold}$, then $(\closure{\schroedingerOperator_{\cover{\magneticPotential},\cover{\electricPotential}}},L^{2}(\cover{\baseManifold}))$
allows for a direct integral decomposition, namely, 
\begin{equation}
L^{2}(\cover{\baseManifold})=\int_{\characterGroup{\coveringTransformationGroup}}^{\oplus}L^{2}(E_{\character})\multipliedBy\exteriorDifferential\character\qquad\textrm{and}\qquad\closure{\schroedingerOperator_{\cover{\magneticPotential},\cover{\electricPotential}}}=\int_{\characterGroup{\coveringTransformationGroup}}^{\oplus}\closure{\schroedingerOperator_{\magneticPotential,\electricPotential,\character}}\multipliedBy\exteriorDifferential\character,\label{eq:MSO_on_abelian_cover_as_direct_integral}
\end{equation}
where $E_{\character}$ denotes the flat vector bundle associated
with $\character\in\characterGroup{\coveringTransformationGroup}$,
and $\exteriorDifferential\character$ denotes the normalized Haar
measure on $\characterGroup{\coveringTransformationGroup}$. Moreover,
if $\character\in\characterGroup{\coveringTransformationGroup}$ is
represented by $\omega_{\character}\in\smoothOneForms{\baseManifold,\mathbb{R}}$
via (\ref{eq:Injection_of_Character_Groups}) and (\ref{eq:Character_Group_as_Jacobian_Torus}),
then $(\closure{\schroedingerOperator_{\magneticPotential,\electricPotential,\character}},L^{2}(E_{\character}))$
is unitarily equivalent to $(\closure{\schroedingerOperator_{\magneticPotential+\omega_{\character},\electricPotential}},L^{2}(\baseManifold))$.\end{thm}
\begin{proof}
For each $\character\in\characterGroup{\coveringTransformationGroup}$,
we use~(\ref{eq:Invariant_smooth_functions_on_cover}) to identify
$\sectionOfBundle{E_{\character}}$ with the space of smooth $\character$-periodic
functions on $\cover{\baseManifold}$, on which $\closure{\schroedingerOperator_{\magneticPotential,\electricPotential,\character}}$
acts as $\closure{\schroedingerOperator_{\cover{\magneticPotential},\cover{\electricPotential}}}$.
With respect to this identification, we have $\normComingFromInnerProduct u^{2}=\int_{F}\norm u^{2}\volumeForm$
for $u\in\sectionOfBundle{E_{\character}}\subset\smoothFunctions{\cover{\baseManifold}}$,
where $F$ is a fundamental domain for the action of $\coveringTransformationGroup$
on $\cover{\baseManifold}$. Following~\cite[Theorem 3.3]{Donnelly1981},
\cite[Proposition 2]{KobayashiOnoSunada1989} and~\cite[Section 3.3]{Post2000},
we verify that the Gelfand transform
\[
\mathcal{I}\colon\smoothCompactlySupportedFunctions{\cover{\baseManifold}}\to\int_{\characterGroup{\coveringTransformationGroup}}^{\oplus}L^{2}(E_{\character})\multipliedBy\exteriorDifferential\character\qquad\textnormal{given by}\qquad(\mathcal{\mathcal{I}}u)_{\character}(x)=\sum_{\groupElement\in\coveringTransformationGroup}\conjugate{\character(\groupElement)}\multipliedBy u(\groupElement\multipliedBy x)
\]
extends to the desired isometry. For $\groupElement\in\coveringTransformationGroup$,
we have 
\[
(\mathcal{\mathcal{I}}u)_{\character}(\groupElement\multipliedBy x)=\sum_{\groupElement'\in\coveringTransformationGroup}\conjugate{\character(\groupElement')}\multipliedBy u(\groupElement'\groupElement\multipliedBy x)=\sum_{\groupElement''\in\coveringTransformationGroup}\conjugate{\character(\groupElement'')}\multipliedBy\conjugate{\character(\groupElement^{-1})}\multipliedBy u(\groupElement''x)=\character(\groupElement)\multipliedBy(\mathcal{\mathcal{I}}u)_{\character}(x).
\]
In other words, $(\mathcal{\mathcal{I}}u)_{\character}\in\sectionOfBundle{E_{\character}}$.
Using the orthogonality of characters, we obtain 
\begin{eqnarray*}
\normComingFromInnerProduct{\mathcal{\mathcal{I}}u}^{2}=\int_{\characterGroup{\coveringTransformationGroup}}\normComingFromInnerProduct{(\mathcal{\mathcal{I}}u)_{\character}}^{2}\exteriorDifferential\character & = & \int_{\characterGroup{\coveringTransformationGroup}}\int_{F}\left(\sum_{\groupElement\in\coveringTransformationGroup}\conjugate{\character(\groupElement)}\multipliedBy\groupElement^{*}u\right)\left(\sum_{\groupElement'\in\coveringTransformationGroup}\character(\groupElement')\multipliedBy\conjugate{\groupElement'^{*}u}\right)\volumeForm\multipliedBy\exteriorDifferential\character\\
 & = & \int_{F}\sum_{\groupElement,\groupElement'\in\coveringTransformationGroup}\groupElement^{*}u\multipliedBy\conjugate{\groupElement'^{*}u}\multipliedBy\int_{\characterGroup{\coveringTransformationGroup}}\conjugate{\character(\groupElement)}\multipliedBy\character(\groupElement')\multipliedBy\exteriorDifferential\character\multipliedBy\volumeForm\\
 & = & \int_{F}\sum_{\groupElement\in\coveringTransformationGroup}\norm{\groupElement^{*}u}^{2}\volumeForm=\int_{\cover{\baseManifold}}\norm u^{2}\volumeForm=\normComingFromInnerProduct u_{L^{2}(\cover{\baseManifold})}^{2}.
\end{eqnarray*}
The inverse of $\mathcal{I}$ reads
\[
(\mathcal{I}^{-1}(u_{\character})_{\character\in\characterGroup{\coveringTransformationGroup}})(x)=\int_{\characterGroup{\coveringTransformationGroup}}u_{\character}(x)\multipliedBy\exteriorDifferential\character,
\]
and we have $\mathcal{I}\multipliedBy\schroedingerOperator_{\cover{\magneticPotential},\cover{\electricPotential}}\multipliedBy\mathcal{I}^{-1}(u_{\character})_{\character\in\characterGroup{\coveringTransformationGroup}}=(\schroedingerOperator_{\magneticPotential,\electricPotential,\character}u_{\character})_{\character\in\characterGroup{\coveringTransformationGroup}}$
as claimed. In the following, let $\omega_{\character}\in\smoothOneForms{\baseManifold,\mathbb{R}}$
represent $\character\in\characterGroup{\coveringTransformationGroup}$,
in particular, its lift $\cover{\omega_{\character}}$ satisfies $\int_{\gamma}\cover{\omega_{\character}}\in2\pi\multipliedBy\mathbb{Z}$
for any closed curve $\gamma$ in $\cover{\baseManifold}$. Following~\cite[Proposition 4]{Sunada1985},
we choose some fixed reference point $x_{0}\in\cover{\baseManifold}$
and define $\cover u_{\character}\colon\cover{\baseManifold}\to\mathbb{S}^{1}$
as 
\[
\cover u_{\character}(x)=e^{i\int_{x_{0}}^{x}\cover{\omega_{\character}}}.
\]
Note that $\cover u_{\character}$ is well-defined, satisfies $\exteriorDifferential\cover u_{\character}=i\multipliedBy\cover u_{\character}\cover{\omega_{\character}}$,
and is $\character$-periodic since for $\groupElement\in\coveringTransformationGroup$,
we have 
\[
\cover u_{\character}(\groupElement\multipliedBy x)=e^{i\int_{x}^{\groupElement x}\cover{\omega_{\character}}}\multipliedBy e^{i\int_{x_{0}}^{x}\cover{\omega_{\character}}}=\character(\groupElement)\multipliedBy\cover u_{\character}(x).
\]
Using that $\cover u_{\character}\in\sectionOfBundle{E_{\character}}$,
we define unitary maps $\mathcal{U}_{\character}\colon\smoothFunctions{\baseManifold}\to\sectionOfBundle{E_{\character}}$
and $\mathcal{W}_{\character}\colon\smoothOneForms{\baseManifold}\to\sectionOfBundle{E_{\character}\otimes T^{*}\baseManifold}$
as $\mathcal{U}_{\character}u=\cover u_{\character}\cover u$ and
$\mathcal{W}_{\character}\omega=\cover u_{\character}\cover{\omega}$,
where $\cover u$ and $\cover{\omega}$ denote the lifts of $u$ and
$\omega$ to $\cover{\baseManifold}$, respectively. For $u\in\smoothFunctions{\baseManifold}$,
we have
\begin{eqnarray*}
\magneticDifferential{\magneticPotential,\character}\multipliedBy\mathcal{U}_{\character}u=\magneticDifferential{\cover{\magneticPotential}}(\cover u_{\character}\cover u) & = & -i\multipliedBy\cover u_{\character}\exteriorDifferential\cover u-i\multipliedBy\cover u\multipliedBy\exteriorDifferential\cover u_{\character}+\cover u_{\character}\cover u\multipliedBy\cover{\magneticPotential}\\
 & = & -i\multipliedBy\cover u_{\character}\exteriorDifferential\cover u+\cover u_{\character}\cover u\multipliedBy(\cover{\omega_{\character}}+\cover{\magneticPotential})=\cover u_{\character}\magneticDifferential{\cover{\magneticPotential}+\cover{\omega_{\character}}}\cover u=\mathcal{W}_{\character}\multipliedBy\magneticDifferential{\magneticPotential+\omega_{\character}}u.
\end{eqnarray*}
We take formal adjoints to obtain 
\[
\mathcal{U}_{\character}^{-1}\multipliedBy\adjoint{\magneticDifferential{\magneticPotential,\character}}=\adjoint{\mathcal{U}_{\character}}\multipliedBy\adjoint{\magneticDifferential{\magneticPotential,\character}}=\adjoint{\magneticDifferential{\magneticPotential+\omega_{\character}}}\multipliedBy\adjoint{\mathcal{W}_{\character}}=\adjoint{\magneticDifferential{\magneticPotential+\omega_{\character}}}\multipliedBy\mathcal{W}_{\character}^{-1}.
\]
Since $\mathcal{U}_{\character}^{-1}\multipliedBy\cover{\electricPotential}\multipliedBy\mathcal{U}_{\character}=\electricPotential$
trivially holds on $\smoothFunctions{\baseManifold}$, the analogue
of~\cite[Lemma 3.1]{HiguchiShirai1999} for manifolds follows, namely,
\[
\mathcal{U}_{\character}^{-1}\multipliedBy\closure{\schroedingerOperator_{\magneticPotential,\electricPotential,\character}}\multipliedBy\mathcal{U}_{\character}=\closure{\schroedingerOperator_{\magneticPotential+\omega_{\character},\electricPotential}}.
\]
\end{proof}
\begin{thm}
\label{thm:GSE_of_abelian_covers}The spectrum of $\closure{\schroedingerOperator_{\cover{\magneticPotential},\cover{\electricPotential}}}$
as in Theorem~\ref{thm:MSO_on_abelian_cover_as_direct_integral}
has band structure
\begin{equation}
\spectrum{\closure{\schroedingerOperator_{\cover{\magneticPotential},\cover{\electricPotential}}}}=\bigcup_{[\omega]\in2\pi\multipliedBy\homologyGroup^{1}(\baseManifold,\cover{\baseManifold},\mathbb{Z})}\spectrum{\closure{\schroedingerOperator_{\magneticPotential-\omega,\electricPotential}}}=\bigcup_{k\in I}[a_{k},b_{k}],\label{eq:Band_structure_of_abelian_covers}
\end{equation}
where $\homologyGroup^{1}(\baseManifold,\cover{\baseManifold},\mathbb{Z})$
is given in~(\ref{eq:Hom_CTG_to_S_as_One_forms})\textup{\emph{,
and we either have $I=\{1,2,\ldots,N\}$ and $b_{k}<a_{k+1}$ for
$k\in I\backslash\{N\}$ as well as $b_{N}=\infty$, or $I=\mathbb{N}$
and $b_{k}<a_{k+1}$ for all $k\in I$ as well as $a_{k}\nearrow\infty$.
Moreover,}} 
\[
\groundStateEnergy(\cover{\magneticPotential},\cover{\electricPotential})=\min_{[\omega]\in2\pi\multipliedBy\homologyGroup^{1}(\baseManifold,\cover{\baseManifold},\mathbb{Z})}\groundStateEnergy(\magneticPotential-\omega,\electricPotential).
\]
\textup{\emph{In particular, }}$\cover{\magneticPotential}$ has no
diamagnetic effect, that is, $\groundStateEnergy(\cover{\magneticPotential},\cover{\electricPotential})=\groundStateEnergy(0,\cover{\electricPotential})$,
if and only if $\magneticPotential$ is closed and $[\magneticPotential]\in2\pi\multipliedBy\homologyGroup^{1}(\baseManifold,\cover{\baseManifold},\mathbb{Z})$.\end{thm}
\begin{proof}
For arbitrary $\omega\in\smoothOneForms{\baseManifold,\mathbb{R}}$,
the spectrum of $\closure{\schroedingerOperator_{\magneticPotential+\omega,\electricPotential}}$
is discrete with smooth eigenfunctions~\cite[Theorem 2.1]{Shigekawa1987}.
Thus, the spectra of the twisted operators $\closure{\schroedingerOperator_{\magneticPotential,\electricPotential,\character}}$
in the direct integral decomposition~(\ref{eq:MSO_on_abelian_cover_as_direct_integral})
are given as unbounded sequences 
\[
\groundStateEnergy(\character)\leq\eigenvalue_{1}(\character)\leq\eigenvalue_{2}(\character)\leq\ldots.
\]
Pertubation theory as in~\cite[Theorem 2]{KodairaSpencer1960} shows
that these eigenvalue functions $\eigenvalue_{k}\colon\characterGroup{\coveringTransformationGroup}\to\mathbb{R}$
are continuous, see also~\cite{KatsudaSunada1988,KobayashiOnoSunada1989,KotaniSunada2000}.
The claim~(\ref{eq:Band_structure_of_abelian_covers}) now follows
from the theory of direct integrals and compactness of $\characterGroup{\coveringTransformationGroup}$,
namely,
\[
\spectrum{\closure{\schroedingerOperator_{\cover{\magneticPotential},\cover{\electricPotential}}}}=\bigcup_{k\in\mathbb{N}_{0}}\bigcup_{\character\in\characterGroup{\coveringTransformationGroup}}\eigenvalue_{k}(\character)=\bigcup_{k\in\mathbb{N}_{0}}\eigenvalue_{k}(\characterGroup{\coveringTransformationGroup}).
\]
The remaining statements are consequences of Theorem~\ref{thm:Gauge_Invariance}
and the fact that each $\character\in\characterGroup{\coveringTransformationGroup}$
is represented by some $[\omega]\in2\pi\multipliedBy\homologyGroup^{1}(\baseManifold,\cover{\baseManifold},\mathbb{Z})$.
\end{proof}
As an example, take the $n$-fold covering of $\baseManifold=\mathbb{S}^{1}=\{z\in\mathbb{C}\,|\,\norm z=1\}$
by $\cover{\baseManifold}=\mathbb{S}^{1}$ with projection $\cover{\pi}(z)=z^{n}$.
If $n>1$ and $\varphi\in(0,2\pi)$ denotes the angular coordinate
of $z=e^{i\varphi}\in\baseManifold\backslash\{1\}$, then $\frac{1}{n}\multipliedBy d\varphi$
extends to a closed $1$-form $\magneticPotential$ on $\baseManifold$
satisfying $[\magneticPotential]\in2\pi\multipliedBy\homologyGroup^{1}(\baseManifold,\cover{\baseManifold},\mathbb{Z})\backslash2\pi\multipliedBy\homologyGroup^{1}(\cover{\baseManifold},\mathbb{Z})$.
For arbitrary $\electricPotential\in\smoothFunctions{\baseManifold,\mathbb{R}}$,
Theorem~\ref{thm:Gauge_Invariance} and Theorem~\ref{thm:GSE_of_abelian_covers}
thus yield 
\[
\groundStateEnergy(\cover{\magneticPotential},\cover{\electricPotential})=\groundStateEnergy(0,\electricPotential)<\groundStateEnergy(\magneticPotential,\electricPotential),
\]
which contrasts the non-magnetic case in which $\groundStateEnergy(0,\cover{\electricPotential})\geq\groundStateEnergy(0,\electricPotential)$
for any cover $\cover{\baseManifold}$~\cite[Corollary of Proposition 1]{KobayashiOnoSunada1989}.
We present another example of this type in Section~\ref{sub:HG}.
The following corollary of Theorem~\ref{thm:Spectral_Inclusion_for_Amenable_Covers}
and Theorem~\ref{thm:GSE_of_abelian_covers} should be compared with
Theorem~\ref{thm:MCV_of_Amenable_Covers_and_Abelian_Subcovers_coincide}.
\begin{cor}
\label{cor:GSE_of_amenable_covers}If $\coverTwo{\baseManifold}$
is a regular amenable cover of $\baseManifold$ with abelian subcover
$\cover{\baseManifold}$, and if $(\coverTwo{\magneticPotential},\coverTwo{\electricPotential})$
and $(\cover{\magneticPotential},\cover{\electricPotential})$ denote
the lifts of $\magneticPotential\in\smoothOneForms{\baseManifold,\mathbb{R}}$
and $\electricPotential\in\smoothFunctions{\baseManifold,\mathbb{R}}$
to $\coverTwo{\baseManifold}$ and $\cover{\baseManifold}$, respectively,
then 
\[
\spectrum{\closure{\schroedingerOperator_{\magneticPotential,\electricPotential}}}\subseteq\spectrum{\closure{\schroedingerOperator_{\cover{\magneticPotential},\cover{\electricPotential}}}}=\bigcup_{[\omega]\in2\pi\multipliedBy\homologyGroup^{1}(\baseManifold,\cover{\baseManifold},\mathbb{Z})}\spectrum{\closure{\schroedingerOperator_{\magneticPotential-\omega,\electricPotential}}}\subseteq\spectrum{\closure{\schroedingerOperator_{\coverTwo{\magneticPotential},\coverTwo{\electricPotential}}}}.
\]
In particular, if $\pi_{1}(\baseManifold)$ is amenable, then 
\[
\groundStateEnergy(\magneticPotential,\electricPotential)\geq\groundStateEnergy(\abelianCover{\magneticPotential}{},\abelianCover{\electricPotential}{})=\min_{[\omega]\in\homologyGroup^{1}(\baseManifold,\mathbb{R})}\groundStateEnergy(\magneticPotential-\omega,\electricPotential)\geq\groundStateEnergy(\universalCover{\magneticPotential}{},\universalCover{\electricPotential}{}).
\]

\end{cor}
Henceforth, let $\electricPotential=0$. Using Hodge theory, we identify
$\homologyGroup^{1}(\baseManifold,\mathbb{R})$ with the space of
harmonic $1$-forms on $\baseManifold$. In particular, let $\omega_{1},\omega_{2},\ldots,\omega_{r_{0}}$
be harmonic $1$-forms that represent an orthonormal basis of $T_{\boldsymbol{1}}\characterGroup{\coveringTransformationGroup}\simeq\homomorphismGroup(\coveringTransformationGroup,\mathbb{R})\subseteq\homologyGroup^{1}(\baseManifold,\mathbb{R})$.
Let $\magneticPotential\in\smoothOneForms{\baseManifold,\mathbb{R}}$
have Hodge decomposition 
\[
\magneticPotential=\magneticPotential_{\mathrm{h}}^{\wedge}+\magneticPotential_{\mathrm{h}}^{\perp}+\adjoint d\beta_{\magneticPotential}+df_{\magneticPotential},
\]
where $\beta_{\magneticPotential}\in\smoothTwoForms{\baseManifold,\mathbb{R}}$,
$f_{\magneticPotential}\in\smoothFunctions{\baseManifold,\mathbb{R}}$,
and $\magneticPotential_{\mathrm{h}}^{\wedge}$ as well as $\magneticPotential_{\mathrm{h}}^{\perp}$
are harmonic such that $\magneticPotential_{\mathrm{h}}^{\wedge}\in T_{\boldsymbol{1}}\characterGroup{\coveringTransformationGroup}$
and $\magneticPotential_{\mathrm{h}}^{\perp}\in T_{\boldsymbol{1}}\characterGroup{\coveringTransformationGroup}^{\perp}$.
If $\magneticPotential_{\mathrm{h}}^{\perp}+\adjoint d\beta_{\magneticPotential}=0$,
then $[\magneticFieldStrength\multipliedBy\magneticPotential]=[\magneticFieldStrength\multipliedBy\magneticPotential_{\mathrm{h}}^{\wedge}]\in2\pi\multipliedBy\homologyGroup^{1}(\baseManifold,\cover{\baseManifold},\mathbb{Z})$
for any $\magneticFieldStrength\in\mathbb{R}$, in particular, $\groundStateEnergy(\magneticFieldStrength\multipliedBy\cover{\magneticPotential})=\groundStateEnergy(0)$.
Thus, we assume that $\magneticPotential_{\mathrm{h}}^{\perp}+\adjoint d\beta_{\magneticPotential}\neq0$
and consider the finite-dimensional vector space 
\[
X_{\magneticPotential}=\mathbb{R}(\magneticPotential_{\mathrm{h}}^{\perp}+\adjoint d\beta_{\magneticPotential})\oplus\bigoplus_{j=1}^{r_{0}}\mathbb{R}\multipliedBy\omega_{j}\subseteq\smoothOneForms{\baseManifold,\mathbb{R}}.
\]
The following generalization of Proposition~\ref{prop:Derivative_of_GSE_on_quotient}
to abelian covers is an extended version of~\cite[Theorem 1.2]{HiguchiShirai1999}
for manifolds.
\begin{thm}
\label{thm:Derivative_of_GSE_of_abelian_covers}Let $\eigenvalue_{0,X_{\magneticPotential}}\colon X_{\magneticPotential}\to\mathbb{R}$
be given by $\eigenvalue_{0,X_{\magneticPotential}}(\omega)=\groundStateEnergy(\omega)$.
In a neighborhood of $0\in X_{\magneticPotential}$, the function
\textup{$\eigenvalue_{0,X_{\magneticPotential}}$} is real analytic
and has positive definite Hessian. In particular, $\eigenvalue_{0}^{\cover{\baseManifold}}\colon\mathbb{R}\to\mathbb{R}$
given by $\eigenvalue_{0}^{\cover{\baseManifold}}(\magneticFieldStrength)=\groundStateEnergy(\magneticFieldStrength\multipliedBy\cover{\magneticPotential})$
is real analytic near $\magneticFieldStrength=0$ and satisfies
\[
\eigenvalue_{0}^{\cover{\baseManifold}}\hspace{0.5mm}''(0)=\min_{[\omega]\in\homomorphismGroup(\coveringTransformationGroup,\mathbb{R})}\frac{\normComingFromInnerProduct{\magneticPotential-\omega}^{2}}{\volume{\baseManifold}}=\frac{\normComingFromInnerProduct{\magneticPotential_{\mathrm{h}}^{\perp}}^{2}+\normComingFromInnerProduct{\adjoint d\beta_{\magneticPotential}}^{2}}{\volume{\baseManifold}},
\]
where $\homomorphismGroup(\coveringTransformationGroup,\mathbb{R})$
denotes the vector space\textup{~(\ref{eq:Hom_CTG_to_R_as_One_Forms})}.\end{thm}
\begin{proof}
The claimed real analyticity of $\eigenvalue_{0,X_{\magneticPotential}}$
follows from pertubation theory as $\eigenvalue_{0,X_{\magneticPotential}}(0)=0$
is a simple eigenvalue of the Laplacian $2\multipliedBy\closure{\schroedingerOperator_{0,0}}$.
Moreover, Proposition~\ref{prop:Derivative_of_GSE_on_quotient} says
that 
\[
\mathrm{Hess}\multipliedBy\eigenvalue_{0,X_{\magneticPotential}}(0)[\omega,\omega]=\frac{\normComingFromInnerProduct{\omega}^{2}}{\volume{\baseManifold}},
\]
which implies positive definiteness near $0\in X_{\magneticPotential}$.
The statements concerning $\eigenvalue_{0}^{\cover{\baseManifold}}$
follow from~\cite[Lemma 4.2]{HiguchiShirai1999} as in the proof
of~\cite[Theorem 4.3]{HiguchiShirai1999}. More precisely, we can
assume that $f_{\magneticPotential}=0$ using~(\ref{eq:GSE_of_Mag_Pot_plus_df}).
The proof of Theorem~\ref{thm:GSE_of_abelian_covers} revealed that
\[
\eigenvalue_{0}^{\cover{\baseManifold}}(\magneticFieldStrength)=\min_{\character\in\characterGroup{\coveringTransformationGroup}}\groundStateEnergy(\magneticFieldStrength\multipliedBy\magneticPotential+\omega_{\character}),
\]
where the closed form $\omega_{\character}$ gives rise to $\character\in\characterGroup{\coveringTransformationGroup}$
via~(\ref{eq:Injection_of_Character_Groups}) and (\ref{eq:Character_Group_as_Jacobian_Torus}).
We verify that for each open neighborhood $U$ of $\boldsymbol{1}\in\characterGroup{\coveringTransformationGroup}$,
there exists $\varepsilon_{U}>0$ such that $\norm{\magneticFieldStrength}<\varepsilon_{U}$
implies
\[
\eigenvalue_{0}^{\cover{\baseManifold}}(\magneticFieldStrength)=\min_{\character\in U}\groundStateEnergy(\magneticFieldStrength\multipliedBy\magneticPotential+\omega_{\character}).
\]
Otherwise, we can find sequences $(\magneticFieldStrength_{j})_{j\in\mathbb{N}}\searrow0$
and $(\character_{j})_{j\in\mathbb{N}}\subset\characterGroup{\coveringTransformationGroup}\backslash U$
with 
\[
\eigenvalue_{0}^{\cover{\baseManifold}}(\magneticFieldStrength_{j})=\groundStateEnergy(\magneticFieldStrength_{j}\magneticPotential+\omega_{\character_{j}})\leq\groundStateEnergy(\magneticFieldStrength_{j}\multipliedBy\magneticPotential).
\]
Note that $\groundStateEnergy(\magneticFieldStrength_{j}\multipliedBy\magneticPotential)\to0$
by virtue of Proposition~\ref{prop:Derivative_of_GSE_on_quotient}.
Since $\characterGroup{\coveringTransformationGroup}\backslash U$
is compact, a subsequence of $(\character_{j})_{j\in\mathbb{N}}$
converges to some $\character_{0}\in\characterGroup{\coveringTransformationGroup}\backslash U$
with $\groundStateEnergy(\omega_{\character_{0}})=0$. However, Theorem~\ref{thm:Gauge_Invariance}
thus implies that $[\omega_{\character_{0}}]\in2\pi\multipliedBy\homologyGroup^{1}(\baseManifold,\mathbb{Z})$
which contradicts $\character_{0}\neq\boldsymbol{1}$. Since $\characterGroup{\coveringTransformationGroup}\simeq\mathbb{T}^{r_{0}}\times\mathbb{Z}_{p_{1}}^{r_{1}}\times\ldots\times\mathbb{Z}_{p_{k}}^{r_{k}}$,
we may regard any open neighborhood $W$ of $0\in\bigoplus_{j=1}^{r_{0}}\mathbb{R}\multipliedBy\omega_{j}\simeq T_{\boldsymbol{1}}\characterGroup{\coveringTransformationGroup}$
as an open neighborhood of $\boldsymbol{1}\in\characterGroup{\coveringTransformationGroup}$,
and find $\varepsilon_{W}>0$ such that $\norm{\magneticFieldStrength}<\varepsilon_{W}$
implies 
\[
\eigenvalue_{0}^{\cover{\baseManifold}}(\magneticFieldStrength)=\min_{\omega\in W}\eigenvalue_{0,X_{\magneticPotential}}(\magneticFieldStrength\multipliedBy\magneticPotential+\omega).
\]
We imitate the proof of~\cite[Theorem 4.3]{HiguchiShirai1999} and
apply~\cite[Lemma 4.2]{HiguchiShirai1999} with $e_{0}=\magneticPotential$,
$e_{j}=\omega_{j}$ for $j=1,2,\ldots,r_{0}$, and $f=\eigenvalue_{0,X_{\magneticPotential}}$
for sufficiently small $\norm{\magneticFieldStrength}$ to obtain
real analyticity of $\eigenvalue_{0}^{\cover{\baseManifold}}$ near
$\magneticFieldStrength=0$ as well as
\begin{eqnarray*}
\eigenvalue_{0}^{\cover{\baseManifold}}\hspace{0.5mm}''(0) & = & \frac{1}{\volume{\baseManifold}}\multipliedBy\mathrm{Det}\left(\begin{array}{ccccc}
\innerProduct{\magneticPotential}{\magneticPotential} & \innerProduct{\magneticPotential}{\omega_{1}} & \innerProduct{\magneticPotential}{\omega_{2}} & \ldots & \innerProduct{\magneticPotential}{\omega_{r_{0}}}\\
\innerProduct{\magneticPotential}{\omega_{1}} & 1 & 0 &  & 0\\
\innerProduct{\magneticPotential}{\omega_{2}} & 0 & 1 &  & 0\\
\vdots &  &  & \ddots & \vdots\\
\innerProduct{\magneticPotential}{\omega_{n}} & 0 & 0 & \ldots & 1
\end{array}\right)\\
 & = & \frac{1}{\volume{\baseManifold}}\multipliedBy\left(\normComingFromInnerProduct{\magneticPotential}^{2}-\sum_{j=1}^{r_{0}}\innerProduct{\magneticPotential}{\omega_{j}}^{2}\right)=\min_{[\omega]\in\homomorphismGroup(\coveringTransformationGroup,\mathbb{R})}\frac{\normComingFromInnerProduct{\magneticPotential-\omega}^{2}}{\volume{\baseManifold}}.
\end{eqnarray*}

\end{proof}
\newpage{}

\section{Ma\~n\'e's critical value and the ground state energy\label{sec:MCV_and_the_GSE}}

Let $\baseManifold$ be a connected closed manifold with potentials
$\magneticPotential\in\smoothOneForms{\baseManifold,\mathbb{R}}$
and $\electricPotential\in\smoothFunctions{\baseManifold,\mathbb{R}}$.
Recall that Ma\~n\'e's critical value of the corresponding lifted
Lagrangian $L_{\cover{\magneticPotential},\cover{\electricPotential}}$
on a regular cover $\cover{\baseManifold}$ of $\baseManifold$ is
given by~\cite[Theorem A]{ContrerasIturriagaPaternainPaternain1998}
\begin{eqnarray}
\manesCriticalValue(L_{\cover{\magneticPotential},\cover{\electricPotential}}) & = & \inf_{f\in\smoothFunctions{\cover{\baseManifold},\mathbb{R}}}\,\sup_{x\in\cover{\baseManifold}}\frac{1}{2}\left|\cover{\magneticPotential}+df\right|_{x}^{2}+\cover{\electricPotential}(x).\label{eq:MCV_in_Ham_Terms_Repeated}
\end{eqnarray}
For the trivial cover $\cover{\baseManifold}=\baseManifold$, the
following is known.
\begin{thm}
\label{thm:GSE_smaller_MCV_compact_case}(\cite[Theorem B]{Paternain2001})
We have $\groundStateEnergy(\magneticPotential,\electricPotential)\leq\manesCriticalValue(L_{\magneticPotential,\electricPotential})$.
\end{thm}
The preceding theorem is proven by evaluating the Rayleigh quotient
of $\schroedingerOperator_{\magneticPotential,\electricPotential}$
at a gauged version of the constant function $1_{\baseManifold}$
given by $1_{\baseManifold}(x)=1$ for $x\in\baseManifold$. For arbitrary
$\varepsilon>0$, let $f_{\varepsilon}\in\smoothFunctions{\baseManifold,\mathbb{R}}$
be such that 
\[
\sup_{x\in\baseManifold}\frac{1}{2}\left|\magneticPotential+df_{\varepsilon}\right|_{x}^{2}+\electricPotential(x)<\manesCriticalValue(L_{\magneticPotential,\electricPotential})+\varepsilon.
\]
Considering $\magneticPotential$ in the gauge $\magneticPotential+df_{\varepsilon}$,
we obtain from~(\ref{eq:GSE_of_Mag_Pot_plus_df})
\[
\groundStateEnergy(\magneticPotential,\electricPotential)=\groundStateEnergy(\magneticPotential+df_{\varepsilon},\electricPotential)\leq\frac{\innerProduct{\schroedingerOperator_{\magneticPotential+df_{\varepsilon},\electricPotential}1_{\baseManifold}}{1_{\baseManifold}}}{\innerProduct{1_{\baseManifold}}{1_{\baseManifold}}}=\frac{\int_{M}\frac{1}{2}\left|\magneticPotential+df_{\varepsilon}\right|^{2}+\electricPotential}{\volume M}<\manesCriticalValue(L_{\magneticPotential,\electricPotential})+\varepsilon.
\]
Note that the proof relies heavily on $1_{\baseManifold}\in L^{2}(\baseManifold)$.
Brooks~\cite{Brooks1981} was the first to discover that for amenable
covers $\cover{\baseManifold}$, there exist replacements $\chi_{\varepsilon}\in\smoothCompactlySupportedFunctions{\cover{\baseManifold},\mathbb{R}}$
for $1{}_{\cover{\baseManifold}}$ as described by Proposition~\ref{prop:Foelner_Sequence_for_amenable_covers}.
In combination with Paternain's approach, we obtain the following
generalization of Theorem~\ref{thm:GSE_smaller_MCV_compact_case}
to amenable covers, including the monopole case.
\begin{thm}
\label{thm:GSE_smaller_MCV_for_amenable_CTG}Let $\cover{\baseManifold}$
be a regular cover of $\baseManifold$ with amenable covering transformation
group. If $\cover{\magneticPotential}\in\smoothOneForms{\cover{\baseManifold},\mathbb{R}}$
and $\cover{\electricPotential}\in\smoothFunctions{\cover{\baseManifold},\mathbb{R}}$
are potentials with $\inf\,\cover{\electricPotential}>-\infty$, then
the associated Lagrangian $L_{\cover{\magneticPotential},\cover{\electricPotential}}$
and the associated magnetic Schr\"odinger operator $\schroedingerOperator_{\cover{\magneticPotential},\cover{\electricPotential}}$
satisfy
\begin{equation}
\groundStateEnergy(\closure{\schroedingerOperator_{\cover{\magneticPotential},\cover{\electricPotential}}})=\groundStateEnergy(\cover{\magneticPotential},\cover{\electricPotential})\leq\manesCriticalValue(L_{\cover{\magneticPotential},\cover{\electricPotential}}).\label{eq:GSE_smaller_MCV_for_amenable_CTG}
\end{equation}

\end{thm}
A proof for the exact case in which $\cover{\magneticPotential}$
and $\cover{\electricPotential}$ are lifts of potentials $\magneticPotential$
and $\electricPotential$ on $\baseManifold$ reads
\[
\groundStateEnergy(\cover{\magneticPotential},\cover{\electricPotential})=\min_{\begin{subarray}{c}
\omega\in\smoothOneForms{\baseManifold,\mathbb{R}}\colon\\
\cover{\omega}\textrm{ is exact}
\end{subarray}}\hspace{-3.5mm}\groundStateEnergy(\cover{\magneticPotential}-\cover{\omega},\cover{\electricPotential})\leq\min_{\begin{subarray}{c}
\omega\in\smoothOneForms{\baseManifold,\mathbb{R}}\colon\\
\cover{\omega}\textrm{ is exact}
\end{subarray}}\hspace{-3.5mm}\groundStateEnergy(\magneticPotential-\omega,\electricPotential)\leq\min_{\begin{subarray}{c}
\omega\in\smoothOneForms{\baseManifold,\mathbb{R}}\colon\\
\cover{\omega}\textrm{ is exact}
\end{subarray}}\hspace{-3.5mm}\manesCriticalValue(L_{\magneticPotential-\omega,\electricPotential})=\manesCriticalValue(L_{\cover{\magneticPotential},\cover{\electricPotential}}),
\]
where we used Theorem~\ref{thm:Gauge_Invariance}, Theorem~\ref{thm:Spectral_Inclusion_for_Amenable_Covers},
Theorem~\ref{thm:GSE_smaller_MCV_compact_case} and Theorem~\ref{thm:MCV_of_Amenable_Covers_and_Abelian_Subcovers_coincide},
respectively.
\begin{proof}
[Proof of Theorem \ref{thm:GSE_smaller_MCV_for_amenable_CTG}]Since
(\ref{eq:GSE_smaller_MCV_for_amenable_CTG}) trivially holds for $\manesCriticalValue(L_{\cover{\magneticPotential},\cover{\electricPotential}})=\infty$,
we assume $\manesCriticalValue(L_{\cover{\magneticPotential},\cover{\electricPotential}})<\infty$.
For arbitrary $\varepsilon>0$, we can find $f_{\varepsilon}\in\smoothFunctions{\cover{\baseManifold},\mathbb{R}}$
such that 
\[
\sup_{x\in\cover{\baseManifold}}\frac{1}{2}\left|\cover{\magneticPotential}+df_{\varepsilon}\right|_{x}^{2}+\cover{\electricPotential}(x)<\manesCriticalValue(L_{\cover{\magneticPotential},\cover{\electricPotential}})+\varepsilon.
\]
Let $\chi_{f_{\varepsilon},\varepsilon}\in\smoothCompactlySupportedFunctions{\cover{\baseManifold},\mathbb{R}}$
be as in Proposition~\ref{prop:Foelner_Sequence_for_amenable_covers}.
As before, we use~(\ref{eq:GSE_of_Mag_Pot_plus_df}) to deduce
\begin{eqnarray*}
\groundStateEnergy(\cover{\magneticPotential},\cover{\electricPotential}) & = & \groundStateEnergy(\cover{\magneticPotential}+df_{\varepsilon},\cover{\electricPotential})\leq\frac{\innerProduct{\schroedingerOperator_{\cover{\magneticPotential}+df_{\varepsilon},\cover{\electricPotential}}\chi_{f_{\varepsilon},\varepsilon}}{\chi_{f_{\varepsilon},\varepsilon}}}{\innerProduct{\chi_{f_{\varepsilon},\varepsilon}}{\chi_{f_{\varepsilon},\varepsilon}}}\\
 & = & \frac{1}{\volume{\baseManifold}}\left(\int_{M}\frac{1}{2}\left|-id\chi_{f_{\varepsilon},\varepsilon}+\chi_{f_{\varepsilon},\varepsilon}(\cover{\magneticPotential}+df_{\varepsilon})\right|^{2}+\chi_{f_{\varepsilon},\varepsilon}^{2}\cover{\electricPotential}\right)\\
 & = & \frac{1}{\volume{\baseManifold}}\left(\frac{1}{2}\normComingFromInnerProduct{d\chi_{f_{\varepsilon},\varepsilon}}^{2}+\int_{M}\multipliedBy\chi_{f_{\varepsilon},\varepsilon}^{2}\left(\frac{1}{2}\left|\cover{\magneticPotential}+df_{\varepsilon}\right|^{2}+\cover{\electricPotential}\right)\right)\\
 & < & \frac{1}{2}\frac{\varepsilon^{2}}{\volume{\baseManifold}}+\manesCriticalValue(L_{\cover{\magneticPotential},\cover{\electricPotential}})+\varepsilon.
\end{eqnarray*}
\vspace{-10mm}

\end{proof}
In order to illustrate the relationship between $\groundStateEnergy$
and $\manesCriticalValue$, let $\cover{\baseManifold}$ be a regular
cover of the closed manifold $\baseManifold$ with magnetic potential
$\cover{\magneticPotential}\in\smoothOneForms{\cover{\baseManifold},\mathbb{R}}$
having constant norm $\norm{\cover{\magneticPotential}}$. In particular,
$\eigenvalue_{0}^{\cover{\baseManifold}}\colon\mathbb{R}\to\mathbb{R}$
given by $\eigenvalue_{0}^{\cover{\baseManifold}}(\magneticFieldStrength)=\groundStateEnergy(\magneticFieldStrength\cover{\magneticPotential})$
is continuous by virtue of Proposition~\ref{prop:GSE_if_norm_constant}.
For the Lagrangian $L_{\magneticFieldStrength\cover{\magneticPotential}}=L_{\magneticFieldStrength\cover{\magneticPotential},0}$,
one easily obtains $\manesCriticalValue(L_{\magneticFieldStrength\cover{\magneticPotential}})=\magneticFieldStrength^{2}\manesCriticalValue(L_{\cover{\magneticPotential}})$
from~(\ref{eq:MCV_in_Ham_Terms_Repeated}). If $\cover{\baseManifold}$
is a non-amenable cover, then the extended theorem of Brooks in~\cite{KobayashiOnoSunada1989}
states that $\eigenvalue_{0}^{\cover{\baseManifold}}(0)>0$, which
implies $\eigenvalue_{0}^{\cover{\baseManifold}}(\magneticFieldStrength)=\groundStateEnergy(\magneticFieldStrength\cover{\magneticPotential})>\manesCriticalValue(L_{\magneticFieldStrength\cover{\magneticPotential}})$
near $\magneticFieldStrength=0$. In Section~\ref{sub:Higher_Genus_Surfaces},
we give an explicit example of this type, where $\cover{\magneticPotential}$
is the lift of some $\magneticPotential\in\smoothOneForms{\baseManifold,\mathbb{R}}$.
In Section~\ref{sub:HG}, we present a similar example with nilpotent
and therefore amenable $\pi_{1}(\baseManifold)$, such that $\eigenvalue_{0}^{\universalCover{\baseManifold}{}}(\magneticFieldStrength)\leq\eigenvalue_{0}^{\abelianCover{\baseManifold}{}}(\magneticFieldStrength)$
with equality only if $\norm{\magneticFieldStrength}\leq\frac{1}{2}$
or $\norm{\magneticFieldStrength}-\frac{1}{4}\in2\pi\multipliedBy\mathbb{Z}$
contrasting Theorem~\ref{thm:MCV_of_Amenable_Covers_and_Abelian_Subcovers_coincide}
which yields $\manesCriticalValue(\universalCoverWithoutHat L{\magneticFieldStrength\cover{\magneticPotential}})=\manesCriticalValue(\abelianCoverWithoutHat L{\magneticFieldStrength\cover{\magneticPotential}})$
for all $\magneticFieldStrength\in\mathbb{R}$, see also Corollary~\ref{cor:GSE_of_amenable_covers}.
It is worth mentioning that all examples of compact homogeneous spaces
$\baseManifold=\Lambda\backslash\Gamma$ in Section~\ref{sec:Hom_Ex_3_Body_Problem},
where $\Gamma$ is a Lie group with cocompact lattice $\Lambda\subset\Gamma$
and left-invariant potential descending to $\magneticPotential\in\smoothOneForms{\baseManifold,\mathbb{R}}$,
satisfy
\begin{equation}
\eigenvalue_{0}^{\cover{\baseManifold}}(\magneticFieldStrength)=\eigenvalue_{0}^{\cover{\baseManifold}}(0)+\magneticFieldStrength^{2}\manesCriticalValue(L_{\cover{\magneticPotential}})\qquad\textnormal{near }\magneticFieldStrength=0\label{eq:GSE_in_terms_of_MCV}
\end{equation}
on the universal cover and various intermediate covers, irrespective
of amenability of $\pi_{1}(\baseManifold)$. In particular, $\eigenvalue_{0}^{\cover{\baseManifold}}$
is a quadratic function near $\magneticFieldStrength=0$ with $\eigenvalue_{0}^{\cover{\baseManifold}}\hspace{0.5mm}''(0)=2\multipliedBy\manesCriticalValue(L_{\cover{\magneticPotential}})$.
Note that each such $\magneticPotential$ coincides with its coclosed
part $\magneticPotential_{\mathrm{cc}}$ since the Hodge dual $*\magneticPotential$
as well as $d*\magneticPotential$ lift to left-invariant forms on
$\Gamma$ for which reason $\exteriorDifferential*\magneticPotential$
is a constant multiple of the non-exact volume form on $\baseManifold$,
hence, $\adjoint{\exteriorDifferential}\magneticPotential=-*\exteriorDifferential*\magneticPotential=0$.
For the trivial cover $\cover{\baseManifold}=\baseManifold$ and arbitrary
$\magneticPotential\in\smoothOneForms{\baseManifold,\mathbb{R}}$,
Paternain~\cite[Proposition 4.2]{Paternain2001} showed that (\ref{eq:GSE_in_terms_of_MCV})
holds precisely if $\manesCriticalValue(L_{\magneticPotential})$
coincides with 
\[
h(L_{\magneticPotential})=\inf_{f\in\smoothFunctions{\baseManifold,\mathbb{R}}}\frac{\int_{\baseManifold}\frac{1}{2}\norm{\magneticPotential+df}^{2}}{\volume{\baseManifold}}=\frac{1}{2}\frac{\normComingFromInnerProduct{\magneticPotential_{\mathrm{cc}}}^{2}}{\volume{\baseManifold}}.
\]
This condition is satisfied if $\baseManifold$ is a homogeneous space
with $\magneticPotential\in\smoothOneForms{\baseManifold,\mathbb{R}}$
as above or, more generally, if $\norm{\magneticPotential_{\mathrm{cc}}}_{x}^{2}=2\multipliedBy h(L_{\magneticPotential})$
for all $x\in\baseManifold$ as this implies 
\[
h(L_{\magneticPotential})\leq\manesCriticalValue(L_{\magneticPotential})\leq\sup_{x\in\baseManifold}\frac{1}{2}\norm{\magneticPotential_{\mathrm{cc}}}_{x}^{2}=h(L_{\magneticPotential}).
\]
 In the former case, Theorem~\ref{thm:MCV_of_Amenable_Covers_and_Abelian_Subcovers_coincide}
and Theorem~\ref{thm:Derivative_of_GSE_of_abelian_covers} give $\eigenvalue_{0}^{\cover{\baseManifold}}\hspace{0.5mm}''(0)=2\multipliedBy\manesCriticalValue(L_{\cover{\magneticPotential}})$
on covers with abelian covering transformation group $\coveringTransformationGroup$
for which $\homomorphismGroup(\coveringTransformationGroup,\mathbb{R})$
is generated by left-invariant forms.

\newpage{}

\section{Homogeneous examples and the $3$-body problem\label{sec:Hom_Ex_3_Body_Problem}}

In this section, we study magnetic Schr\"odinger operators on covers
$\cover{\baseManifold}$ of compact homogeneous spaces $\baseManifold=\Lambda\backslash\Gamma$,
and compare with the corresponding classical data. In each case, $\Gamma$
is a Lie group equipped with a left-invariant metric and a left-invariant
magnetic field $\magneticField\in\smoothTwoForms{\Gamma,\mathbb{R}}$,
and $\Lambda\subset\Gamma$ is a cocompact lattice, that is, a discrete
subgroup such that the quotient $\Lambda\backslash\Gamma$ is compact.
According to Theorem~\ref{thm:MSO_are_Essentially_Selfadjoint},
the corresponding magnetic Schr\"odinger operators are essentially
self-adjoint. Therefore, we use the same symbol $\schroedingerOperator^{\cover{\baseManifold}}$
for the respective operator and its self-adjoint closure with the
exception of the last subsection, where we discuss the quantum analogue
of the classical $3$-body problem.\global\long\def\closure#1{#1}

\subsection{Tori}

We consider $\mathbb{T}^{n}=\mathbb{Z}^{n}\backslash\mathbb{R}^{n}=\mathbb{R}^{n}/\mathbb{Z}^{n}$
with its usual flat metric coming from the Euclidean metric on $\mathbb{R}^{n}$.
Note that the lift of any left-invariant $2$-form $\magneticField$
on $\mathbb{T}^{n}$ to the universal cover $\mathbb{R}^{n}$ takes
the form $\magneticField^{\mathbb{R}^{n}}=\sum_{j<k}B_{jk}dy^{j}\wedge dy^{k}$
for some $B_{jk}\in\mathbb{R}$. We briefly recall that there exists
an isometry $Q\in O(n)$ of $\mathbb{R}^{n}$ with respect to which
\begin{equation}
\magneticField^{\mathbb{R}^{n}}=\sum_{j=1}^{\frac{r(B)}{2}}\lambda_{j}dx^{2j-1}\wedge dx^{2j}\text{ with all }\lambda_{j}\neq0,\label{eq:Normal_form_of_skewsymmetric_form}
\end{equation}
where $(x_{1},x_{2},\ldots,x_{n})$ denote the linear coordinates
given by $Q\colon\mathbb{R}^{n}\to\mathbb{R}^{n}$, and $r(B)\in2\multipliedBy\mathbb{N}_{0}$
denotes the rank of the skew-symmetric matrix $B=(B_{jk})_{j,k=1}^{n}$
with entries $B_{jk}=-B_{kj}$ for $j>k$. Since $iB$ is hermitian,
it has real eigenvalues and can be diagonalized by a unitary matrix
each of whose columns $u$ satisfies $(iB)u=\lambda u$ for some $\lambda\in\mathbb{R}$.
Complex conjugation leads to $(iB)\conjugate u=-\lambda\conjugate u$.
Hence, $iB$ can be diagonalized by a unitary matrix having columns
\[
(u_{1},\conjugate{u_{1}},u_{2},\conjugate{u_{2}},\ldots u_{\frac{r(B)}{2}},\conjugate{u_{\frac{r(B)}{2}}},u_{r(B)+1},u_{r(B)+2},\ldots,u_{n}),
\]
where $u_{j}\in\ker B\cap\mathbb{R}^{n}$ for $j>r(B)$. Moreover,
$B$ has eigenvalues of the form 
\[
(-i\lambda_{1},i\lambda_{1},-i\lambda_{2},i\lambda_{2},\ldots,-i\lambda_{\frac{r(B)}{2}},i\lambda_{\frac{r(B)}{2}},0,0,\ldots,0),
\]
where all $\lambda_{j}\neq0$. If $u_{j}=v_{j}+iw_{j}$ with real
and imaginary parts $v_{j},w_{j}\in\mathbb{R}^{n}$, then we can use
the orthonormal basis $\{v_{j}\pm w_{j}\}_{j\leq\frac{r(B)}{2}}\cup\{u_{j}\}_{j>r(B)}$
of $\mathbb{R}^{n}$ to obtain~(\ref{eq:Normal_form_of_skewsymmetric_form}).
Since $\homologyGroup^{1}(\mathbb{R}^{n},\mathbb{R})$ is trivial,
any magnetic potential of (\ref{eq:Normal_form_of_skewsymmetric_form})
is gauge equivalent to
\begin{equation}
\magneticPotential=\sum_{j=1}^{\frac{r(B)}{2}}\lambda_{j}\multipliedBy x_{2j-1}\multipliedBy dx^{2j}.\label{eq:Mag_Pot_on_Rn}
\end{equation}
If $\magneticField^{\mathbb{R}^{n}}$ has bounded primitives, then,
by amenability of $\pi_{1}(\mathbb{T}^{n})=\mathbb{Z}^{n}$ and~\cite[Lemma 5.3]{Paternain2006},
it also has left-invariant ones, that is, $[\magneticField]=0\in\homologyGroup^{2}(\mathbb{T}^{n},\mathbb{R})$,
and we have $\magneticField=0$. The critical value on $\mathbb{R}^{n}$
is therefore given by
\[
\manesCriticalValue(H_{\magneticPotential}^{\mathbb{R}^{n}})=\begin{cases}
0 & \textnormal{if }\magneticField=0\\
\infty & \textnormal{otherwise},
\end{cases}
\]
where $H_{\magneticPotential}^{\mathbb{R}^{n}}\colon T^{*}\mathbb{R}^{n}\to\mathbb{R}$
is the magnetic Hamiltonian corresponding to the potential~(\ref{eq:Mag_Pot_on_Rn}).
If $\magneticField$ is symplectic, meaning $r(B)=n$, then a similarly
drastic behavior can be observed for the spectrum of the corresponding
magnetic Schr\"odinger operator. Following~\cite[Section 1.4.3]{FournaisHelffer2010},
we let $\mathrm{Tr}^{+}(\magneticField)=\sum_{j=1}^{\frac{r(B)}{2}}\norm{\lambda_{j}}$.
\begin{thm}
\label{thm:Spectrum_of_Tori}The closure of
\[
\schroedingerOperator_{\magneticPotential}^{\mathbb{R}^{n}}=\sum_{j=1}^{\frac{r(B)}{2}}\left(\left(\frac{1}{i}\frac{\partial}{\partial x_{2j-1}}\right)^{2}+\left(\frac{1}{i}\frac{\partial}{\partial x_{2j}}+\lambda_{j}x_{2j-1}\right)^{2}\right)+\sum_{j=r(B)+1}^{n}\left(\frac{1}{i}\frac{\partial}{\partial x_{j}}\right)^{2}
\]
with initial domain $\smoothCompactlySupportedFunctions{\mathbb{R}^{n}}$
has the purely essential spectrum
\[
\spectrum{\closure{\schroedingerOperator_{\magneticPotential}^{\mathbb{R}^{n}}}}=\essentialSpectrum{\closure{\schroedingerOperator_{\magneticPotential}^{\mathbb{R}^{n}}}}=\begin{cases}
\mathrm{\mathrm{[\mathrm{Tr}^{+}(\magneticField),\infty)}} & \textnormal{if }r(B)<n\\
\left\{ \sum_{j=1}^{\frac{r(B)}{2}}\norm{\lambda_{j}}(2k_{j}+1)\,\Bigl|\, k_{j}\in\{0,1,2,\ldots\}\right\}  & \textnormal{\textnormal{if }\ensuremath{\magneticField\textnormal{ is symplectic}}}.
\end{cases}
\]
In any case, we have $\groundStateEnergy(\closure{\schroedingerOperator_{\magneticPotential}^{\mathbb{R}^{n}}})=\mathrm{Tr}^{+}(\magneticField)$.\end{thm}
\begin{proof}
The theorem is a special case of~\cite[Theorem 2.2 and Theorem 2.4]{Shigekawa1991},
which deal with asymptotically constant magnetic fields on $\mathbb{R}^{n}$.
For the reader's convenience, we provide a less involved proof along
the lines of~\cite[Section 1.4.3]{FournaisHelffer2010}. Due to Theorem~\ref{thm:Eigenspaces_are_infinite_dimensional},
we know that $\closure{\schroedingerOperator_{\magneticPotential}^{\mathbb{R}^{n}}}$
has purely essential spectrum. Since $\closure{\schroedingerOperator_{\magneticPotential}^{\mathbb{R}^{n}}}$
is a sum of operators as in~\cite[Theorem VIII.33]{ReedSimon1980},
we see that 
\[
\left(\bigotimes_{j=1}^{\frac{r(B)}{2}}\smoothCompactlySupportedFunctions{\mathbb{R}^{2}}\right)\otimes\smoothCompactlySupportedFunctions{\mathbb{R}^{n-r(B)}}\subset\left(\bigotimes_{j=1}^{\frac{r(B)}{2}}L^{2}(\mathbb{R}^{2})\right)\otimes L^{2}(\mathbb{R}^{n-r(B)})\simeq L^{2}(\mathbb{R}^{n})
\]
is a core, and it suffices to study the self-adjoint model operators
$\schroedingerOperator_{\lambda}^{\mathbb{R}^{2}}$ and $\Delta_{\mathbb{R}^{n-r(B)}}$
given as 
\begin{equation}
\schroedingerOperator_{\lambda}^{\mathbb{R}^{2}}=\left(\frac{1}{i}\frac{\partial}{\partial x}\right)^{2}+\left(\frac{1}{i}\frac{\partial}{\partial y}+\lambda x\right)^{2}\qquad\textnormal{ and }\qquad\Delta_{\mathbb{R}^{n-r(B)}}=\sum_{j=1}^{n-r(B)}\left(\frac{1}{i}\frac{\partial}{\partial x_{j}}\right)^{2}\label{eq:Model_Operators_on_Rn}
\end{equation}
on their initial domains $\smoothCompactlySupportedFunctions{\mathbb{R}^{2}}\subset L^{2}(\mathbb{R}^{2})$
and $\smoothCompactlySupportedFunctions{\mathbb{R}^{n-r(B)}}\subset L^{2}(\mathbb{R}^{n-r(B)})$,
respectively. Using Fourier transformation or quasi-modes as in Theorem~\ref{thm:Spec_of_Sphere_Bundle_of_Hyp_Space}
and Theorem~\ref{thm:EMSO_on_Sol} below, the minimal Laplacian $\Delta_{\mathbb{R}^{n-r(B)}}$
is easily seen to have spectrum $\spectrum{\Delta_{\mathbb{R}^{n-r(B)}}}=[0,\infty)$,
see also \cite[Theorems 3.5.3, 3.7.4, 8.3.1]{Davies1995}. The spectral
analysis of $\schroedingerOperator_{\lambda}^{\mathbb{R}^{2}}$ dates
back to Landau in the 1930s. Using dominated convergence, one can
show that $\domain{\schroedingerOperator_{\lambda}^{\mathbb{R}^{2}}}$
contains the space $\mathcal{S}(\mathbb{R}^{2})\subset\smoothFunctions{\mathbb{R}^{2}}\cap L^{2}(\mathbb{R}^{2})$
of Schwartz functions, and the restriction of $\schroedingerOperator_{\lambda}^{\mathbb{R}^{2}}$
to $\mathcal{S}(\mathbb{R}^{2})$ is given by the differential operator
in~(\ref{eq:Model_Operators_on_Rn}). We conjugate by the partial
Fourier transformation $\mathcal{F}_{y}\colon\mathcal{S}(\mathbb{R}^{2})\to\mathcal{S}(\mathbb{R}^{2})$
given by
\[
\mathcal{F}_{y}u(x,\momentumVariable y{})=\frac{1}{\sqrt{2\pi}}\intop_{-\infty}^{\infty}u(x,y)\multipliedBy e^{-i\momentumVariable y{}y}\multipliedBy dy
\]
to obtain 
\begin{equation}
\mathcal{F}_{y}\multipliedBy\schroedingerOperator_{\lambda}^{\mathbb{R}^{2}}\multipliedBy\mathcal{F}_{y}^{-1}=-\frac{\partial^{2}}{\partial x^{2}}+\lambda^{2}\left(x+\frac{\momentumVariable y{}}{\lambda}\right)^{2}.\label{eq:Conjugated_Landau_Hamiltonian}
\end{equation}
The shift $x\to x+\frac{\momentumVariable y{}}{\lambda}$ leads to
the harmonic oscillator in $x$ with frequency $\norm{\lambda}$,
which has the well-known discrete spectrum $\{\norm{\lambda}(2k+1)\,|\, k=0,1,2,\ldots\}$
with eigenfunctions~\cite{ReedSimon1980} 
\[
u_{k}(x)=\frac{1}{\sqrt{2^{k}\, k!}}\left(\frac{\norm{\lambda}}{\pi}\right)^{\frac{1}{4}}P_{k}\left(\sqrt{\norm{\lambda}}x\right)\multipliedBy e^{-\frac{\norm{\lambda}x^{2}}{2}},\quad\textnormal{where}\quad P_{k}(x)=\left(-1\right)^{k}\, e^{x^{2}}\,\frac{d^{k}}{dx^{k}}e^{-x^{2}}
\]
is the $k$th Hermite polynomial. Each $u_{k}$ leads to an infinite-dimensional
eigenspace of~(\ref{eq:Conjugated_Landau_Hamiltonian}) containing
products of the form $u_{k}(x+\frac{\momentumVariable y{}}{\lambda})\psi(\momentumVariable y{})$
where $\psi\in\smoothCompactlySupportedFunctions{\mathbb{R}}$.
\end{proof}

\subsection{Higher genus surfaces and $PSL(2,\mathbb{R})$\label{sub:Higher_Genus_Surfaces}}

\global\long\def\hyperbolicSpace{\mathbb{H}}

Let $\hyperbolicSpace=\{z=x+iy\in\mathbb{C}\simeq\mathbb{R}^{2}\,|\,\mathrm{Im}z=y>0\}$
denote the upper half plane with its usual metric $ds^{2}=y^{-2}(dx^{2}+dy^{2})$
of constant curvature $-1$. Recall that $PSL(2,\mathbb{R})$ can
be regarded as the group of orientation-preserving isometries of $\hyperbolicSpace$,
that is, Möbius transformations 
\[
z\mapsto\frac{az+b}{cz+d}\qquad\textnormal{with }\left(\begin{array}{cc}
a & b\\
c & d
\end{array}\right)\in SL(2,\mathbb{R}).
\]
This allows to view $PSL(2,\mathbb{R})$ as the unit sphere bundle
$S\hyperbolicSpace$ by identifying $(x,y,v)\in S\hyperbolicSpace$
with the unique Möbius transformation that takes $i$ to $x+iy$,
and whose derivative takes $(0,1)\in T_{i}\hyperbolicSpace$ to $v\in T_{x+iy}\hyperbolicSpace$.
We let $\Lambda\subset PSL(2,\mathbb{R})$ be a cocompact lattice
that acts on $\hyperbolicSpace$ without fixed points. The quotient
space $\baseManifold=\Lambda\backslash\hyperbolicSpace$ is a compact
hyperbolic surface, whose unit sphere bundle $S\baseManifold$ can
be identified with $\Lambda\backslash PSL(2,\mathbb{R})$. Following~\cite[Section 5.2]{CieliebakFrauenfelderPaternain2010},
we let $\magneticFieldStrength\in\mathbb{R}$ and consider the magnetic
flow $\phi^{\magneticFieldStrength}$ on $T\baseManifold$ generated
by the magnetic field $\magneticFieldStrength\multipliedBy\volumeForm_{M}$.
In other words, $\phi^{\magneticFieldStrength}$ is the Hamiltonian
flow of $E\colon T\baseManifold\to\mathbb{R}$ given by $E(x,v)=\frac{1}{2}\norm v_{x}^{2}$
with respect to the twisted symplectic form 
\[
\omega_{\magneticFieldStrength}=-\flat^{*}\exteriorDifferential\lambda+\magneticFieldStrength\multipliedBy\pi_{T\baseManifold}^{*}\volumeForm_{M},
\]
where $\flat\colon T\baseManifold\to T^{*}\baseManifold$ is the canonical
isomorphism induced by the metric, $\lambda$ is the Liouville $1$-form
on $T^{*}M$, and $\pi_{T\baseManifold}\colon TM\to M$ is the canonical
projection. As the lift of $\volumeForm_{M}$ to $\hyperbolicSpace$
has the primitve $\magneticPotential=y^{-1}dx$ with constant norm
$1$, the corresponding Hamiltonian 
\begin{equation}
H_{\magneticFieldStrength}^{\hyperbolicSpace}\colon T^{*}\hyperbolicSpace\to\mathbb{R}\qquad\textnormal{given by}\qquad H_{\magneticFieldStrength}^{\hyperbolicSpace}(x,y,p_{x},p_{y})=\frac{1}{2}\left((yp_{x}+\magneticFieldStrength)^{2}+y^{2}p_{y}^{2}\right)\label{eq:Hamiltonian_on_Hyp_Space}
\end{equation}
has critical value $\manesCriticalValue(H_{\magneticFieldStrength}^{\hyperbolicSpace})\leq\frac{1}{2}\magneticFieldStrength^{2}$.
In fact, $\manesCriticalValue(H_{\magneticFieldStrength}^{\hyperbolicSpace})=\frac{1}{2}\magneticFieldStrength^{2}$
as was shown in~\cite[Example 6.2]{Contreras2001} and~\cite[Lemma 6.11]{CieliebakFrauenfelderPaternain2010}.
For $k\geq0$, we denote the restriction of $\phi^{\magneticFieldStrength}$
to $E^{-1}(k)$ by $\phi^{\magneticFieldStrength,k}$. Using a simple
scaling in the fibres of $T\hyperbolicSpace$ of the form $(x,v)\mapsto(x,\magneticFieldStrength v)$,
we obtain the following from~\cite{CieliebakFrauenfelderPaternain2010}:
\begin{itemize}
\item For $k>\manesCriticalValue(H_{\magneticFieldStrength}^{\hyperbolicSpace})$,
the flow $\phi^{\magneticFieldStrength,k}$ is conjugate to the underlying
geodesic flow on $\baseManifold$ up to a constant time scaling.
\item For $k=\manesCriticalValue(H_{\magneticFieldStrength}^{\hyperbolicSpace})$,
the flow $\phi^{\magneticFieldStrength,k}$ is conjugate to the horocycle
flow on $S\baseManifold$ up to a constant time scaling, in particular,
it has no closed orbits.
\item For $0\leq k<\manesCriticalValue(H_{\magneticFieldStrength}^{\hyperbolicSpace})$,
all orbits of $\phi^{\magneticFieldStrength,k}$ are closed and their
projections to $\hyperbolicSpace$ are circles.
\end{itemize}
The quantum analogue of~(\ref{eq:Hamiltonian_on_Hyp_Space}) is known
as the Maass Laplacian~\cite{Maass1952}
\begin{equation}
\schroedingerOperator_{\magneticFieldStrength}^{\hyperbolicSpace}=\frac{1}{2}\multipliedBy y^{2}\left(\left(\frac{1}{i}\frac{\partial}{\partial x}+\magneticFieldStrength y^{-1}\right)^{2}+\left(\frac{1}{i}\frac{\partial}{\partial y}\right)^{2}\right).\label{eq:Maass_Laplacian}
\end{equation}
Due to its close connections to number theory, in particular, to automorphic
forms, $\schroedingerOperator_{\magneticFieldStrength}^{\hyperbolicSpace}$
has been widely studied in both the mathematics and theoretical physics
literature~\cite{Roelcke1966,Roelcke1966a,Elstrodt1973a,Elstrodt1973,Elstrodt1974,ComtetHouston1985,Comtet1987}.
Note that $\schroedingerOperator_{\magneticFieldStrength}^{\hyperbolicSpace}$
and $\schroedingerOperator_{-\magneticFieldStrength}^{\hyperbolicSpace}$
are conjugate to each other with respect to the isometry $\mathcal{I}\colon L^{2}(\hyperbolicSpace)\to L^{2}(\hyperbolicSpace)$
given by $\mathcal{I}u(x,y)=u(-x,y)$. Derivations of the well-known
spectrum of (\ref{eq:Maass_Laplacian}) can be found in~\cite{InahamaShirai2003,MorameTruc2008},
which motivated the spectral analysis in the remainder of this section.
\begin{thm}
\label{thm:Spec_of_Hyp_Space}The closure of~(\ref{eq:Maass_Laplacian})
has spectrum

\begin{equation}
\spectrum{\closure{\schroedingerOperator_{\magneticFieldStrength}^{\hyperbolicSpace}}}=\begin{cases}
\left[\groundStateEnergyOfContinuousSpectrum(\closure{\schroedingerOperator_{\magneticFieldStrength}^{\hyperbolicSpace}}),\infty\right) & \textnormal{if }|\magneticFieldStrength|\leq\frac{1}{2}\\
\purePointSpectrum{\closure{\closure{\schroedingerOperator_{\magneticFieldStrength}^{\hyperbolicSpace}}}}\cup\left[\groundStateEnergyOfContinuousSpectrum(\closure{\schroedingerOperator_{\magneticFieldStrength}^{\hyperbolicSpace}}),\infty\right) & \textnormal{if }|\magneticFieldStrength|>\frac{1}{2},
\end{cases}\label{eq:Spec_of_Hyp_space}
\end{equation}
where $\groundStateEnergyOfContinuousSpectrum(\closure{\schroedingerOperator_{\magneticFieldStrength}^{\hyperbolicSpace}})=\frac{1}{2}\left(|\magneticFieldStrength|^{2}+\frac{1}{4}\right)$
and 
\[
\purePointSpectrum{\closure{\schroedingerOperator_{\magneticFieldStrength}^{\hyperbolicSpace}}}=\bigcup_{0\leq k<|\magneticFieldStrength|-\frac{1}{2}}\left\{ \frac{1}{2}\left((2k+1)|\magneticFieldStrength|-k(k+1)\right)\right\} \subset\essentialSpectrum{\closure{\schroedingerOperator_{\magneticFieldStrength}^{\hyperbolicSpace}}}.
\]

\end{thm}
The preceding theorem is visualized in Figure~\ref{fig:Plot_of_Hyp_Space}.
We briefly verify that for $|\magneticFieldStrength|>\frac{1}{2}$,
any $v\in\smoothCompactlySupportedFunctions{(-\infty,0),\mathbb{C}}$
leads to a ground state of~(\ref{eq:Maass_Laplacian}) given by 
\[
u(x,y)=y^{|\magneticFieldStrength|}\intop_{-\infty}^{\infty}e^{\momentumVariable x{}(ix+y)}v(\momentumVariable x{})d\momentumVariable x{}.
\]
Note that $w\colon\mathbb{R}\times\mathbb{R}_{+}\to\mathbb{C}$ with
$w(x,y)=y^{-1}u(x,y)$ is the inverse partial Fourier transform of
$\widehat{w}(\momentumVariable x{},y)=\sqrt{2\pi}\multipliedBy y^{|\magneticFieldStrength|-1}e^{\momentumVariable x{}y}v(\momentumVariable x{})$.
We apply Parseval's theorem to obtain 
\[
\normComingFromInnerProduct u_{L^{2}(\hyperbolicSpace)}=\normComingFromInnerProduct w_{L^{2}(\mathbb{R}\times\mathbb{R}_{+})}=\normComingFromInnerProduct{\widehat{w}}_{L^{2}(\mathbb{R}\times\mathbb{R}_{+})},
\]
where $\mathbb{R}\times\mathbb{R}_{+}$ is equipped with standard
Lebesque product measure. Hence,
\[
\normComingFromInnerProduct u_{L^{2}(\hyperbolicSpace)}^{2}\leq2\pi\intop_{-\infty}^{\infty}\intop_{0}^{\infty}y^{2|\magneticFieldStrength|-2}e^{2\momentumVariable x{}y}|v(\momentumVariable x{})|^{2}dy\multipliedBy d\momentumVariable x{}<\infty
\]
since $|\magneticFieldStrength|>\frac{1}{2}$ and since $v$ vanishes
on $[-\varepsilon,\infty)$ for some $\varepsilon>0$. Moreover,
\begin{eqnarray*}
\left(y\frac{1}{i}\frac{\partial}{\partial x}+|\magneticFieldStrength|\right)^{2}u & = & (y\momentumVariable x{}+|\magneticFieldStrength|)^{2}u\\
y^{2}\left(\frac{1}{i}\frac{\partial}{\partial y}\right)^{2}u & = & -\left(y^{2}\momentumVariable x2+2|\magneticFieldStrength|y\momentumVariable x{}+|\magneticFieldStrength|(|\magneticFieldStrength|-1)\right)u.
\end{eqnarray*}
In the following, we examine the $3$-dimensional unit sphere bundle
$S\baseManifold\simeq\Lambda\backslash PSL(2,\mathbb{R})$. The cover
$PSL(2,\mathbb{R})\simeq S\hyperbolicSpace$ is diffeomorphic to $\hyperbolicSpace\times\mathbb{S}^{1}$
with coordinates $(x,y,\varphi)$, with respect to which we have the
following left-invariant $1$-forms~\cite[Section 6.3]{CieliebakFrauenfelderPaternain2010}
\[
\magneticPotential_{1}=\frac{\cos\varphi\multipliedBy dx+\sin\varphi\multipliedBy dy}{y}\hspace{1cm}\magneticPotential_{2}=\frac{-\sin\varphi\multipliedBy dx+\cos\varphi\multipliedBy dy}{y}\hspace{1cm}\magneticPotential_{3}=\frac{dx}{y}+d\varphi.
\]
We endow $S\hyperbolicSpace$ with the left-invariant metric
\[
ds^{2}=\magneticPotential_{1}^{2}+\magneticPotential_{2}^{2}+\magneticPotential_{3}^{2}=\frac{1}{y^{2}}\left(dx^{2}+dy^{2}+(y\multipliedBy d\varphi+dx)^{2}\right),
\]
that is,
\[
g=\frac{1}{y^{2}}\left(\begin{array}{ccc}
2 & 0 & y\\
0 & 1 & 0\\
y & 0 & y^{2}
\end{array}\right)\hspace{3mm}\sqrt{|g|}=y^{-2}\hspace{3mm}g^{-1}=\left(\begin{array}{ccc}
y^{2} & 0 & -y\\
0 & y^{2} & 0\\
-y & 0 & 2
\end{array}\right).
\]
We let $\magneticFieldStrength\in\mathbb{R}$ and consider $\magneticFieldStrength\multipliedBy\magneticPotential_{3}$
as a left-invariant magnetic potential with associated magnetic field
$\magneticField=\magneticFieldStrength\multipliedBy\frac{1}{y^{2}}dx\wedge dy$.
The corresponding Hamiltonian $H_{\magneticFieldStrength}^{S\baseManifold}\colon T^{*}S\baseManifold\to\mathbb{R}$
and its lift 
\[
H_{\magneticFieldStrength}^{\overline{SL(2,\mathbb{R})}}(x,y,\varphi,p_{x},p_{y},p_{\varphi})=\frac{1}{2}\left((y\multipliedBy p_{x}-p_{\varphi})^{2}+(y\multipliedBy p_{y})^{2}+(p_{\varphi}+\magneticFieldStrength)^{2}\right)
\]
to the universal cover $\overline{SL(2,\mathbb{R})}$ are known to
have critical values~\cite[Section 6.3]{CieliebakFrauenfelderPaternain2010}
\[
\manesCriticalValue(H_{\magneticFieldStrength}^{S\baseManifold})=\frac{1}{2}\magneticFieldStrength^{2}\quad\textnormal{and}\quad\manesCriticalValue(H_{\magneticFieldStrength}^{\overline{SL(2,\mathbb{R})}})=\frac{1}{4}\magneticFieldStrength^{2}.
\]
We examine the corresponding magnetic Schr\"odinger operator 
\[
\schroedingerOperator_{\magneticFieldStrength}^{S\hyperbolicSpace}=\frac{1}{2}\left(\left(y\frac{1}{i}\frac{\partial}{\partial x}-\frac{1}{i}\frac{\partial}{\partial\varphi}\right)^{2}+y^{2}\left(\frac{1}{i}\frac{\partial}{\partial y}\right)^{2}+\left(\frac{1}{i}\frac{\partial}{\partial\varphi}+\magneticFieldStrength\right)^{2}\right)
\]
on the intermediate cover $S\hyperbolicSpace$ and its lift $\schroedingerOperator_{\magneticFieldStrength}^{\overline{SL(2,\mathbb{R})}}$
to $\overline{SL(2,\mathbb{R})}$. Figure~\ref{fig:Plot_of_Sphere_Bundle_of_Hyp_Space}
summarizes the content of the following thoerem, which uses the notation
$\lfloor|\magneticFieldStrength|\rfloor=\max\{m\in\mathbb{Z}\,|\, m\leq|\magneticFieldStrength|\}$.

\begin{figure}
\noindent \begin{centering}
\psfrag{0}{} \psfrag{1}{\hspace{-0.6mm}\raisebox{-0.8mm}{\scriptsize{$1$}}} \psfrag{2}{\hspace{-0.6mm}\raisebox{-0.8mm}{\scriptsize{$2$}}} \psfrag{3}{\hspace{-0.6mm}\raisebox{-0.8mm}{\scriptsize{$3$}}} \psfrag{4}{\hspace{-0.6mm}\raisebox{-0.8mm}{\scriptsize{$4$}}} \psfrag{5}{\hspace{-0.6mm}\raisebox{-0.8mm}{\scriptsize{$5$}}}\psfrag{B}{\hspace{-2mm}\raisebox{-3.5mm}{$|\magneticFieldStrength|$}}\subfloat[$\spectrum{\closure{\schroedingerOperator_{\magneticFieldStrength}^{\hyperbolicSpace}}}$
as given in~(\ref{eq:Spec_of_Hyp_space})\label{fig:Plot_of_Hyp_Space}]{\noindent \begin{centering}
\psfrag{a}{\raisebox{-0mm}{$\groundStateEnergy(\closure{\schroedingerOperator_{\magneticFieldStrength}^{\overline{SL(2,\mathbb{R})}}})$}}\includegraphics[scale=0.75]{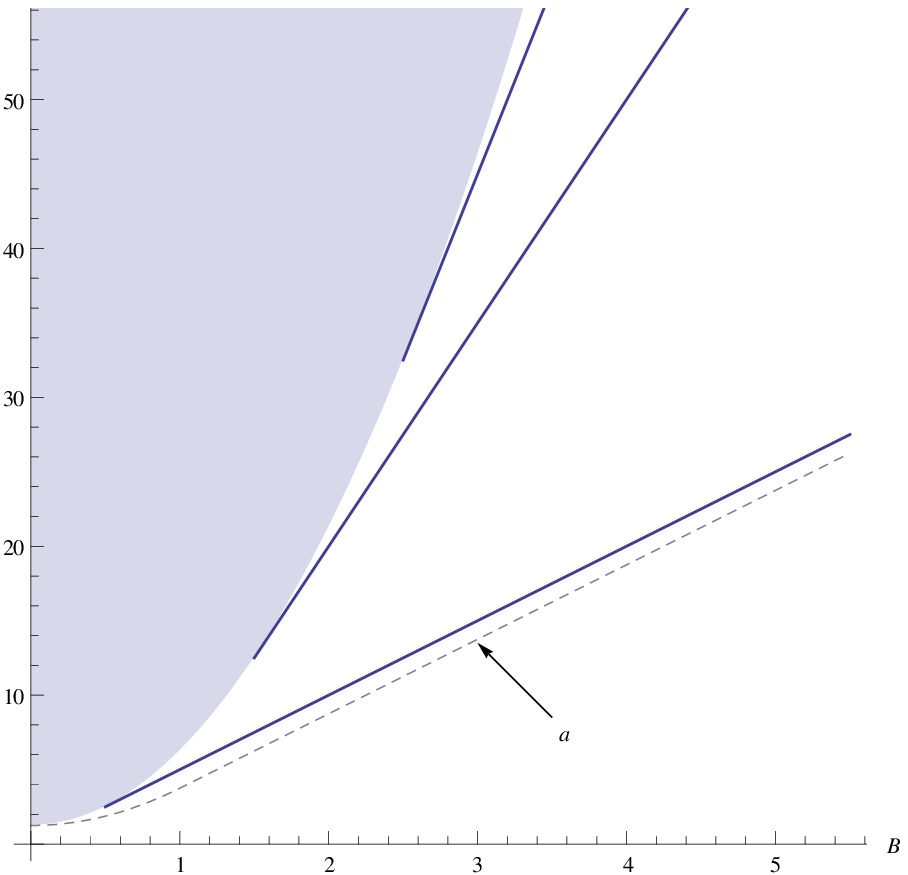}
\par\end{centering}

}\hfill{}\subfloat[$\spectrum{\closure{\schroedingerOperator_{\magneticFieldStrength}^{S\hyperbolicSpace}}}$
as given in~(\ref{eq:Spec_of_Sphere_Bundle_of_Hyp_Space})\label{fig:Plot_of_Sphere_Bundle_of_Hyp_Space}]{\noindent \begin{centering}
\psfrag{a}{\raisebox{-0mm}{$\groundStateEnergy(\closure{\schroedingerOperator_{\magneticFieldStrength}^{\overline{SL(2,\mathbb{R})}}})$}} \psfrag{b}{\raisebox{-0mm}{$\frac{1}{2}|\magneticFieldStrength|$}}  \psfrag{c}{\raisebox{-0mm}{$\frac{1}{2}(\frac{1}{2}|\magneticFieldStrength|^{2}+\frac{1}{4})$}}\includegraphics[scale=0.75]{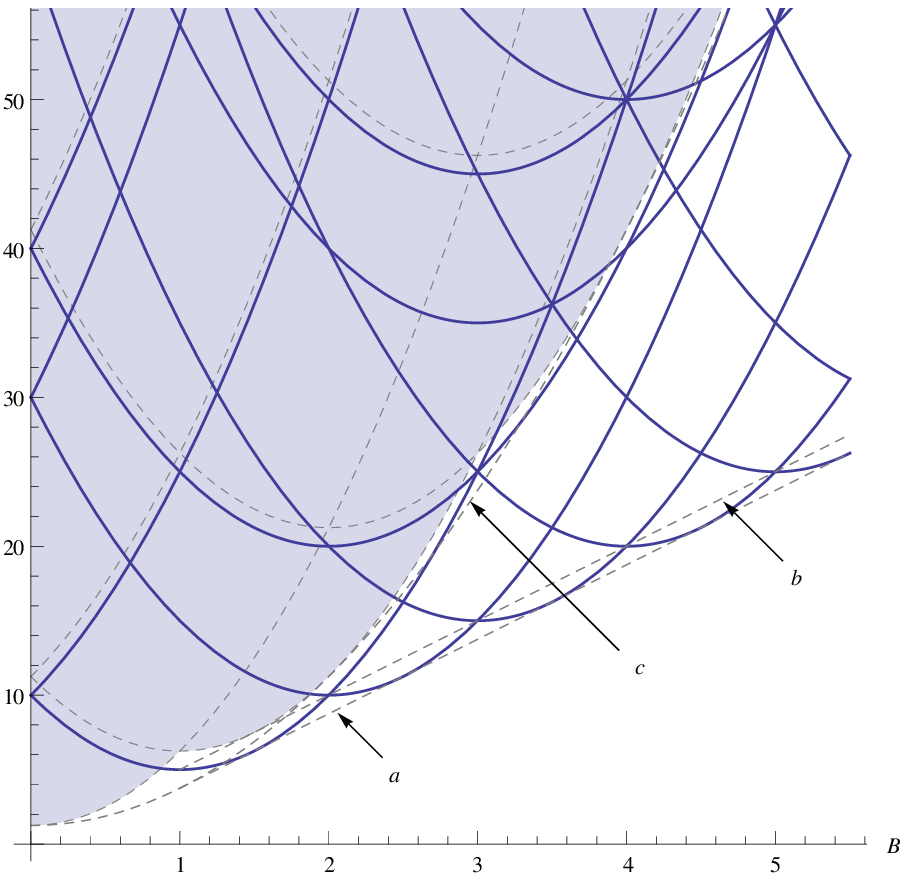}
\par\end{centering}

}
\par\end{centering}

\caption{Plots of (\ref{eq:Spec_of_Hyp_space}) and (\ref{eq:Spec_of_Sphere_Bundle_of_Hyp_Space}),
where solid lines indicate point spectrum and shaded regions indicate
continuous spectrum. The dotted lines are drawn for comparison.}
\end{figure}

\begin{thm}
\label{thm:Spec_of_Sphere_Bundle_of_Hyp_Space}The closures of $\schroedingerOperator_{\magneticFieldStrength}^{S\hyperbolicSpace}$
and $\schroedingerOperator_{\magneticFieldStrength}^{\overline{SL(2,\mathbb{R})}}$
have spectra
\begin{eqnarray}
\spectrum{\closure{\schroedingerOperator_{\magneticFieldStrength}^{S\hyperbolicSpace}}} & = & \purePointSpectrum{\closure{\closure{\schroedingerOperator_{\magneticFieldStrength}^{S\hyperbolicSpace}}}}\cup[\groundStateEnergyOfContinuousSpectrum(\closure{\closure{\schroedingerOperator_{\magneticFieldStrength}^{S\hyperbolicSpace}}}),\infty)\label{eq:Spec_of_Sphere_Bundle_of_Hyp_Space}\\
\spectrum{\closure{\schroedingerOperator_{\magneticFieldStrength}^{\overline{SL(2,\mathbb{R})}}}} & = & \continuousSpectrum{\closure{\schroedingerOperator_{\magneticFieldStrength}^{\overline{SL(2,\mathbb{R})}}}}=[\groundStateEnergy(\closure{\schroedingerOperator_{\magneticFieldStrength}^{\overline{SL(2,\mathbb{R})}}}),\infty),\label{eq:Spec_of_Univ_Cover_of_Sphere_Bundle_of_Hyp_Space}
\end{eqnarray}
where
\begin{eqnarray*}
\purePointSpectrum{\closure{\schroedingerOperator_{\magneticFieldStrength}^{S\hyperbolicSpace}}} & = & \bigcup_{m\in\mathbb{Z}\backslash\{0\}}\multipliedBy\bigcup_{0\leq k<|m|}\left\{ \frac{1}{2}\left((\magneticFieldStrength+m)^{2}+(2k+1)|m|-k(k+1)\right)\right\} \subset\essentialSpectrum{\closure{\schroedingerOperator_{\magneticFieldStrength}^{S\hyperbolicSpace}}}\\
\groundStateEnergyOfContinuousSpectrum(\closure{\closure{\schroedingerOperator_{\magneticFieldStrength}^{S\hyperbolicSpace}}}) & = & \frac{1}{2}\left(\lfloor|\magneticFieldStrength|\rfloor^{2}+(|\magneticFieldStrength|-\lfloor|\magneticFieldStrength|\rfloor)^{2}+\frac{1}{4}\right)\\
\groundStateEnergy(\closure{\closure{\schroedingerOperator_{\magneticFieldStrength}^{\overline{SL(2,\mathbb{R})}}}}) & = & \begin{cases}
\frac{1}{2}\left(\frac{1}{2}|\magneticFieldStrength|^{2}+\frac{1}{4}\right) & \textnormal{if }|\magneticFieldStrength|\leq1\\
\frac{1}{2}\left(|\magneticFieldStrength|-\frac{1}{4}\right) & \textnormal{if }|\magneticFieldStrength|>1.
\end{cases}
\end{eqnarray*}
In particular, 
\[
\groundStateEnergy(\closure{\closure{\schroedingerOperator_{\magneticFieldStrength}^{S\hyperbolicSpace}}})=\begin{cases}
\frac{1}{2}\left(|\magneticFieldStrength|^{2}+\frac{1}{4}\right) & \textnormal{if }|\magneticFieldStrength|\leq\frac{7}{8}\\
\frac{1}{2}\left(1+(1-|\magneticFieldStrength|)^{2}\right) & \textnormal{if }\frac{7}{8}<|\magneticFieldStrength|\leq1\\
\frac{1}{2}\left(\lfloor|\magneticFieldStrength|\rfloor+(|\magneticFieldStrength|-\lfloor|\magneticFieldStrength|\rfloor)^{2}\right) & \textnormal{if }|\magneticFieldStrength|>1,
\end{cases}
\]
which has local minima at $|\magneticFieldStrength|=\frac{7}{8}$,
and $\frac{1}{2}(|\magneticFieldStrength|-\frac{1}{4})\leq\groundStateEnergy(\closure{\schroedingerOperator_{\magneticFieldStrength}^{S\hyperbolicSpace}})\leq\frac{1}{2}|\magneticFieldStrength|$
for $|\magneticFieldStrength|\geq1$. \end{thm}
\begin{proof}
The operators $\schroedingerOperator_{\magneticFieldStrength}^{S\hyperbolicSpace}$
and $\schroedingerOperator_{\magneticFieldStrength}^{\overline{SL(2,\mathbb{R})}}$
allow for a partial diagonalization by means of discrete and continuous
Fourier transformation, respectively. Since we use both techniques
repeatedly in the following sections, we recall them in detail.

Let $D\subset\domain{\schroedingerOperator_{\magneticFieldStrength}^{S\hyperbolicSpace}}$
denote the set of functions which are finite linear combinations of
products $u\multipliedBy w$ with $u\in\smoothCompactlySupportedFunctions{\hyperbolicSpace}$
and $w\in\smoothFunctions{\mathbb{S}^{1}}$. According to~\cite[Theorem II.10]{ReedSimon1975},
the mapping $u\otimes w\mapsto u\multipliedBy w$ from $L^{2}(\hyperbolicSpace,y^{-2}\multipliedBy dx\multipliedBy dy)\otimes L^{2}(\mathbb{S}^{1},d\varphi)$
to $L^{2}(S\hyperbolicSpace)\simeq L^{2}(\hyperbolicSpace\times\mathbb{S}^{1},y^{-2}\multipliedBy dx\multipliedBy dy\multipliedBy d\varphi)$
extends to an isometry, which implies that $D$ is dense in $L^{2}(S\hyperbolicSpace)$.
Fourier analysis on $\mathbb{S}^{1}$ shows that the functions $(e_{m})_{m\in\mathbb{Z}}$
given by $e_{m}(\varphi)=\frac{1}{\sqrt{2\pi}}e^{im\varphi}$ yield
a complete orthonormal set of eigenfunctions of the symmetric operator
$\frac{1}{i}\frac{\partial}{\partial\varphi}$. Thus, $L^{2}(\mathbb{S}^{1},d\varphi)\simeq\bigoplus_{m\in\mathbb{Z}}\mathbb{C}\multipliedBy e_{m}$,
and therefore 
\[
L^{2}(S\hyperbolicSpace)\simeq\bigoplus_{m=-\infty}^{\infty}L_{m}\qquad\textnormal{with}\qquad L_{m}=L^{2}(\hyperbolicSpace)\otimes\mathbb{C}\multipliedBy e_{m}.
\]
The number $m$ is the quantum analogue of the classical angular momentum.
On each $D_{m}=D\cap L_{m}$, $\schroedingerOperator_{\magneticFieldStrength}^{S\hyperbolicSpace}$
reduces to a shifted version of the Maass Laplacian~(\ref{eq:Maass_Laplacian}),
namely,
\[
\schroedingerOperator_{\magneticFieldStrength,m}^{S\hyperbolicSpace}=\frac{1}{2}\left(\left(y\frac{1}{i}\frac{\partial}{\partial x}-m\right)^{2}+y^{2}\left(\frac{1}{i}\frac{\partial}{\partial y}\right)^{2}+(m+\magneticFieldStrength)^{2}\right)=\schroedingerOperator_{-m}^{\hyperbolicSpace}+\frac{1}{2}(m+\magneticFieldStrength)^{2}
\]
with domain $\smoothCompactlySupportedFunctions{\hyperbolicSpace}$.
From~(\ref{eq:Spec_of_Hyp_space}), we get
\[
\spectrum{\closure{\schroedingerOperator_{\magneticFieldStrength,m}^{S\hyperbolicSpace}}}=\begin{cases}
\left[\frac{1}{2}\left(\magneticFieldStrength^{2}+\frac{1}{4}\right),\infty\right) & \textnormal{if }m=0\\
\purePointSpectrum{\closure{\schroedingerOperator_{\magneticFieldStrength,m}^{S\hyperbolicSpace}}}\cup\left[\frac{1}{2}\left((\magneticFieldStrength+m)^{2}+|m|^{2}+\frac{1}{4}\right),\infty\right) & \textnormal{if }m\neq0,
\end{cases}
\]
where
\[
\purePointSpectrum{\closure{\schroedingerOperator_{\magneticFieldStrength,m}^{S\hyperbolicSpace}}}=\bigcup_{0\leq k<|m|}\left\{ \frac{1}{2}\left((\magneticFieldStrength+m)^{2}+(2k+1)|m|-k(k+1)\right)\right\} \subset\essentialSpectrum{\closure{\schroedingerOperator_{\magneticFieldStrength,m}^{S\hyperbolicSpace}}}.
\]
The claim~(\ref{eq:Spec_of_Sphere_Bundle_of_Hyp_Space}) now follows
from 
\[
\spectrum{\closure{\schroedingerOperator_{\magneticFieldStrength}^{S\hyperbolicSpace}}}=\overline{\bigcup_{m\in\mathbb{Z}}\spectrum{\closure{\schroedingerOperator_{\magneticFieldStrength,m}^{S\hyperbolicSpace}}}}.
\]
In order to determine $\spectrum{\closure{\schroedingerOperator_{\magneticFieldStrength}^{\overline{SL(2,\mathbb{R})}}}}$,
we use that 
\[
L^{2}(\overline{SL(2,\mathbb{R})})\simeq L^{2}(\hyperbolicSpace\times\mathbb{R},y^{-2}dx\multipliedBy dy\multipliedBy d\varphi)
\]
and consider the partial Fourier transformation
\[
\mathcal{F}_{\varphi}\colon L^{2}(\hyperbolicSpace\times\mathbb{R},y^{-2}dx\multipliedBy dy\multipliedBy d\varphi)\to L^{2}(\hyperbolicSpace\times\mathbb{R},y^{-2}dx\multipliedBy dy\multipliedBy d\momentumVariable{\varphi}{})
\]
 given by
\[
\mathcal{F}_{\varphi}u(x,y,\momentumVariable{\varphi}{})=\frac{1}{\sqrt{2\pi}}\intop_{-\infty}^{\infty}u(x,y,\varphi)\multipliedBy e^{-i\momentumVariable{\varphi}{}\varphi}\multipliedBy d\varphi.
\]
Let $D$ denote the linear span of $\smoothCompactlySupportedFunctions{\hyperbolicSpace}\times\mathcal{S}(\mathbb{R})$.
Using cut-off functions, dominated convergence and the canonical isometry
$L^{2}(\hyperbolicSpace\times\mathbb{R})\simeq L^{2}(\hyperbolicSpace)\otimes L^{2}(\mathbb{S}^{1})$~\cite[Theorem II.10]{ReedSimon1975},
one sees that $D$ is a core for $\closure{\schroedingerOperator_{\magneticFieldStrength}^{\overline{SL(2,\mathbb{R})}}}$.
Moreover, $\mathcal{F}_{\varphi}$ restricts to an isometry of $D$,
on which we have
\begin{equation}
\mathcal{F}_{\varphi}\multipliedBy\closure{\schroedingerOperator_{\magneticFieldStrength}^{\overline{SL(2,\mathbb{R})}}}\multipliedBy\mathcal{F}_{\varphi}^{-1}=\frac{1}{2}\left(\left(y\frac{1}{i}\frac{\partial}{\partial x}-\momentumVariable{\varphi}{}\right)^{2}+y^{2}\left(\frac{1}{i}\frac{\partial}{\partial y}\right)^{2}+\left(\momentumVariable{\varphi}{}+\magneticFieldStrength\right)^{2}\right).\label{eq:Fourier_Transformation_as_Direct_Integral}
\end{equation}
More generally, one can use the canonical isomorphism~\cite[Theorem II.10]{ReedSimon1980}
\[
\mathcal{I}\colon L^{2}(\hyperbolicSpace\times\mathbb{R},y^{-2}dx\multipliedBy dy\multipliedBy d\momentumVariable{\varphi}{})\to L^{2}(\mathbb{R},d\momentumVariable{\varphi}{},L^{2}(\hyperbolicSpace,y^{-2}dx\multipliedBy dy))
\]
given by $\mathcal{I}\hat{u}(\momentumVariable{\varphi}{})(x,y)=\hat{u}(x,y,\momentumVariable{\varphi}{})$
to reinterpret~(\ref{eq:Fourier_Transformation_as_Direct_Integral})
as a constant fibre direct integral decomposition over $L^{2}(\mathbb{R},d\momentumVariable{\varphi}{},L^{2}(\hyperbolicSpace,y^{-2}dx\multipliedBy dy))$~\cite[Section XIII.16]{ReedSimon1978}.
For each $\momentumVariable{\varphi}{}\in\mathbb{R}$, we define $\schroedingerOperator_{\magneticFieldStrength,\momentumVariable{\varphi}{}}^{\overline{SL(2,\mathbb{R})}}$
as the differential operator~(\ref{eq:Fourier_Transformation_as_Direct_Integral})
acting on $\smoothCompactlySupportedFunctions{\hyperbolicSpace}\subset L^{2}(\hyperbolicSpace,y^{-2}dx\multipliedBy dy\multipliedBy d\momentumVariable{\varphi}{})$.
According to~(\ref{eq:Spec_of_Hyp_space}), we have 
\[
\spectrum{\closure{\schroedingerOperator_{\magneticFieldStrength,\momentumVariable{\varphi}{}}^{\overline{SL(2,\mathbb{R})}}}}=\begin{cases}
\left[\groundStateEnergyOfContinuousSpectrum(\closure{\schroedingerOperator_{\magneticFieldStrength,\momentumVariable{\varphi}{}}^{\overline{SL(2,\mathbb{R})}}}),\infty\right) & \textnormal{if }|\momentumVariable{\varphi}{}|\leq\frac{1}{2}\\
\purePointSpectrum{\closure{\schroedingerOperator_{\magneticFieldStrength,\momentumVariable{\varphi}{}}^{\overline{SL(2,\mathbb{R})}}}}\cup\left[\groundStateEnergyOfContinuousSpectrum(\closure{\schroedingerOperator_{\magneticFieldStrength,\momentumVariable{\varphi}{}}^{\overline{SL(2,\mathbb{R})}}}),\infty\right) & \textnormal{if }|\momentumVariable{\varphi}{}|>\frac{1}{2},
\end{cases}
\]
where 
\[
\groundStateEnergyOfContinuousSpectrum(\closure{\schroedingerOperator_{\magneticFieldStrength,\momentumVariable{\varphi}{}}^{\overline{SL(2,\mathbb{R})}}})=\frac{1}{2}\left(\left(\magneticFieldStrength+\momentumVariable{\varphi}{}\right)^{2}+\momentumVariable{\varphi}2+\frac{1}{4}\right)=\left(\frac{1}{2}\magneticFieldStrength+\momentumVariable{\varphi}{}\right)^{2}+\frac{1}{2}\left(\frac{1}{2}\magneticFieldStrength^{2}+\frac{1}{4}\right)
\]
 and
\[
\purePointSpectrum{\closure{\schroedingerOperator_{\magneticFieldStrength,\momentumVariable{\varphi}{}}^{\overline{SL(2,\mathbb{R})}}}}=\bigcup_{0\leq k<|\momentumVariable{\varphi}{}|-\frac{1}{2}}\left\{ \frac{1}{2}\left(\left(\magneticFieldStrength+\momentumVariable{\varphi}{}\right)^{2}+(2k+1)|\momentumVariable{\varphi}{}|-k(k+1)\right)\right\} ,
\]
with infimum $\frac{1}{2}\left(|\magneticFieldStrength|-\frac{1}{4}\right)$
for $|\magneticFieldStrength|>1$ attained at $k=0$ and $\momentumVariable{\varphi}{}=\frac{\magneticFieldStrength}{|\magneticFieldStrength|}\left(\frac{1}{2}-|\magneticFieldStrength|\right)$.
Since the mappings $\momentumVariable{\varphi}{}\mapsto\groundStateEnergyOfContinuousSpectrum(\closure{\schroedingerOperator_{\magneticFieldStrength,\momentumVariable{\varphi}{}}^{\overline{SL(2,\mathbb{R})}}})$
and $\momentumVariable{\varphi}{}\mapsto\frac{1}{2}((\magneticFieldStrength+\momentumVariable{\varphi}{})^{2}+(2k+1)|\momentumVariable{\varphi}{}|-k(k+1))$
are continuous, we can use~\cite[Theorem XIII.85]{ReedSimon1978}
to obtain
\[
\spectrum{\closure{\schroedingerOperator_{\magneticFieldStrength}^{\overline{SL(2,\mathbb{R})}}}}=\bigcup_{\momentumVariable{\varphi}{}\in\mathbb{R}}\spectrum{\closure{\schroedingerOperator_{\magneticFieldStrength,\momentumVariable{\varphi}{}}^{\overline{SL(2,\mathbb{R})}}}}.
\]
Moreover, $\lambda\in\purePointSpectrum{\closure{\schroedingerOperator_{\magneticFieldStrength}^{\overline{SL(2,\mathbb{R})}}}}$
if and only if $\{\momentumVariable{\varphi}{}\in\mathbb{R}\,\big|\,\lambda\in\purePointSpectrum{\closure{\schroedingerOperator_{\magneticFieldStrength,\momentumVariable{\varphi}{}}^{\overline{SL(2,\mathbb{R})}}}}\}$
has non-zero Lebesgue measure. Thus, $\purePointSpectrum{\closure{\schroedingerOperator_{\magneticFieldStrength}^{\overline{SL(2,\mathbb{R})}}}}=\varnothing$,
which completes the proof.
\end{proof}
It is worth mentioning, that Erd\"os~\cite[Section E, Remark 1]{ErdHos1997}
constructed a radially symmetric magnetic potential $\magneticPotential$
on the Euclidean space $\mathbb{R}^{n}$ such that $\groundStateEnergy(\magneticFieldStrength\multipliedBy\magneticPotential)$
is a non-monotone function of $\magneticFieldStrength$. However,
$\magneticPotential$ has non-constant norm whereas $\magneticPotential_{3}$
above is left-invariant. The spectrum of the Laplacian $\Delta^{S\hyperbolicSpace}=2\multipliedBy\closure{\schroedingerOperator_{0}^{S\hyperbolicSpace}}$
should be compared with~\cite[Example A]{Sunada1988}.

\subsection{Heisenberg group\label{sub:HG}}

The Heisenberg group $\heisenbergGroup$ is the semidirect product
$\mathbb{R}\ltimes_{\eta}\mathbb{R}^{2}$ with $\eta\colon\mathbb{R}\to\mathrm{Aut}(\mathbb{R}^{2})$
given by $\eta(x)\multipliedBy(y,z)=(y,x\multipliedBy y+z)$. In other
words, $\heisenbergGroup$ is $\mathbb{R}^{3}$ viewed as a nilpotent
Lie group with multiplication
\[
(x,y,z)\multipliedBy(x',y',z')=(x+x',y+y',z+z'+xy')
\]
coming from the matrix representation
\[
\left(\begin{array}{ccc}
1 & x & z\\
0 & 1 & y\\
0 & 0 & 1
\end{array}\right).
\]
Following~\cite{CieliebakFrauenfelderPaternain2010}, we consider
the cocompact lattice $\Lambda=\mathbb{Z}\ltimes_{\eta}\mathbb{Z}^{2}$
of matrices with $x,y,z\in\mathbb{Z}$. All cocompact lattices of
$\heisenbergGroup$ are isomorphic to $\Lambda$~\cite{Scott1983}.
The left-invariant $1$-forms 
\[
\magneticPotential_{1}=dx\hspace{1cm}\magneticPotential_{2}=dy\hspace{1cm}\magneticPotential_{3}=dz-x\multipliedBy dy
\]
give rise to the left-invariant metric
\[
ds^{2}=\magneticPotential_{1}^{2}+\magneticPotential_{2}^{2}+\magneticPotential_{3}^{2}=dx^{2}+dy^{2}+(dz-x\multipliedBy dy)^{2},
\]
that is,
\[
g=\left(\begin{array}{ccc}
1 & 0 & 0\\
0 & 1+x^{2} & -x\\
0 & -x & 1
\end{array}\right)\hspace{5mm}\sqrt{|g|}=1\hspace{5mm}g^{-1}=\left(\begin{array}{ccc}
1 & 0 & 0\\
0 & 1 & x\\
0 & x & 1+x^{2}
\end{array}\right).
\]
We let $\magneticFieldStrength\in\mathbb{R}$ and regard $\magneticFieldStrength\magneticPotential_{3}$
as a left-invariant magnetic potential. The corresponding Hamiltonian
$H_{\magneticFieldStrength}^{\Lambda\backslash\heisenbergGroup}\colon T^{*}(\Lambda\backslash\heisenbergGroup)\to\mathbb{R}$
and its lift $H_{\magneticFieldStrength}^{\heisenbergGroup}\colon T^{*}\heisenbergGroup\to\mathbb{R}$
given by 
\[
H_{\magneticFieldStrength}^{\heisenbergGroup}(x,y,z,p_{x},p_{y},p_{z})=\frac{1}{2}\left(p_{x}^{2}+((p_{y}-\magneticFieldStrength\multipliedBy x)+x\multipliedBy(p_{z}+\magneticFieldStrength))^{2}+(p_{z}+\magneticFieldStrength)^{2}\right)
\]
have identical critical values~\cite[Section 6.3]{CieliebakFrauenfelderPaternain2010}
\[
\manesCriticalValue(H_{\magneticFieldStrength}^{\Lambda\backslash\heisenbergGroup})=\manesCriticalValue(H_{\magneticFieldStrength}^{\heisenbergGroup})=\frac{1}{2}\magneticFieldStrength^{2}
\]
since $\pi_{1}(\Lambda\backslash\heisenbergGroup)\simeq\Lambda$ is
nilpotent and therefore amenable. The maximal abelian cover of $\Lambda\backslash\heisenbergGroup$
is $[\Lambda,\Lambda]\backslash\heisenbergGroup$, where $[\Lambda,\Lambda]=\{0\}\times\{0\}\times\mathbb{Z}$.
In the following, we study the corresponding magnetic Schr\"odinger
operator $\schroedingerOperator_{\magneticFieldStrength}^{[\Lambda,\Lambda]\backslash\heisenbergGroup}$
acting as 
\[
\schroedingerOperator_{\magneticFieldStrength}^{[\Lambda,\Lambda]\backslash\heisenbergGroup}=\frac{1}{2}\left(\left(\frac{1}{i}\frac{\partial}{\partial x}\right)^{2}+\left(\frac{1}{i}\frac{\partial}{\partial y}+x\multipliedBy\frac{1}{i}\frac{\partial}{\partial z}\right)^{2}+\left(\frac{1}{i}\frac{\partial}{\partial z}+\magneticFieldStrength\right)^{2}\right)
\]
on $\smoothCompactlySupportedFunctions{\mathbb{R}^{2}\times\mathbb{Z}\backslash\mathbb{R}}\subset L^{2}(\mathbb{R}^{2}\times\mathbb{Z}\backslash\mathbb{R})\simeq L^{2}([\Lambda,\Lambda]\backslash\heisenbergGroup)$,
and its lift $\schroedingerOperator_{\magneticFieldStrength}^{\heisenbergGroup}$
to $\heisenbergGroup$ with domain $\smoothCompactlySupportedFunctions{\mathbb{R}^{3}}\subset L^{2}(\mathbb{R}^{3})\simeq L^{2}(\heisenbergGroup)$.
Figure~\ref{fig:Plot_of_HG} visualizes the following thoerem.

\begin{figure}
\noindent \begin{centering}
\psfrag{0}{}
\psfrag{6}{}
\psfrag{1}{\hspace{0mm}\raisebox{-0.8mm}{\scriptsize{$5$}}} \psfrag{2}{\hspace{-1mm}\raisebox{-0.8mm}{\scriptsize{$10$}}} \psfrag{3}{\hspace{-1mm}\raisebox{-0.8mm}{\scriptsize{$15$}}} \psfrag{4}{\hspace{-1mm}\raisebox{-0.8mm}{\scriptsize{$20$}}} \psfrag{5}{\hspace{-1mm}\raisebox{-0.8mm}{\scriptsize{$25$}}}\psfrag{B}{\hspace{-2mm}\raisebox{-3.5mm}{$|\magneticFieldStrength|$}}\psfrag{a}{$\groundStateEnergy(\closure{\schroedingerOperator_{\magneticFieldStrength}^{\heisenbergGroup}})$}\includegraphics[scale=0.75]{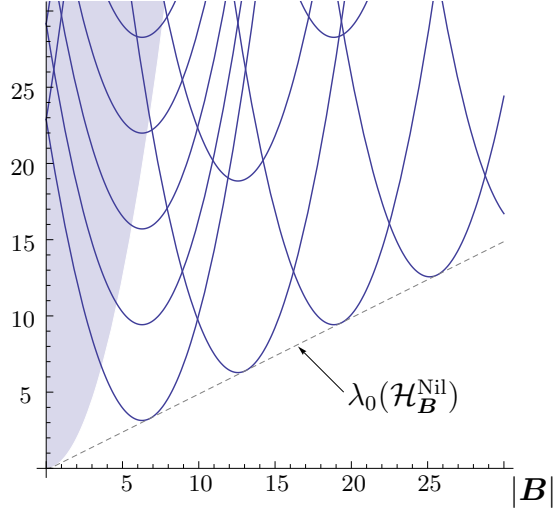}
\par\end{centering}

\caption{Plot of (\ref{eq:Spec_of_Abelian_Subvcover_of_HG}), where solid lines
indicate point spectrum and the shaded region indicates continuous
spectrum.\label{fig:Plot_of_HG}}
\end{figure}

\begin{thm}
The closures of $\schroedingerOperator_{\magneticFieldStrength}^{[\Lambda,\Lambda]\backslash\heisenbergGroup}$
and $\schroedingerOperator_{\magneticFieldStrength}^{\heisenbergGroup}$
have spectra
\begin{eqnarray}
\spectrum{\closure{\schroedingerOperator_{\magneticFieldStrength}^{[\Lambda,\Lambda]\backslash\heisenbergGroup}}} & = & \purePointSpectrum{\closure{\schroedingerOperator_{\magneticFieldStrength}^{[\Lambda,\Lambda]\backslash\heisenbergGroup}}}\cup[\groundStateEnergyOfContinuousSpectrum(\closure{\schroedingerOperator_{\magneticFieldStrength}^{[\Lambda,\Lambda]\backslash\heisenbergGroup}}),\infty)\label{eq:Spec_of_Abelian_Subvcover_of_HG}\\
\spectrum{\closure{\schroedingerOperator_{\magneticFieldStrength}^{\heisenbergGroup}}} & = & \continuousSpectrum{\closure{\schroedingerOperator_{\magneticFieldStrength}^{\heisenbergGroup}}}=[\groundStateEnergy(\closure{\schroedingerOperator_{\magneticFieldStrength}^{\heisenbergGroup}}),\infty),\label{eq:Spec_of_HG}
\end{eqnarray}
where
\begin{eqnarray*}
\purePointSpectrum{\closure{\schroedingerOperator_{\magneticFieldStrength}^{[\Lambda,\Lambda]\backslash\heisenbergGroup}}} & = & \bigcup_{m\in\mathbb{Z}\backslash\{0\}}\multipliedBy\bigcup_{k\in\mathbb{N}_{0}}\left\{ \frac{1}{2}\left((\magneticFieldStrength+2\pi m)^{2}+2\pi\multipliedBy(2k+1)|m|\right)\right\} \subset\essentialSpectrum{\closure{\schroedingerOperator_{\magneticFieldStrength}^{[\Lambda,\Lambda]\backslash\heisenbergGroup}}}\\
\groundStateEnergyOfContinuousSpectrum(\closure{\schroedingerOperator_{\magneticFieldStrength}^{[\Lambda,\Lambda]\backslash\heisenbergGroup}}) & = & \frac{1}{2}|\magneticFieldStrength|^{2}\\
\groundStateEnergy(\closure{\schroedingerOperator_{\magneticFieldStrength}^{\heisenbergGroup}}) & = & \begin{cases}
\frac{1}{2}|\magneticFieldStrength|^{2} & \textnormal{if }|\magneticFieldStrength|\leq\frac{1}{2}\\
\frac{1}{2}\left(|\magneticFieldStrength|-\frac{1}{4}\right) & \textnormal{if }|\magneticFieldStrength|>\frac{1}{2}.
\end{cases}
\end{eqnarray*}
In particular, the function $\magneticFieldStrength\mapsto\groundStateEnergy(\closure{\schroedingerOperator_{\magneticFieldStrength}^{[\Lambda,\Lambda]\backslash\heisenbergGroup}})$
has countably many local minima.\end{thm}
\begin{proof}
We mimic the proof of Theorem~\ref{thm:Spec_of_Sphere_Bundle_of_Hyp_Space}
and use that $\schroedingerOperator_{\magneticFieldStrength}^{[\Lambda,\Lambda]\backslash\heisenbergGroup}$
and $\schroedingerOperator_{\magneticFieldStrength}^{\heisenbergGroup}$
allow for discrete and continuous Fourier transformation in the $z$-coordinate,
respectively. One obtains shifted versions of the magnetic Schr\"odinger
operators in~(\ref{eq:Model_Operators_on_Rn}) acting on $\smoothCompactlySupportedFunctions{\mathbb{R}^{2}}$
as
\[
\schroedingerOperator_{\magneticFieldStrength,\momentumVariable z{}}^{[\Lambda,\Lambda]\backslash\heisenbergGroup}=\schroedingerOperator_{\magneticFieldStrength,\momentumVariable z{}}^{\heisenbergGroup}=\frac{1}{2}\left(\left(\frac{1}{i}\frac{\partial}{\partial x}\right)^{2}+\left(\frac{1}{i}\frac{\partial}{\partial y}+2\pi\momentumVariable z{}x\right)^{2}+\left(2\pi\momentumVariable z{}+\magneticFieldStrength\right)^{2}\right)
\]
with $\momentumVariable z{}\in\mathbb{Z}$ in the former case and
$\momentumVariable z{}\in\mathbb{R}$ in the latter case. According
to Theorem~\ref{thm:Spectrum_of_Tori}, 
\[
\spectrum{\closure{\schroedingerOperator_{\magneticFieldStrength,\momentumVariable z{}}^{\heisenbergGroup}}}=\begin{cases}
\mathrm{\left[\frac{1}{2}\magneticFieldStrength^{2},\infty\right)} & \textnormal{if }\momentumVariable z{}=0\\
\left\{ \frac{1}{2}\left((\magneticFieldStrength+2\pi\momentumVariable z{})^{2}+2\pi\multipliedBy(2k+1)|\momentumVariable z{}|\right)\,|\, k=0,1,2,\ldots\right\}  & \textnormal{if }\momentumVariable z{}\neq0,
\end{cases}
\]
with ground state energy 
\begin{equation}
\groundStateEnergy(\closure{\closure{\schroedingerOperator_{\magneticFieldStrength,\momentumVariable z{}}^{\heisenbergGroup}}})=\frac{1}{2}((\magneticFieldStrength+2\pi\momentumVariable z{})^{2}+2\pi|\momentumVariable z{}|).\label{eq:Ground_State_of_HG_with_Angular_Momentum}
\end{equation}
Using 
\[
\spectrum{\closure{\schroedingerOperator_{\magneticFieldStrength}^{[\Lambda,\Lambda]\backslash\heisenbergGroup}}}=\bigcup_{m\in\mathbb{Z}}\spectrum{\closure{\schroedingerOperator_{\magneticFieldStrength,m}^{[\Lambda,\Lambda]\backslash\heisenbergGroup}}}=\bigcup_{m\in\mathbb{Z}}\spectrum{\closure{\schroedingerOperator_{\magneticFieldStrength,m}^{\heisenbergGroup}}},
\]
we easily deduce (\ref{eq:Spec_of_Abelian_Subvcover_of_HG}). As for~(\ref{eq:Spec_of_HG}),
note that $\schroedingerOperator_{\magneticFieldStrength}^{\heisenbergGroup}$
is conjugate to a direct integral of model operators $\closure{\closure{\schroedingerOperator_{\magneticFieldStrength,\momentumVariable z{}}^{\heisenbergGroup}}}$
over $L^{2}(\mathbb{R},d\momentumVariable z{},L^{2}(\mathbb{R}^{2},dx\multipliedBy dy))$.
On the full measure set $\{\momentumVariable z{}\neq0\}=\mathbb{R}\backslash\{0\}\subset\mathbb{R}$,
the operators $\closure{\closure{\schroedingerOperator_{\magneticFieldStrength,\momentumVariable z{}}^{\heisenbergGroup}}}$
have pure point spectrum with nowhere constant eigenvalue functions
that depend continuously on $\momentumVariable z{}$, hence~\cite[Theorem XIII.85]{ReedSimon1978}
\[
\spectrum{\schroedingerOperator_{\magneticFieldStrength}^{\heisenbergGroup}}=\bigcup_{\momentumVariable z{}\neq0}\spectrum{\closure{\closure{\schroedingerOperator_{\magneticFieldStrength,\momentumVariable z{}}^{\heisenbergGroup}}}},
\]
and $\purePointSpectrum{\schroedingerOperator_{\magneticFieldStrength}^{\heisenbergGroup}}=\varnothing$.
The claim now follows by an inspection of~(\ref{eq:Ground_State_of_HG_with_Angular_Momentum}),
namely,
\[
\inf_{\momentumVariable z{}\neq0}\groundStateEnergy(\closure{\closure{\schroedingerOperator_{\magneticFieldStrength,\momentumVariable z{}}^{\heisenbergGroup}}})=\left\{ \begin{array}{rll}
\frac{1}{2}|\magneticFieldStrength|^{2} & \textnormal{ obtained for }\momentumVariable z{}\to0 & \textnormal{if }|\magneticFieldStrength|\leq\frac{1}{2}\\
\frac{1}{2}(|\magneticFieldStrength|-\frac{1}{4}) & \textnormal{ attained at }\momentumVariable z{}=\frac{1}{2\pi}\frac{\magneticFieldStrength}{|\magneticFieldStrength|}(|\magneticFieldStrength|-\frac{1}{4}) & \textnormal{if }|\magneticFieldStrength|>\frac{1}{2}.
\end{array}\right.
\]

\end{proof}

\subsection{Solvable geometry}

The Lie group $\solvableGeometry$ is the semidirect product $\solvableGeometry=\mathbb{R}^{2}\rtimes_{\eta}\mathbb{R}$
with $\eta\colon\mathbb{R}\to\mathrm{Aut}(\mathbb{R}^{2})$ given
by $\eta(z)\multipliedBy(x,y)=(e^{z}\multipliedBy x,e^{-z}\multipliedBy y)$.
In other words, $\solvableGeometry$ is the manifold $\mathbb{R}^{3}$
equipped with the multiplication
\[
(x,y,z)\multipliedBy(x',y',z')=(x+e^{z}x',y+e^{-z}y',z+z')
\]
coming from the matrix representation
\[
\left(\begin{array}{ccc}
e^{z} & 0 & x\\
0 & e^{-z} & y\\
0 & 0 & 1
\end{array}\right).
\]
The left-invariant $1$-forms 
\[
\magneticPotential_{x}=e^{-z}dx\hspace{1cm}\magneticPotential_{y}=e^{z}dy\hspace{1cm}\magneticPotential_{z}=dz
\]
give rise to the left-invariant metric 
\[
\mathrm{d}s^{2}=\magneticPotential_{x}^{2}+\magneticPotential_{y}^{2}+\magneticPotential_{z}^{2}=e^{-2z}dx^{2}+e^{2z}dy^{2}+dz^{2},
\]
in particular, $L^{2}(\solvableGeometry)=L^{2}(\mathbb{R}^{3},dx\multipliedBy dy\multipliedBy dz)$.
Following~\cite{ButlerPaternain2008}, we consider compact quotients
of $\solvableGeometry$ obtained from hyperbolic gluing maps of $\mathbb{T}^{2}$
as follows. Let $A\in SL(2,\mathbb{Z})$ have real eigenvalues $\lambda>1$
and $\lambda^{-1}<1$, and let $P\in GL(2,\mathbb{R})$ be such that
\begin{equation}
PAP^{-1}=\left(\begin{array}{cc}
\lambda & 0\\
0 & \lambda^{-1}
\end{array}\right).\label{eq:Hyperbolic_Gluing_Map_Conjugated}
\end{equation}
The image of the injective homomorphism
\begin{equation}
\mathbb{Z}^{2}\rtimes_{A}\mathbb{Z}\hookrightarrow\solvableGeometry\qquad\textnormal{given by }\left(\left(\begin{array}{c}
m\\
n
\end{array}\right),l\right)\mapsto\left(P\left(\begin{array}{c}
m\\
n
\end{array}\right),l\multipliedBy\log\lambda\right)\label{eq:Cocompact_Lattice_in_Sol}
\end{equation}
is a cocompact lattice $\Lambda_{A}$ in $\solvableGeometry$. The
closed $3$-manifold $\Lambda_{A}\backslash\solvableGeometry$ is
a torus bundle over $\mathbb{S}^{1}$. Its  harmonic monopole $dx\wedge dy$
generates $\homologyGroup^{2}(\Lambda_{A}\backslash\solvableGeometry,\mathbb{R})$
and is Hodge dual to the generator $dz$ of $\homologyGroup^{1}(\Lambda_{A}\backslash\solvableGeometry,\mathbb{R})$.
For the sake of convenience, we introduce 
\[
\derivative x=\frac{1}{i}\frac{\partial}{\partial x}\hspace{1cm}\derivative y=\frac{1}{i}\frac{\partial}{\partial y}\hspace{1cm}\derivative z=\frac{1}{i}\frac{\partial}{\partial z}.
\]
A simple integration by parts shows that $\derivative x$, $\derivative y$
and $\derivative z$ are symmetric on $\smoothCompactlySupportedFunctions{\mathbb{R}^{3}}$,
on which we have $\Delta^{\solvableGeometry}=e^{2z}\multipliedBy\derivative x^{2}+e^{-2z}\multipliedBy\derivative y^{2}+\derivative z^{2}$.
Since $\solvableGeometry$ is a globally symmetric space of non-compact
type, $\spectrum{\closure{\Delta^{\solvableGeometry}}}$ is known
to be purely continuous and of the form $[\groundStateEnergy(\closure{\Delta^{\solvableGeometry}}),\infty)$
for some $\groundStateEnergy(\closure{\Delta^{\solvableGeometry}})\geq0$,
see also~\cite{Sunada1988}. On the other hand, $\solvableGeometry$
is simply connected and has cocompact, solvable and therefore amenable
lattices, which yields~$\groundStateEnergy(\closure{\Delta^{\solvableGeometry}})=0$
by virtue of Brooks' theorem~\cite[Theorem 1]{Brooks1981}. Using
ideas from~\cite{InahamaShirai2003}, we give a more basic derivation
of $\spectrum{\closure{\Delta^{\solvableGeometry}}}=\continuousSpectrum{\closure{\Delta^{\solvableGeometry}}}=[0,\infty)$
in the following section.\vspace{-3mm}

\subsubsection{Exact case}

\global\long\def\hypSpaceFirstCoor{x_{\hyperbolicSpace}}
\global\long\def\hypSpaceSecondCoor{y_{\hyperbolicSpace}}
\global\long\def\realLineCoor{t}
Let $(\magneticFieldStrength_{x},\magneticFieldStrength_{y})\in\mathbb{R}^{2}$
and denote its norm by $\magneticFieldStrength=\sqrt{\magneticFieldStrength_{x}^{2}+\magneticFieldStrength_{y}^{2}}$.
We consider the left-invariant magnetic potential $\magneticFieldStrength_{x}\magneticPotential_{x}+\magneticFieldStrength_{y}\magneticPotential_{y}=\magneticFieldStrength_{x}e^{-z}dx+\magneticFieldStrength_{y}e^{z}dy$
with associated Hamiltonian
\[
H_{\magneticFieldStrength_{x},\magneticFieldStrength_{y}}^{\solvableGeometry}(x,y,z,p_{x},p_{y},p_{z})=\frac{1}{2}\left((e^{z}\multipliedBy p_{x}+\magneticFieldStrength_{x})^{2}+(e^{-z}\multipliedBy p_{y}+\magneticFieldStrength_{y}){}^{2}+p_{z}^{2}\right).
\]
Note that $H_{\magneticFieldStrength_{x},\magneticFieldStrength_{y}}^{\solvableGeometry}$
descends to any quotient of $\solvableGeometry$. Using the same arguments
as in~\cite[Section 3.1]{MacariniPaternain2010}, one sees that any
cocompact lattice $\Lambda_{A}\subset\solvableGeometry$ as in (\ref{eq:Cocompact_Lattice_in_Sol})
satisfies 
\[
\manesCriticalValue(H_{\magneticFieldStrength_{x},\magneticFieldStrength_{y}}^{\Lambda_{A}\backslash\solvableGeometry})=\manesCriticalValue(H_{\magneticFieldStrength_{x},\magneticFieldStrength_{y}}^{\solvableGeometry})=\frac{1}{2}\magneticFieldStrength^{2}.
\]
For the case $\magneticFieldStrength_{y}=0$, Macarini and Schlenk~\cite[Proposition 7.1]{MacariniSchlenk2011}
discovered that the magnetic flow on $(H_{\magneticFieldStrength_{x},0}^{\Lambda_{A}\backslash\solvableGeometry})^{-1}(k)$
has non-vanishing topological entropy if and only if $k>\manesCriticalValue(H_{\magneticFieldStrength_{x},0}^{\solvableGeometry})$.
For $\magneticFieldStrength\leq\frac{1}{2}$, the critical energy
$\manesCriticalValue(H_{\magneticFieldStrength_{x},\magneticFieldStrength_{y}}^{\solvableGeometry})$
turns out to coincide with the ground state energy of the corresponding
magnetic Schr\"odinger operator, which is given as
\begin{equation}
\schroedingerOperator_{\magneticFieldStrength_{x},\magneticFieldStrength_{y}}^{\solvableGeometry}=\frac{1}{2}\left((e^{z}\multipliedBy\derivative x+\magneticFieldStrength_{x})^{2}+(e^{-z}\multipliedBy\derivative y+\magneticFieldStrength_{y}){}^{2}+\derivative z^{2}\right)\label{eq:EMSO_on_Sol}
\end{equation}
on its initial domain $C_{0}^{\infty}(\mathbb{R}^{3})$.
\begin{thm}
\label{thm:EMSO_on_Sol}The spectrum of $\closure{\schroedingerOperator_{\magneticFieldStrength_{x},\magneticFieldStrength_{y}}^{\solvableGeometry}}$
depends on $\magneticFieldStrength=\sqrt{\magneticFieldStrength_{x}^{2}+\magneticFieldStrength_{y}^{2}}$
as follows:
\begin{enumerate}
\item We have $\spectrum{\closure{\schroedingerOperator_{\magneticFieldStrength_{x},\magneticFieldStrength_{y}}^{\solvableGeometry}}}=\spectrum{\closure{\schroedingerOperator_{\magneticFieldStrength_{y},\magneticFieldStrength_{x}}^{\solvableGeometry}}}\supseteq[\frac{1}{2}\magneticFieldStrength^{2},\infty)$.
\item If $0\leq\magneticFieldStrength\leq\frac{1}{2}$\textup{,} then equality
holds in 1., that is,\textup{ $\spectrum{\closure{\schroedingerOperator_{\magneticFieldStrength_{x},\magneticFieldStrength_{y}}^{\solvableGeometry}}}=[\frac{1}{2}\magneticFieldStrength^{2},\infty)$.}
\item If $\magneticFieldStrength>\frac{1}{2}$, then we have $\groundStateEnergy(\closure{\schroedingerOperator_{\magneticFieldStrength_{x},\magneticFieldStrength_{y}}^{\solvableGeometry}})\geq\frac{1}{2}\left(\magneticFieldStrength-\frac{1}{4}\right)$
with equality if also $\magneticFieldStrength_{y}=0$.
\item If $|\magneticFieldStrength_{x}|>\frac{1}{2}$, then $\frac{1}{2}(|\magneticFieldStrength_{x}|+|\magneticFieldStrength_{y}|^{2}-\frac{1}{4})\in\spectrum{\closure{\schroedingerOperator_{\magneticFieldStrength_{x},\magneticFieldStrength_{y}}^{\solvableGeometry}}}$.
\end{enumerate}
\end{thm}
\begin{proof}
[Proof of (1).]\global\long\def\Exponent{c}
Since $\schroedingerOperator_{\magneticFieldStrength_{x},\magneticFieldStrength_{y}}^{\solvableGeometry}$
and $\schroedingerOperator_{\magneticFieldStrength_{y},\magneticFieldStrength_{x}}^{\solvableGeometry}$
are conjugate to each other with respect to the idempotent isometry
$\mathcal{I}\colon L^{2}(\solvableGeometry)\to L^{2}(\solvableGeometry)$
given by $\mathcal{I}u(x,y,z)=u(y,x,-z)$, their spectra coincide.
In order to verify $\spectrum{\closure{\schroedingerOperator_{\magneticFieldStrength_{x},\magneticFieldStrength_{y}}^{\solvableGeometry}}}\supseteq\left[\frac{1}{2}\magneticFieldStrength^{2},\infty\right)$,
it suffices to show that for any $\kappa\geq0$ there exists a Weyl
sequence $(u_{n})_{n\in\mathbb{N}}$ with $u_{n}\in C_{0}^{\infty}(\mathbb{R}^{3})$
such that
\[
\normComingFromInnerProduct{u_{n}}_{L^{2}(\mathbb{R}^{3})}\to\sqrt{\pi}\hspace{1cm}\textnormal{and}\hspace{1cm}\normComingFromInnerProduct{(2\multipliedBy\schroedingerOperator_{\magneticFieldStrength_{x},\magneticFieldStrength_{y}}^{\solvableGeometry}-\magneticFieldStrength^{2}-\kappa^{2})u_{n}}_{L^{2}(\mathbb{R}^{3})}\to0.
\]
We consider products of the form $u_{n}(x,y,z)=\chi_{n}(x)\multipliedBy\psi_{n}(y)\multipliedBy\zeta_{n}(z)\multipliedBy\gamma_{n}(x,y)$
with $\chi_{n},\psi_{n},\zeta_{n}\in\smoothCompactlySupportedFunctions{\mathbb{R}}$
and $\gamma_{n}\in\smoothFunctions{\mathbb{R}^{2}}$. The classical
Poincar\'e inequality~\cite[Section 5.8.1]{Evans2010} implies that
for each compact set $K\subset\mathbb{R}$ there exists $C_{K}>0$
such that if $\zeta_{n}$ is supported inside $K$, we have
\[
\normComingFromInnerProduct{\zeta_{n}}_{L^{2}(\mathbb{R})}\leq C_{K}\normComingFromInnerProduct{\derivative z^{2}\zeta_{n}}_{L^{2}(\mathbb{R})}.
\]
Hence, we choose $(\chi_{n},\psi_{n},\zeta_{n})_{n\in\mathbb{N}}$
with growing supports. In order to motivate the choices below, we
conjugate~(\ref{eq:EMSO_on_Sol}) by the partial Fourier transformation
$\mathcal{F}_{x,y}\colon L^{2}(\mathbb{R}^{3})\to L^{2}(\mathbb{R}^{3})$
given by
\[
\mathcal{F}_{x,y}u(\momentumVariable x{},\momentumVariable y{},z)=\frac{1}{2\pi}\intop_{-\infty}^{\infty}\intop_{-\infty}^{\infty}u(x,y,z)\multipliedBy e^{-i(\momentumVariable x{}x+\momentumVariable y{}y)}\multipliedBy dx\multipliedBy dy
\]
to obtain that on $\mathcal{F}_{x,y}(\smoothCompactlySupportedFunctions{\mathbb{R}^{3}})$
\[
\mathcal{F}_{x,y}\left(2\multipliedBy\schroedingerOperator_{\magneticFieldStrength_{x},\magneticFieldStrength_{y}}^{\solvableGeometry}-\magneticFieldStrength^{2}\right)\mathcal{F}_{x,y}^{-1}=\derivative z^{2}+\momentumVariable x2e^{2z}+2\magneticFieldStrength_{x}\momentumVariable x{}e^{z}+\momentumVariable y2e^{-2z}+2\magneticFieldStrength_{y}\momentumVariable y{}e^{-z}.
\]
This suggests to choose $(\gamma_{n})_{n\in\mathbb{N}}$ such that
the transforms $(\mathcal{F}_{x,y}\multipliedBy\zeta_{n}\multipliedBy\gamma_{n}\multipliedBy\mathcal{F}_{x,y}^{-1})_{n\in\mathbb{N}}$
concentrate around the line $\{0\}\times\{0\}\times\mathbb{R}$ fast
enough to compensate for the growing support of $(\zeta_{n})_{n\in\mathbb{N}}$
and the resulting growth of the factors $(e^{k\multipliedBy z}\zeta_{n})_{n\in\mathbb{N}}$
for $k=-2,-1,1,2$. We control these concentration and growth rates
by real sequences $(c_{n})_{n\in\mathbb{N}}$ and $(s_{n})_{n\in\mathbb{N}}$
with $c_{n}\searrow0$ and $s_{n}\nearrow\infty$. First, we choose
$(\chi_{n})_{n\in\mathbb{N}}$ as smooth real-valued cut-off functions
on $\mathbb{R}$ with 
\begin{equation}
0\leq\chi_{n}\leq1\hspace{1cm}\chi_{n}\equiv1\textnormal{ on }[-s_{n},s_{n}]\hspace{1cm}\chi_{n}\equiv0\textnormal{ outside }[-(s_{n}+1),(s_{n}+1)],\label{eq:Smooth_Cutoff_Functions}
\end{equation}
and such that the first and second derivatives are uniformly bounded
by some $C>0$, that is, 
\[
\sup_{n\in\mathbb{N}}\max\left\{ \normComingFromInnerProduct{\chi_{n}^{'}}_{L^{\infty}(\mathbb{R})},\normComingFromInnerProduct{\chi_{n}^{''}}_{L^{\infty}(\mathbb{R})}\right\} \leq C.
\]
We let $(\psi_{n})_{n\in\mathbb{N}}=(\chi_{n})_{n\in\mathbb{N}}$
denote the same sequence in $\smoothCompactlySupportedFunctions{\mathbb{R}}$
with the implicit understanding that $\chi_{n}$ and $\psi_{n}$ will
henceforth be regarded as functions of $x$ and $y$, respectively.
Moreover, we choose some $\upsilon\in\smoothCompactlySupportedFunctions{\mathbb{R},\mathbb{R}}$
with norm $\normComingFromInnerProduct{\upsilon}_{L^{2}(\mathbb{R})}=1$
and support in $\left(\frac{1}{2},1\right)$ to define 
\[
\zeta_{n}(z)=2^{-\frac{n}{2}}e^{i\kappa z}\upsilon(2^{-n}z),\quad\textnormal{ and let }\quad\gamma_{n}(x,y)=\Exponent_{n}\multipliedBy e^{-\frac{1}{2}\Exponent_{n}^{2}(x^{2}+y^{2})}.
\]
Note that
\[
u_{n}(x,y,z)=2^{-\frac{n}{2}}\Exponent_{n}\multipliedBy\chi_{n}(x)\multipliedBy\psi_{n}(y)\multipliedBy e^{i\kappa z}\multipliedBy e^{-\frac{1}{2}\Exponent_{n}^{2}(x^{2}+y^{2})}\upsilon(2^{-n}z)
\]
is supported in $\mathbb{R}^{2}\times(2^{n-1},2^{n})$, for which
reason $(u_{n})_{n\in\mathbb{N}}$ is an orthogonal sequence. If we
require $c_{n}s_{n}\to\infty$, then the dominated convergence theorem
yields 
\begin{eqnarray}
\normComingFromInnerProduct{u_{n}}^{2} & = & \intop_{-(s_{n}+1)}^{s_{n}+1}\chi_{n}^{2}(x)\multipliedBy e^{-\Exponent_{n}^{2}x^{2}}\Exponent_{n}\mathrm{d}x\multipliedBy\intop_{-(s_{n}+1)}^{s_{n}+1}\psi_{n}^{2}(y)\multipliedBy e^{-\Exponent_{n}^{2}y^{2}}\Exponent_{n}\mathrm{d}y\multipliedBy\intop_{-\infty}^{\infty}\upsilon^{2}(2^{-n}z)\multipliedBy2^{-n}\mathrm{d}z\nonumber \\
 & = & \left(\intop_{-c_{n}(s_{n}+1)}^{c_{n}(s_{n}+1)}\chi_{n}^{2}(\Exponent_{n}^{-1}t)\multipliedBy e^{-t^{2}}\mathrm{d}t\right)^{2}\to\pi.\label{eq:Norm_of_Quasi_Modes_for_EMSO_on_Sol}
\end{eqnarray}
One easily computes that
\begin{equation}
\derivative x^{2}u_{n}(x,y,z)=\psi_{n}(y)\multipliedBy\zeta_{n}(z)\multipliedBy\gamma_{n}(x,y)\left((\Exponent_{n}^{2}-\Exponent_{n}^{4}x^{2})\multipliedBy\chi_{n}(x)+2\multipliedBy\Exponent_{n}^{2}\multipliedBy x\multipliedBy\chi_{n}^{'}(x)-\chi_{n}^{''}(x)\right).\label{eq:Norm_of_exp_two_z_DxDx_f}
\end{equation}
The first summand equals $(\Exponent_{n}^{2}-\Exponent_{n}^{4}x^{2})u_{n}(x,y,z)$,
and we obtain
\begin{eqnarray}
\normComingFromInnerProduct{e^{2z}(\Exponent_{n}^{2}-\Exponent_{n}^{4}x^{2})u_{n}}^{2} & \leq & \intop_{-\infty}^{\infty}(\Exponent_{n}^{2}-\Exponent_{n}^{4}x^{2})^{2}e^{-\Exponent_{n}^{2}x^{2}}\Exponent_{n}\mathrm{d}x\intop_{-\infty}^{\infty}e^{-\Exponent_{n}^{2}y^{2}}\Exponent_{n}\mathrm{d}y\intop_{-\infty}^{\infty}e^{4z}\upsilon^{2}(2^{-n}z)\multipliedBy2^{-n}\mathrm{d}z\nonumber \\
 & \leq & \sqrt{\pi}\multipliedBy\Exponent_{n}^{4}\intop_{-\infty}^{\infty}(1-t^{2})^{2}e^{-t^{2}}\mathrm{d}t\sup_{z\in[2^{n-1},2^{n}]}|e^{4z}|\leq\frac{3}{4}\multipliedBy\pi\multipliedBy\Exponent_{n}^{4}\multipliedBy e^{2^{n+2}}.\label{eq:Norm_of_exp_two_z_DxDx_f2}
\end{eqnarray}
As for the second summand, note that $\chi_{n}^{'}\equiv0$ outside
$I_{n}=[-(s_{n}+1),-s_{n}]\cup[s_{n},s_{n}+1]$, so 
\begin{eqnarray}
\normComingFromInnerProduct{\Exponent_{n}^{2}\multipliedBy x\multipliedBy e^{2z}\chi_{n}^{'}\multipliedBy\psi_{n}\multipliedBy\zeta_{n}\multipliedBy\gamma_{n}}^{2} & \leq & \Exponent_{n}^{2}\multipliedBy e^{2^{n+2}}\intop_{I_{n}}(\Exponent_{n}\multipliedBy x\multipliedBy\chi_{n}^{'}(x))^{2}e^{-\Exponent_{n}^{2}x^{2}}\Exponent_{n}\mathrm{d}x\intop_{-\infty}^{\infty}e^{-\Exponent_{n}^{2}y^{2}}\Exponent_{n}\mathrm{d}y\nonumber \\
 & \leq & 2\sqrt{\pi}\multipliedBy C^{2}\multipliedBy\Exponent_{n}^{2}\multipliedBy e^{2^{n+2}}\intop_{c_{n}s_{n}}^{c_{n}(s_{n}+1)}t^{2}e^{-t^{2}}\mathrm{d}t,\label{eq:Norm_of_exp_two_z_DxDx_f3}
\end{eqnarray}
and similarly for $e^{2z}\chi_{n}^{''}\multipliedBy\psi_{n}\multipliedBy\zeta_{n}\multipliedBy\gamma_{n}$.
We set $\Exponent_{n}=e^{-3^{n}}$ and $s_{n}=e^{4^{n}}$ to satisfy
$c_{n}s_{n}\to\infty$ and $\Exponent_{n}^{2}\multipliedBy e^{2^{n+2}}\to0$,
in particular, $\normComingFromInnerProduct{e^{2z}\derivative x^{2}u_{n}}_{L^{2}(\mathbb{R}^{3})}\to0$.
Similar computations show that 
\[
\normComingFromInnerProduct{e^{z}\derivative xu_{n}}_{L^{2}(\mathbb{R}^{3})}\to0\hspace{1cm}\normComingFromInnerProduct{e^{-2z}\derivative y^{2}u_{n}}_{L^{2}(\mathbb{R}^{3})}\to0\hspace{1cm}\normComingFromInnerProduct{e^{-z}\derivative yu_{n}}_{L^{2}(\mathbb{R}^{3})}\to0.
\]
Since 
\[
\derivative z\zeta_{n}(z)=2^{-\frac{n}{2}}e^{i\kappa z}(\kappa\multipliedBy\upsilon(2^{-n}z)-2^{-n}i\multipliedBy\upsilon'(2^{-n}z)),
\]
we have
\[
\derivative z^{2}\zeta_{n}(z)=2^{-\frac{n}{2}}e^{i\kappa z}(\kappa^{2}\upsilon(2^{-n}z)-2^{-(n-1)}i\multipliedBy\kappa\multipliedBy\upsilon'(2^{-n}z)-2^{-2n}\upsilon''(2^{-n}z)).
\]
A simple change of variables leads to
\begin{eqnarray*}
\left\Vert (\derivative z^{2}-\kappa^{2})u_{n}\right\Vert _{L^{2}(\mathbb{R}^{3})}^{2} & \leq & \intop_{-\infty}^{\infty}e^{-\Exponent_{n}^{2}x^{2}}\Exponent_{n}\mathrm{d}x\multipliedBy\intop_{-\infty}^{\infty}e^{-\Exponent_{n}^{2}y^{2}}\Exponent_{n}\mathrm{d}y\multipliedBy\intop_{-\infty}^{\infty}\left|2^{-(n-1)}i\multipliedBy\kappa\multipliedBy\upsilon'(t)+2^{-2n}\upsilon''(t)\right|^{2}\mathrm{d}t\\
 & = & \pi\left(2^{-2n+2}\kappa^{2}\normComingFromInnerProduct{\upsilon'}_{L^{2}(\mathbb{R})}^{2}+2^{-4n}\normComingFromInnerProduct{\upsilon''}_{L^{2}(\mathbb{R})}^{2}\right)\to0.
\end{eqnarray*}

\emph{Proof of (2).} In order to show that $\groundStateEnergy(\closure{\schroedingerOperator_{\magneticFieldStrength_{x},\magneticFieldStrength_{y}}^{\solvableGeometry}})\geq\frac{1}{2}\magneticFieldStrength^{2}$
for $\magneticFieldStrength\leq\frac{1}{2}$, we define creation and
annihilation operators acting on $C_{0}^{\infty}(\mathbb{R}^{3})$
as follows
\begin{eqnarray*}
\mathcal{A}_{x}=e^{z}\multipliedBy\derivative x-2\multipliedBy i\multipliedBy\magneticFieldStrength_{x}\multipliedBy\derivative z & \qquad & \mathcal{A}_{y}=e^{-z}\multipliedBy\derivative y+2\multipliedBy i\multipliedBy\magneticFieldStrength_{y}\multipliedBy\derivative z\\
\mathcal{A}_{x}^{\dagger}=e^{z}\multipliedBy\derivative x+2\multipliedBy i\multipliedBy\magneticFieldStrength_{x}\multipliedBy\derivative z & \qquad & \mathcal{A}_{y}^{\dagger}=e^{-z}\multipliedBy\derivative y-2\multipliedBy i\multipliedBy\magneticFieldStrength_{y}\multipliedBy\derivative z.
\end{eqnarray*}
As $\derivative x$, $\derivative y$ and $\derivative z$ are symmetric
on $C_{0}^{\infty}(\mathbb{R}^{3})$, we have $\adjoint{\mathcal{A}_{x}}u=\mathcal{A}_{x}^{\dagger}u$
and $\adjoint{\mathcal{A}_{y}}u=\mathcal{A}_{y}^{\dagger}u$ for any
$u\in C_{0}^{\infty}(\mathbb{R}^{3})$. One easily computes the following
operator identities on $C_{0}^{\infty}(\mathbb{R}^{3})$ 
\begin{eqnarray*}
\mathcal{A}_{x}^{\dagger}\multipliedBy\mathcal{A}_{x} & = & e^{2z}\multipliedBy\derivative x^{2}+2\multipliedBy\magneticFieldStrength_{x}\multipliedBy e^{z}\multipliedBy\derivative x+4\multipliedBy\magneticFieldStrength_{x}^{2}\multipliedBy\derivative z^{2}\\
\mathcal{A}_{y}^{\dagger}\multipliedBy\mathcal{A}_{y} & = & e^{-2z}\multipliedBy\derivative y^{2}+2\multipliedBy\magneticFieldStrength_{y}\multipliedBy e^{-z}\multipliedBy\derivative y+4\multipliedBy\magneticFieldStrength_{y}^{2}\multipliedBy\derivative z^{2},
\end{eqnarray*}
In particular, 
\[
2\multipliedBy\schroedingerOperator_{\magneticFieldStrength_{x},\magneticFieldStrength_{y}}^{\solvableGeometry}=\mathcal{A}_{x}^{\dagger}\multipliedBy\mathcal{A}_{x}+\mathcal{A}_{y}^{\dagger}\multipliedBy\mathcal{A}_{y}+(1-4\multipliedBy\magneticFieldStrength^{2})\multipliedBy\derivative z^{2}+\magneticFieldStrength^{2}.
\]
For $u\in\smoothCompactlySupportedFunctions{\mathbb{R}^{3}}$, we
obtain
\[
\innerProductThatAdapts{\left(2\multipliedBy\schroedingerOperator_{\magneticFieldStrength_{x},\magneticFieldStrength_{y}}^{\solvableGeometry}-\magneticFieldStrength^{2}\right)u}u_{L^{2}(\mathbb{R}^{3})}=\normComingFromInnerProduct{\mathcal{A}_{x}u}_{L^{2}(\mathbb{R}^{3})}^{2}+\normComingFromInnerProduct{\mathcal{A}_{y}u}_{L^{2}(\mathbb{R}^{3})}^{2}+(1-4\multipliedBy\magneticFieldStrength^{2})\normComingFromInnerProduct{\multipliedBy\derivative zu}_{L^{2}(\mathbb{R}^{3})}^{2},
\]
 which gives $\spectrum{\closure{\schroedingerOperator_{\magneticFieldStrength_{x},\magneticFieldStrength_{y}}^{\solvableGeometry}}}\subseteq\left[\frac{1}{2}\magneticFieldStrength^{2},\infty\right)$
for $\magneticFieldStrength^{2}\leq\frac{1}{4}$ by virtue of~\cite[Theorem 4.3.1]{Davies1995}.

\emph{Proof of (3).} Let $\magneticFieldStrength=\sqrt{\magneticFieldStrength_{x}^{2}+\magneticFieldStrength_{y}^{2}}>\frac{1}{2}$,
and choose $\varphi\in[0,2\pi)$ such that
\[
\magneticFieldStrength_{x}=\sin\varphi\multipliedBy\magneticFieldStrength\hspace{1cm}\text{and}\hspace{1cm}\magneticFieldStrength_{y}=\cos\varphi\multipliedBy\magneticFieldStrength.
\]
We consider further pairs of creation and annihilation operators acting
on $C_{0}^{\infty}(\mathbb{R}^{3})$ as
\begin{eqnarray*}
\mathcal{K}_{x}=e^{z}\multipliedBy\derivative x+\sin\varphi\multipliedBy\left(\magneticFieldStrength-\frac{1}{2}-i\multipliedBy\derivative z\right) & \qquad & \mathcal{K}_{y}=e^{-z}\multipliedBy\derivative y+\cos\varphi\multipliedBy\left(\magneticFieldStrength-\frac{1}{2}+i\multipliedBy\derivative z\right)\\
\mathcal{K}_{x}^{\dagger}=e^{z}\multipliedBy\derivative x+\sin\varphi\multipliedBy\left(\magneticFieldStrength-\frac{1}{2}+i\multipliedBy\derivative z\right) & \qquad & \mathcal{K}_{y}^{\dagger}=e^{-z}\multipliedBy\derivative y+\cos\varphi\multipliedBy\left(\magneticFieldStrength-\frac{1}{2}-i\multipliedBy\derivative z\right).
\end{eqnarray*}
On $C_{0}^{\infty}(\mathbb{R}^{3})$, we have $\adjoint{\mathcal{K}_{x}}=\mathcal{K}_{x}^{\dagger}$
and $\adjoint{\mathcal{K}_{y}}=\mathcal{K}_{y}^{\dagger}$. Moreover,
\begin{eqnarray*}
\mathcal{K}_{x}^{\dagger}\multipliedBy\mathcal{K}_{x} & = & (e^{z}\multipliedBy\derivative x+\sin\varphi\multipliedBy\magneticFieldStrength)^{2}+\sin^{2}\varphi\multipliedBy\left(\derivative z^{2}-\magneticFieldStrength+\frac{1}{4}\right)\\
\mathcal{K}_{y}^{\dagger}\multipliedBy\mathcal{K}_{y} & = & (e^{-z}\multipliedBy\derivative y+\cos\varphi\multipliedBy\magneticFieldStrength)^{2}+\cos^{2}\varphi\multipliedBy\left(\derivative z^{2}-\magneticFieldStrength+\frac{1}{4}\right).
\end{eqnarray*}
The claimed inequality now follows as in 2. from the following operator
identity 
\[
2\multipliedBy\schroedingerOperator_{\magneticFieldStrength_{x},\magneticFieldStrength_{y}}^{\solvableGeometry}=\mathcal{K}_{x}^{\dagger}\multipliedBy\mathcal{K}_{x}+\mathcal{K}_{y}^{\dagger}\multipliedBy\mathcal{K}_{y}+\magneticFieldStrength-\frac{1}{4}.
\]
The statement about the special case $\magneticFieldStrength_{y}=0$
is a consequence of 4., which is proven below.

\emph{Proof of (4).} We finally show that $\frac{1}{2}(|\magneticFieldStrength_{x}|+|\magneticFieldStrength_{y}|^{2}-\frac{1}{4})\in\spectrum{\closure{\schroedingerOperator_{\magneticFieldStrength_{x},\magneticFieldStrength_{y}}^{\solvableGeometry}}}$
whenever $|\magneticFieldStrength_{x}|>\frac{1}{2}$ by relating $\schroedingerOperator_{\magneticFieldStrength_{x},\magneticFieldStrength_{y}}^{\solvableGeometry}$
to the Maass Laplacian~(\ref{eq:Maass_Laplacian}). We consider $\mathbb{R}\times\hyperbolicSpace$
with coordinates $(\realLineCoor,\hypSpaceFirstCoor,\hypSpaceSecondCoor)$
and metric $d\realLineCoor^{2}+\hypSpaceSecondCoor^{-2}(d\hypSpaceFirstCoor^{2}+d\hypSpaceSecondCoor^{2})$,
and let $\mathcal{I}\colon L^{2}(\mathbb{R}^{3})\to L^{2}(\mathbb{R}\times\hyperbolicSpace)$
be the isometry given by
\begin{equation}
\mathcal{I}u(\realLineCoor,\hypSpaceFirstCoor,\hypSpaceSecondCoor)=\hypSpaceSecondCoor^{\frac{1}{2}}\multipliedBy u(\hypSpaceFirstCoor,\realLineCoor,\log\hypSpaceSecondCoor).\label{eq:Isometry_between_Sol_and_R_x_Hyp_Space}
\end{equation}
Similar canonical transformations appear in~\cite{Duru1983,Comtet1987,InahamaShirai2003}.
Note that $\mathcal{I}$ maps $\smoothCompactlySupportedFunctions{\mathbb{R}^{3}}$
onto $\smoothCompactlySupportedFunctions{\mathbb{R}\times\hyperbolicSpace}$,
and the inverse is given by
\[
\mathcal{I}^{-1}w(x,y,z)=e^{-\frac{1}{2}z}\multipliedBy w(y,x,e^{z}).
\]
Imitating the direct computation in~\cite{InahamaShirai2003}, we
obtain
\[
(\mathcal{I}\multipliedBy\derivative z\multipliedBy\mathcal{I}^{-1}w)(\realLineCoor,\hypSpaceFirstCoor,\hypSpaceSecondCoor)=\hypSpaceSecondCoor^{\frac{1}{2}}\multipliedBy(\derivative z\multipliedBy\mathcal{I}^{-1}w)(\hypSpaceFirstCoor,\realLineCoor,\log\hypSpaceSecondCoor)=\hypSpaceSecondCoor^{\frac{1}{2}}\multipliedBy\left(\hypSpaceSecondCoor^{\frac{1}{2}}\multipliedBy\derivative{\hypSpaceSecondCoor}+\frac{i}{2}\hypSpaceSecondCoor^{-\frac{1}{2}}\right)w\multipliedBy(\realLineCoor,\hypSpaceFirstCoor,\hypSpaceSecondCoor).
\]
Similarly, we obtain the following operator identities on $\smoothCompactlySupportedFunctions{\mathbb{R}\times\hyperbolicSpace}$
\[
\mathcal{I}\multipliedBy\derivative x\multipliedBy\mathcal{I}^{-1}=\derivative{\hypSpaceFirstCoor}\hspace{10mm}\mathcal{I}\multipliedBy\derivative y\multipliedBy\mathcal{I}^{-1}=\derivative{\realLineCoor}\hspace{10mm}\mathcal{I}\multipliedBy u(x,y,z)\multipliedBy\mathcal{I}^{-1}=u(\hypSpaceFirstCoor,\realLineCoor,\log\hypSpaceSecondCoor),
\]
where the last identity refers to multiplication by $u\in\smoothFunctions{\mathbb{R}^{3}}$.
Using that 
\[
\mathcal{I}\multipliedBy\derivative z^{2}\multipliedBy\mathcal{I}^{-1}=\left(\hypSpaceSecondCoor\multipliedBy\derivative{\hypSpaceSecondCoor}+\frac{i}{2}\right)^{2}=\hypSpaceSecondCoor^{2}\multipliedBy\derivative{\hypSpaceSecondCoor}^{2}-\frac{1}{4},
\]
we obtain from~(\ref{eq:EMSO_on_Sol}) 
\[
\mathcal{I}\multipliedBy\schroedingerOperator_{\magneticFieldStrength_{x},\magneticFieldStrength_{y}}^{\solvableGeometry}\multipliedBy\mathcal{I}^{-1}=\frac{1}{2}\left((\hypSpaceSecondCoor\multipliedBy\derivative{\hypSpaceFirstCoor}+\magneticFieldStrength_{x})^{2}+\hypSpaceSecondCoor^{2}\multipliedBy\derivative{\hypSpaceSecondCoor}^{2}\right)+\frac{1}{2}\left(\hypSpaceSecondCoor^{-1}\multipliedBy\derivative{\realLineCoor}+\magneticFieldStrength_{y}\right)^{2}-\frac{1}{8}.
\]
The first summand equals $\schroedingerOperator_{\magneticFieldStrength_{x}}^{\hyperbolicSpace}$
as in~(\ref{eq:Maass_Laplacian}) with spectrum given in~(\ref{eq:Spec_of_Hyp_space}).
In particular, $\groundStateEnergy(\schroedingerOperator_{\magneticFieldStrength_{x}}^{\hyperbolicSpace})=\frac{1}{2}\norm{\magneticFieldStrength_{x}}$
if $\norm{\magneticFieldStrength_{x}}>\frac{1}{2}$. Hence, there
exists a Weyl sequence $(w_{n})_{n\in\mathbb{N}}$ with $w_{n}\in\smoothCompactlySupportedFunctions{\hyperbolicSpace}$
and $\normComingFromInnerProduct{w_{n}}_{L^{2}(\hyperbolicSpace)}=1$
such that
\[
\normComingFromInnerProductThatAdapts{\left((\hypSpaceSecondCoor\multipliedBy\derivative{\hypSpaceFirstCoor}+\magneticFieldStrength_{x})^{2}+\hypSpaceSecondCoor^{2}\multipliedBy\derivative{\hypSpaceSecondCoor}^{2}-|\magneticFieldStrength_{x}|\right)w_{n}}_{L^{2}(\hyperbolicSpace)}\to0.
\]
We choose a real sequence $c_{n}\searrow0$ such that $w_{n}(\hypSpaceFirstCoor,\hypSpaceSecondCoor)=0$
if $\hypSpaceSecondCoor\leq c_{n}^{\frac{1}{2}}$, and let $(\chi_{n})_{n\in\mathbb{N}}$
be the sequence of cut-off functions given in~(\ref{eq:Smooth_Cutoff_Functions})
with~$s_{n}=c_{n}^{-2}$. The claim follows once we verified that
the Weyl sequence $(u_{n})_{n\in\mathbb{N}}$ with elements $u_{n}\in\smoothCompactlySupportedFunctions{\mathbb{R}\times\hyperbolicSpace}$
given by 
\[
u_{n}(\realLineCoor,\hypSpaceFirstCoor,\hypSpaceSecondCoor)=c_{n}^{\frac{1}{2}}e^{-\frac{1}{2}\Exponent_{n}^{2}\realLineCoor^{2}}\chi_{n}(\realLineCoor)\multipliedBy w_{n}(\hypSpaceFirstCoor,\hypSpaceSecondCoor)
\]
satisfies $\normComingFromInnerProduct{u_{n}}_{L^{2}(\mathbb{R}^{3})}\to\pi^{\frac{1}{4}}$
and 
\[
\normComingFromInnerProductThatAdapts{\left(\mathcal{I}\multipliedBy\schroedingerOperator_{\magneticFieldStrength_{x},\magneticFieldStrength_{y}}^{\solvableGeometry}\multipliedBy\mathcal{I}^{-1}-\frac{1}{2}\left(|\magneticFieldStrength_{x}|+|\magneticFieldStrength_{y}|^{2}-\frac{1}{4}\right)\right)u_{n}}_{L^{2}(\mathbb{R}^{3})}\to0.
\]
The first statement follows by dominated convergence as in~(\ref{eq:Norm_of_Quasi_Modes_for_EMSO_on_Sol}).
In order to prove the second statement, it suffices to show that
\[
\normComingFromInnerProduct{\hypSpaceSecondCoor^{-k}\derivative{\realLineCoor}u_{n}}_{L^{2}(\mathbb{R}\times\hyperbolicSpace)}\to0\qquad\textnormal{for }k=1,2.
\]
The corresponding computations can be carried out along the lines
of~(\ref{eq:Norm_of_exp_two_z_DxDx_f}), (\ref{eq:Norm_of_exp_two_z_DxDx_f2})
and (\ref{eq:Norm_of_exp_two_z_DxDx_f3}), where one additionally
uses that $\hypSpaceSecondCoor^{-1}\leq c_{n}^{-\frac{1}{2}}$ on
the support of~$w_{n}$.
\end{proof}
If one compares with~\cite{InahamaShirai2003}, it appears natural
to suspect that $\frac{1}{2}$$\left(\magneticFieldStrength-\frac{1}{4}\right)$
is an eigenvalue of $\closure{\schroedingerOperator_{\magneticFieldStrength_{x},\magneticFieldStrength_{y}}^{\solvableGeometry}}$
for $\magneticFieldStrength=\sqrt{\magneticFieldStrength_{x}^{2}+\magneticFieldStrength_{y}^{2}}>\frac{1}{2}$.
In the following, we briefly outline the spectral analysis of the
magnetic Schr\"odinger operator~(\ref{eq:EMSO_on_Sol}) on compact
quotients of $\solvableGeometry$ and their maximal abelian covers.
Let $A=(A_{ij})\in SL(2,\mathbb{Z})$ and $P=(P_{ij})\in GL(2,\mathbb{R})$
satisfy~(\ref{eq:Hyperbolic_Gluing_Map_Conjugated}), and denote
the corresponding cocompact lattice of $\solvableGeometry$ by $\Lambda_{A}$.
We have $A_{12}\neq0$ as $A$ has positive, distinct eigenvalues
$\lambda^{\pm1}=\frac{1}{2}\left(\mathrm{Tr}A\pm\sqrt{(\mathrm{Tr}A)^{2}-4}\right)\notin\mathbb{Q}$.
Since the columns of $P^{-1}$ are scalar multiples of the eigenvectors
$(A_{12},\lambda-A_{11})^{T}$ and $(A_{12},\lambda^{-1}-A_{11})^{T}$
of $A$, and since $\lambda^{\pm1}\notin\mathbb{Q}$, the only solution
to 
\begin{equation}
q_{1}\multipliedBy P_{11}+q_{2}\multipliedBy P_{12}=0\qquad\textnormal{or}\qquad q_{1}\multipliedBy P_{21}+q_{2}\multipliedBy P_{22}=0\label{eq:Only_zeros_give_continuous_spectrum}
\end{equation}
with $(q_{1},q_{2})\in\mathbb{Q}^{2}$ is given by $q_{1}=q_{2}=0$,
which we will use later on. In order to verify that $[\Lambda_{A},\Lambda_{A}]\backslash\solvableGeometry$
is homeomorphic to $\mathbb{T}^{2}\times\mathbb{R}$,  we note that
$[\Lambda_{A},\Lambda_{A}]$ is easily seen to be generated by 
\[
\left\{ \left(P\left(I-A^{l}\right)\left(\begin{array}{c}
m\\
n
\end{array}\right),0\right)\in\mathbb{R}^{3}\,\bigg|\,(l,m,n)\in\mathbb{Z}^{3}\right\} .
\]
In other words, $[\Lambda_{A},\Lambda_{A}]$ is a $\mathbb{Z}^{2}$-subgroup
of $\solvableGeometry$, and we can find $M=(M_{ij})\in GL(2,\mathbb{R})$
with integer entries $M_{ij}\in\mathbb{Z}$ such that
\[
a=M_{11}\left(\begin{array}{c}
P_{11}\\
P_{21}
\end{array}\right)+M_{21}\left(\begin{array}{c}
P_{12}\\
P_{22}
\end{array}\right)\qquad\textnormal{and}\qquad b=M_{12}\left(\begin{array}{c}
P_{11}\\
P_{21}
\end{array}\right)+M_{22}\left(\begin{array}{c}
P_{12}\\
P_{22}
\end{array}\right)
\]
form a $\mathbb{Z}^{2}$-basis of the corresponding lattice in $\mathbb{R}^{2}\simeq\mathbb{R}^{2}\times\{0\}\subset\solvableGeometry$.
We let $a^{*},b^{*}\in\left(\mathbb{R}^{2}\right)^{*}$ denote the
corresponding dual basis vectors, that is, 
\[
\left(\begin{array}{cc}
a_{1}^{*} & a_{2}^{*}\\
b_{1}^{*} & b_{2}^{*}
\end{array}\right)=\mathrm{Det}(PM)^{-1}\left(\begin{array}{cc}
M_{22} & -M_{12}\\
-M_{21} & M_{11}
\end{array}\right)\left(\begin{array}{cc}
P_{22} & -P_{12}\\
-P_{21} & P_{11}
\end{array}\right).
\]
Using the dual lattice $\mathbb{Z}a^{*}\oplus\mathbb{Z}b^{*}$, we
perform discrete Fourier analysis on $[\Lambda_{A},\Lambda_{A}]\backslash\solvableGeometry$.
The functions $(e_{m,n})_{m,n\in\mathbb{Z}}\subset\smoothFunctions{\mathbb{Z}a\oplus\mathbb{Z}b\backslash\mathbb{R}^{2}}$
given by 
\[
e_{m,n}(x,y)=e^{2\pi i\multipliedBy(ma^{*}+nb^{*})(x,y)}=e^{2\pi i\multipliedBy((ma_{1}^{*}+nb_{1}^{*})x+(ma_{2}^{*}+nb_{2}^{*})y)}
\]
yield a complete orthogonal set of smooth eigenfunctions of $\derivative x$
and $\derivative y$ in $L^{2}(\mathbb{Z}a\oplus\mathbb{Z}b\backslash\mathbb{R}^{2},dx\multipliedBy dy)$.
On each $\mathbb{C}e_{m,n}\times\smoothCompactlySupportedFunctions{\mathbb{R}}$,
the operator $\schroedingerOperator_{\magneticFieldStrength_{x},\magneticFieldStrength_{y}}^{[\Lambda_{A},\Lambda_{A}]\backslash\solvableGeometry}$
reduces to
\begin{equation}
\schroedingerOperator_{\magneticFieldStrength_{x},\magneticFieldStrength_{y},m,n}^{[\Lambda_{A},\Lambda_{A}]\backslash\solvableGeometry}=\frac{1}{2}\left((2\pi\multipliedBy(ma_{1}^{*}+nb_{1}^{*})\multipliedBy e^{z}+\magneticFieldStrength_{x})^{2}+(2\pi\multipliedBy(ma_{2}^{*}+nb_{2}^{*})\multipliedBy e^{-z}+\magneticFieldStrength_{y}){}^{2}+\derivative z^{2}\right).\label{eq:EMSO_on_Abel_Subcover_of_Sol}
\end{equation}
For $m=n=0,$ we obtain $\schroedingerOperator_{\magneticFieldStrength_{x},\magneticFieldStrength_{y},0,0}^{[\Lambda_{A},\Lambda_{A}]\backslash\solvableGeometry}=\frac{1}{2}\left(\magneticFieldStrength_{x}^{2}+\magneticFieldStrength_{y}^{2}+\derivative z^{2}\right)$
with purely continuous spectrum $\spectrum{\closure{\schroedingerOperator_{\magneticFieldStrength_{x},\magneticFieldStrength_{y},0,0}^{[\Lambda_{A},\Lambda_{A}]\backslash\solvableGeometry}}}=\left[\frac{1}{2}\magneticFieldStrength^{2},\infty\right)$.
With regard to $(m,n)\neq(0,0)$, note that, according to~(\ref{eq:Only_zeros_give_continuous_spectrum}),
we have $ma_{1}^{*}+nb_{1}^{*}=0$, that is, 
\[
(-n\multipliedBy M_{11}+m\multipliedBy M_{12})P_{21}+(-n\multipliedBy M_{21}+m\multipliedBy M_{22})P_{22}=0,
\]
if and only if $(-n,m)$ is in the kernel of $M$, that is, precisely
if $m=n=0$, and similarily for $ma_{2}^{*}+nb_{2}^{*}=0$. Hence,
any $(m,n)\neq(0,0)$ leads to a Schr\"odinger operator of the form
$\schroedingerOperator_{\magneticFieldStrength_{x},\magneticFieldStrength_{y},m,n}^{[\Lambda_{A},\Lambda_{A}]\backslash\solvableGeometry}=\frac{1}{2}\derivative z^{2}+\electricPotential(z)$
with $\electricPotential(z)\to\infty$ for $|z|\to\infty$. Such operators
are known to have purely discrete spectrum~\cite[Theorem XIII.16]{ReedSimon1978},
and we obtain 
\[
\spectrum{\closure{\schroedingerOperator_{\magneticFieldStrength_{x},\magneticFieldStrength_{y}}^{[\Lambda_{A},\Lambda_{A}]\backslash\solvableGeometry}}}=\left[\frac{1}{2}\magneticFieldStrength^{2},\infty\right)\cup\overline{\bigcup_{(m,n)\neq(0,0)}\purePointSpectrum{\closure{\schroedingerOperator_{\magneticFieldStrength_{x},\magneticFieldStrength_{y},m,n}^{[\Lambda_{A},\Lambda_{A}]\backslash\solvableGeometry}}}}.
\]
The spectral analysis on the compact quotient $\Lambda_{A}\backslash\solvableGeometry$
can be carried out as in~\cite[Section 5]{BolsinovDullinVeselov2006},
where the special case of the Laplacian $\Delta^{\Lambda_{A}\backslash\solvableGeometry}=2\multipliedBy\schroedingerOperator_{0,0}^{\Lambda_{A}\backslash\solvableGeometry}$
is considered. In particular, $\schroedingerOperator_{\magneticFieldStrength_{x},\magneticFieldStrength_{y}}^{\Lambda_{A}\backslash\solvableGeometry}$
allows for a similar decomposition in the $\mathbb{T}^{2}$-fibres
as its lift $\schroedingerOperator_{\magneticFieldStrength_{x},\magneticFieldStrength_{y}}^{[\Lambda_{A},\Lambda_{A}]\backslash\solvableGeometry}$.
We let $(w_{m,n,l})_{l\in\mathbb{N}}\subset\smoothFunctions{(\Lambda_{A}\cap\mathbb{R}^{2}\times\{0\})\backslash\solvableGeometry}$
denote the eigenfunctions of~(\ref{eq:EMSO_on_Abel_Subcover_of_Sol})
for $M_{11}=M_{22}=1$, $M_{12}=M_{21}=0$ and $(m,n)\neq(0,0)$.
In general, they do not descend to functions on $\Lambda_{A}\backslash\solvableGeometry$
since they are not invariant under $(x,y,z)\mapsto\left(\lambda x,\lambda^{-1}y,z+\log\lambda\right)$.
This can be overcome by averaging, namely, one considers

\[
u_{m,n,l}=\sum_{k\in\mathbb{Z}}w_{m,n,l}(\lambda^{k}x,\lambda^{-k}y,z+\log\lambda^{k})
\]
instead, where one has to use the exponential decay of the eigenfunctions
$(w_{m,n,l})_{l\in\mathbb{N}}$. For $m=n=0$, the $\Lambda_{A}$-invariant
eigenfunctions are easily seen to be of the form $u_{0,0,l}(x,y,z)=e^{2\pi ilz\multipliedBy(\log\lambda)^{-1}}$
with eigenvalues $\frac{1}{2}(\magneticFieldStrength_{x}^{2}+\magneticFieldStrength_{y}^{2}+4\pi^{2}l^{2}(\log\lambda)^{-2})$.
Finally, one can work along the proof of~\cite[Theorem 2]{BolsinovDullinVeselov2006}
to show that $(u_{m,n,l})_{m,n,l\in\mathbb{Z}}$ is a Hilbert basis
of~$L^{2}(\Lambda_{A}\backslash\solvableGeometry)$.

\subsubsection{Monopole case}

Bolsinov and Taimanov~\cite{BolsinovTaimanov2000} made the remarkable
discovery that the geodesic flow on compact quotients $\Lambda_{A}\backslash\solvableGeometry$
is completely integrable in the sense of Liouville despite having
non-zero topological entropy. This result motivated further study
of the classical and quantum dynamics on $\solvableGeometry$-quotients~\cite{BolsinovDullinVeselov2006}.
In particular, Butler and Paternain~\cite{ButlerPaternain2008} considered
the magnetic flow generated by a scalar multiple of the distinguished
monopole $dx\wedge dy$ in $\homologyGroup^{2}(\Lambda_{A}\backslash\solvableGeometry,\mathbb{R})$.
They showed that as soon as the magnetic field is turned on, the flow
gains positive Liouville entropy and therefore ceases being Liouville
integrable. Although the monopole $dx\wedge dy$ becomes exact when
lifted to the universal cover $\solvableGeometry$, none of its primitives
is bounded since $\pi_{1}(\Lambda_{A}\backslash\solvableGeometry)\simeq\Lambda_{A}$
is solvable and therefore amenable~\cite[Lemma 5.3]{Paternain2006}.
Hence, the critical value of the respective magnetic Hamiltonian on
$\solvableGeometry$ is infinite, see also~\cite{CieliebakFrauenfelderPaternain2010}.
In the following, we show how the spectral analysis of the corresponding
quantum system reduces to the study of Schr\"odinger operators on
$2$-dimensional hyperbolic space. Following~\cite{ButlerPaternain2008},
we let $\magneticFieldStrength\in\mathbb{R}$ and consider the left-invariant
magnetic field 
\[
\magneticField=\magneticFieldStrength\multipliedBy dx\wedge dy\in\smoothTwoForms{\solvableGeometry,\mathbb{R}}\qquad\textnormal{with potential}\qquad\magneticPotential=\magneticFieldStrength\multipliedBy x\multipliedBy dy\in\smoothOneForms{\solvableGeometry,\mathbb{R}}.
\]
The associated magnetic Schr\"odinger operator $\schroedingerOperator_{\magneticFieldStrength}^{\solvableGeometry}$
acts on its initial domain $C_{0}^{\infty}(\mathbb{R}^{3})$ as 
\begin{equation}
\schroedingerOperator_{\magneticFieldStrength}^{\solvableGeometry}=\frac{1}{2}\left(e^{2z}\multipliedBy\derivative x^{2}+e^{-2z}(\derivative y+\magneticFieldStrength\multipliedBy x){}^{2}+\derivative z^{2}\right).\label{eq:MSO_on_Sol}
\end{equation}

\begin{prop}
The spectrum of the closure of~(\ref{eq:MSO_on_Sol}) depends on
$\magneticFieldStrength$ as follows:
\begin{enumerate}
\item We have $\spectrum{\closure{\schroedingerOperator_{\magneticFieldStrength}^{\solvableGeometry}}}\subseteq\left[\frac{1}{2}\magneticFieldStrength,\infty\right).$
\item If $\hyperbolicSpace$ denotes the hyperbolic upper half-plane with
coordinates $(\hypSpaceFirstCoor,\hypSpaceSecondCoor)\in\mathbb{R}\times\mathbb{R}^{+}$
and metric $\hypSpaceSecondCoor^{-2}(d\hypSpaceFirstCoor^{2}+d\hypSpaceSecondCoor^{2})$,
then
\[
\spectrum{\closure{\schroedingerOperator_{\magneticFieldStrength}^{\solvableGeometry}}}=\spectrum{\closure{\schroedingerOperator_{0,\electricPotential_{\magneticFieldStrength}}^{\hyperbolicSpace}-\frac{1}{8}}}\quad\textnormal{with}\quad\electricPotential_{\magneticFieldStrength}(x_{\hypSpaceFirstCoor},\hypSpaceSecondCoor)=\frac{1}{2}\magneticFieldStrength^{2}\multipliedBy\frac{\hypSpaceFirstCoor^{2}}{\hypSpaceSecondCoor^{2}}.
\]

\end{enumerate}
\end{prop}
\begin{proof}
[Proof of 1.]We define creation and annihilation operators acting
on $\smoothCompactlySupportedFunctions{\mathbb{R}^{3}}$ as 
\begin{eqnarray*}
\mathcal{A} & = & i\multipliedBy e^{z}\multipliedBy\derivative x+e^{-z}(\derivative y+\magneticFieldStrength\multipliedBy x)\\
\mathcal{A}^{\dagger} & = & -i\multipliedBy e^{z}\multipliedBy\derivative x+e^{-z}(\derivative y+\magneticFieldStrength\multipliedBy x).
\end{eqnarray*}
Note that $\adjoint{\mathcal{A}}u=\mathcal{A}^{\dagger}u$ for $u\in\smoothCompactlySupportedFunctions{\mathbb{R}^{3}}$.
Since
\[
\mathcal{A}^{\dagger}\mathcal{A}=e^{2z}\multipliedBy\derivative x^{2}+e^{-2z}(\derivative y+\magneticFieldStrength\multipliedBy x)^{2}-\magneticFieldStrength,
\]
the claim follows from $2\multipliedBy\schroedingerOperator_{\magneticFieldStrength}^{\solvableGeometry}=\mathcal{A}^{\dagger}\mathcal{A}+\derivative z^{2}+\magneticFieldStrength$
by an application of~\cite[Theorem 4.3.1]{Davies1995}. We note that
$\schroedingerOperator_{\magneticFieldStrength}^{\solvableGeometry}$
has the so-called translation shape invariance~\cite{InahamaShirai2003}
\[
\mathcal{A}\mathcal{A}^{\dagger}+\derivative z^{2}+\magneticFieldStrength=\schroedingerOperator_{\magneticFieldStrength}^{\solvableGeometry}+2\multipliedBy\magneticFieldStrength.
\]

\emph{Proof of 2.} We reuse the isometry $\mathcal{I}\colon L^{2}(\mathbb{R}^{3},dx\multipliedBy dy\multipliedBy dz)\to L^{2}(\mathbb{R}\times\hyperbolicSpace,\hypSpaceSecondCoor^{-2}d\realLineCoor\multipliedBy d\hypSpaceFirstCoor\multipliedBy d\hypSpaceSecondCoor)$
given in (\ref{eq:Isometry_between_Sol_and_R_x_Hyp_Space}) to obtain
the following operator identity on $\smoothCompactlySupportedFunctions{\mathbb{R}\times\hyperbolicSpace}$
\[
\schroedingerOperator_{\magneticFieldStrength}^{\mathbb{R}\times\hyperbolicSpace}=\mathcal{I}\multipliedBy\schroedingerOperator_{\magneticFieldStrength}^{\solvableGeometry}\multipliedBy\mathcal{I}^{-1}=\frac{1}{2}\left(\hypSpaceSecondCoor^{2}\left(\derivative{\hypSpaceFirstCoor}^{2}+\derivative{\hypSpaceSecondCoor}^{2}\right)+\hypSpaceSecondCoor^{-2}\left(\derivative{\realLineCoor}+\magneticFieldStrength\multipliedBy\hypSpaceFirstCoor\right)^{2}-\frac{1}{4}\right).
\]
The linear span $D$ of $\mathcal{S}(\mathbb{R})\times\smoothCompactlySupportedFunctions{\hyperbolicSpace}$
is easily seen to be a core for $\closure{\schroedingerOperator_{\magneticFieldStrength}^{\mathbb{R}\times\hyperbolicSpace}}$
by using the canonical isometry $L^{2}(\mathbb{R})\otimes L^{2}(\hyperbolicSpace)\simeq L^{2}(\mathbb{R}\times\hyperbolicSpace)$,
cut-off functions, and dominated convergence. We conjugate by the
partial Fourier transformation $\mathcal{F}_{\realLineCoor}\colon D\to D$
given by
\[
\mathcal{F}_{\realLineCoor}u(\momentumVariable{\realLineCoor}{},\hypSpaceFirstCoor,\hypSpaceSecondCoor)=\frac{1}{\sqrt{2\pi}}\intop_{-\infty}^{\infty}u(\realLineCoor,\hypSpaceFirstCoor,\hypSpaceSecondCoor)\multipliedBy e^{-i\momentumVariable{\realLineCoor}{}\realLineCoor}\multipliedBy d\realLineCoor
\]
to obtain a direct integral decomposition of $\schroedingerOperator_{\magneticFieldStrength}^{\mathbb{R}\times\hyperbolicSpace}$
over $L^{2}(\mathbb{R},d\momentumVariable{\realLineCoor}{},L^{2}(\hyperbolicSpace,\hypSpaceSecondCoor^{-2}d\hypSpaceFirstCoor\multipliedBy d\hypSpaceSecondCoor))$
with operators
\[
\schroedingerOperator_{\magneticFieldStrength,\momentumVariable{\realLineCoor}{}}^{\mathbb{R}\times\hyperbolicSpace}=\frac{1}{2}\left(\hypSpaceSecondCoor^{2}\left(\derivative{\hypSpaceFirstCoor}^{2}+\derivative{\hypSpaceSecondCoor}^{2}\right)+\hypSpaceSecondCoor^{-2}\left(\momentumVariable{\realLineCoor}{}+\magneticFieldStrength\multipliedBy\hypSpaceFirstCoor\right)^{2}-\frac{1}{4}\right)
\]
acting on $\smoothCompactlySupportedFunctions{\hyperbolicSpace}$
in the $L^{2}(\hyperbolicSpace)$-fibres. For $\magneticFieldStrength\neq0$,
the operator $\schroedingerOperator_{\magneticFieldStrength,\momentumVariable{\realLineCoor}{}}^{\mathbb{R}\times\hyperbolicSpace}$
is conjugate to $\schroedingerOperator_{\magneticFieldStrength,0}^{\mathbb{R}\times\hyperbolicSpace}=\schroedingerOperator_{0,\electricPotential_{\magneticFieldStrength}}^{\hyperbolicSpace}-\frac{1}{8}$
with respect to the isometry $\mathcal{J}\colon L^{2}(\hyperbolicSpace)\to L^{2}(\hyperbolicSpace)$
given by $\mathcal{J}u(\hypSpaceFirstCoor,\hypSpaceSecondCoor)=u(\hypSpaceFirstCoor-\magneticFieldStrength^{-1}\momentumVariable{\realLineCoor}{},\hypSpaceSecondCoor)$.
Hence, $\spectrum{\closure{\schroedingerOperator_{0,\electricPotential_{\magneticFieldStrength}}^{\hyperbolicSpace}-\frac{1}{8}}}=\spectrum{\closure{\schroedingerOperator_{\magneticFieldStrength,\momentumVariable{\realLineCoor}{}}^{\mathbb{R}\times\hyperbolicSpace}}}$
for any $\momentumVariable{\realLineCoor}{}\in\mathbb{R}$, which
implies the statement by virtue of~\cite[Theorem XIII.85]{ReedSimon1978}.
For $\magneticFieldStrength=0$, we have Laplace operators whose spectra
are explicitly given in Theorem~\ref{thm:Spec_of_Hyp_Space} and
Theorem~\ref{thm:EMSO_on_Sol}.
\end{proof}

\subsection{The planar restricted $3$-body problem}

\global\long\def\closure#1{\overline{#1}}

Let $\baseManifold=\mathbb{R}^{2}\backslash\{(0,0)\}$ and $\magneticFieldStrength>0$.
We study the Hamiltonian $H_{\magneticFieldStrength}\colon T^{*}\baseManifold\to\mathbb{R}$
given by
\[
H_{\magneticFieldStrength}(x,p)=\frac{1}{2}\norm p^{2}-\frac{1}{\norm x}-\magneticFieldStrength\multipliedBy(x_{2}\multipliedBy p_{1}-x_{1}\multipliedBy p_{2})=\frac{1}{2}\left(\left(p_{1}-\magneticFieldStrength\multipliedBy x_{2}\right)^{2}+\left(p_{2}+\magneticFieldStrength\multipliedBy x_{1}\right)^{2}\right)-\frac{1}{\norm x}-\frac{\magneticFieldStrength^{2}}{2}\norm x^{2}.
\]
For $\magneticFieldStrength=1$, the system describes the Kepler problem
in rotating coordinates, that is, the planar restricted $3$-body
problem in a rotating frame with one of the primitives having zero
mass~\cite[Section 3]{AlbersFrauenfelderKoertPaternain2012}. Since
$\baseManifold$ has the isometry group $O(2)$, whose discrete subgroups
produce non-compact quotients, none of the developed theorems applies.
With respect to polar coordinates $(r,\varphi)\in\mathbb{R}^{+}\times\mathbb{S}^{1}$
given by $(x_{1},x_{2})=r\multipliedBy(\cos\varphi,\sin\varphi)$,
the Hamiltonian reads 
\[
H_{\magneticFieldStrength}^{\textrm{Polar}}(r,\varphi,p_{r},p_{\varphi})=\frac{1}{2}\left(p_{r}^{2}+\frac{p_{\varphi}^{2}}{r^{2}}\right)-\frac{1}{r}+\magneticFieldStrength\multipliedBy p_{\varphi}.
\]
Hence, $H_{\magneticFieldStrength}^{\textrm{Polar}}$ and $p_{\varphi}$
are integrals of motion. If $p_{\varphi}=0$, we have $H_{\magneticFieldStrength}^{\textrm{Polar}}=\frac{1}{2}p_{r}^{2}-\frac{1}{r}$,
which leads to $\dot{r}^{2}=2\multipliedBy(H_{\magneticFieldStrength}^{\textrm{Polar}}+\frac{1}{r})$.
Energy surfaces with $H_{\magneticFieldStrength}^{\textrm{Polar}}<0$,
respectively $H_{\magneticFieldStrength}^{\textrm{Polar}}>0$, are
easily seen to be compact, respectively non-compact, and $r(t)=(\frac{3}{\sqrt{2}}\multipliedBy t+r(0)^{\frac{3}{2}})^{\frac{2}{3}}$
is an unbounded solution with $H_{\magneticFieldStrength}^{\textrm{Polar}}=0$.
On the other hand, $p_{\varphi}\neq0$ leads to 
\[
H_{\magneticFieldStrength}^{\textrm{Polar}}(r,\varphi,p_{r},p_{\varphi})=\frac{1}{2}p_{r}^{2}+\frac{1}{2}\left(\frac{1}{p_{\varphi}}-\frac{p_{\varphi}}{r}\right)^{2}-\frac{1}{2}\frac{1}{p_{\varphi}^{2}}+\magneticFieldStrength\multipliedBy p_{\varphi},
\]
which has no solutions with energy less than $E_{0}=\magneticFieldStrength\multipliedBy p_{\varphi}-\frac{1}{2}p_{\varphi}^{-2}$.
The solutions with $H_{\magneticFieldStrength}^{\textrm{Polar}}=E_{0}$
correspond to $r=p_{\varphi}^{2}$ and $\dot{\varphi}=p_{\varphi}^{-3}+\magneticFieldStrength$.
Since
\[
2\multipliedBy p_{\varphi}^{2}\multipliedBy(H_{\magneticFieldStrength}^{\textrm{Polar}}-\magneticFieldStrength\multipliedBy p_{\varphi})=p_{\varphi}^{2}\multipliedBy p_{r}^{2}+\left(1-\frac{p_{\varphi}^{2}}{r}\right)^{2}-1,
\]
any orbit with $H_{\magneticFieldStrength}^{\textrm{Polar}}<\magneticFieldStrength\multipliedBy p_{\varphi}$
must be bounded, whereas energy surfaces with $H_{\magneticFieldStrength}^{\textrm{Polar}}\geq\magneticFieldStrength\multipliedBy p_{\varphi}$
are non-compact. In the following, we study the corresponding Schr\"odinger
operator with initial domain $\smoothCompactlySupportedFunctions{\baseManifold}$,
namely, 
\begin{equation}
\schroedingerOperator_{\magneticFieldStrength}=\frac{1}{2}\Delta^{\baseManifold}-\frac{1}{\norm x}+i\multipliedBy\magneticFieldStrength\multipliedBy\left(x_{2}\multipliedBy\frac{\partial}{\partial x_{1}}-x_{1}\multipliedBy\frac{\partial}{\partial x_{2}}\right).\label{eq:MSO_of_Keppler_Problem}
\end{equation}
Note that $\schroedingerOperator_{\magneticFieldStrength}$ is symmetric.
With respect to polar coordinates, it takes the form
\[
\schroedingerOperator_{\magneticFieldStrength}^{\textrm{Polar}}=-\frac{1}{2}\left(\frac{\partial^{2}}{\partial r^{2}}+\frac{1}{r}\frac{\partial}{\partial r}+\frac{1}{r^{2}}\frac{\partial^{2}}{\partial\varphi^{2}}\right)-\frac{1}{r}-i\multipliedBy\magneticFieldStrength\multipliedBy\frac{\partial}{\partial\varphi},
\]
and has domain $\smoothCompactlySupportedFunctions{\mathbb{R}^{+}\times\mathbb{S}^{1}}$.
Since $\schroedingerOperator_{\magneticFieldStrength}^{\textrm{Polar}}$
is spherically symmetric, we can separate variables as in~\cite[Section X.I, Example 4]{ReedSimon1975}.
Following the proof of Theorem~\ref{thm:Spec_of_Sphere_Bundle_of_Hyp_Space},
we let $D\subset\domain{\schroedingerOperator_{\magneticFieldStrength}^{\textrm{Polar}}}$
denote the set of finite linear combinations of products $u\multipliedBy w$
with $u\in\smoothCompactlySupportedFunctions{\mathbb{R}^{+}}$ and
$w\in\smoothFunctions{\mathbb{S}^{1}}$. Using the canonical isometry
$L^{2}(M)\simeq L^{2}(\mathbb{R}^{+},r\multipliedBy dr)\otimes L^{2}(\mathbb{S}^{1},d\varphi)$~\cite[Theorem II.10]{ReedSimon1975},
we see that $D$ is dense in $L^{2}(\baseManifold)$. Moreover, we
use the decomposition of $L^{2}(\mathbb{S}^{1},d\varphi)$ into eigenspaces
of $\frac{1}{i}\frac{\partial}{\partial\varphi}$ with eigenfunctions
$e_{m}(\varphi)=\frac{1}{\sqrt{2\pi}}e^{im\varphi}$ to obtain 
\[
L^{2}(\baseManifold)\simeq\bigoplus_{m=-\infty}^{\infty}L_{m}\qquad\textnormal{where}\qquad L_{m}=L^{2}(\mathbb{R}^{+},r\multipliedBy dr)\otimes\mathbb{C}\multipliedBy e_{m}.
\]
On each $D\cap L_{m}$, the operator $\schroedingerOperator_{\magneticFieldStrength}^{\textrm{Polar}}$
reduces to a one-dimensional Schr\"odinger operator acting on $\smoothCompactlySupportedFunctions{\mathbb{R}^{+}}$
as 
\[
\schroedingerOperator_{\magneticFieldStrength,m}^{\textrm{Polar}}=-\frac{1}{2}\left(\frac{d^{2}}{dr^{2}}+\frac{1}{r}\frac{d}{dr}-\frac{m^{2}}{r^{2}}\right)-\frac{1}{r}+\magneticFieldStrength\multipliedBy m.
\]
We use the isometry $\mathcal{I}\colon L^{2}(\mathbb{R}^{+},r\multipliedBy dr)\to L^{2}(\mathbb{R}^{+},dr)$
defined by $\mathcal{I}u(r)=r^{\frac{1}{2}}\multipliedBy u(r)$ to
remove the first order derivative. On $\mathcal{I}(\smoothCompactlySupportedFunctions{\mathbb{R}^{+}})=\smoothCompactlySupportedFunctions{\mathbb{R}^{+}}$,
we have 
\begin{equation}
\schroedingerOperator_{\magneticFieldStrength,m}=\mathcal{I}\multipliedBy\schroedingerOperator_{\magneticFieldStrength,m}^{\textrm{Polar}}\multipliedBy\mathcal{I}^{-1}=-\frac{1}{2}\frac{d^{2}}{dr^{2}}+\frac{m^{2}-\frac{1}{4}}{2r^{2}}-\frac{1}{r}+\magneticFieldStrength\multipliedBy m.\label{eq:One_dimensional_SO_without_first_Derivative}
\end{equation}

Since $\schroedingerOperator_{\magneticFieldStrength,m}$ commutes
with complex conjugation, a theorem by von Neumann~\cite[Theorem X.3]{ReedSimon1975}
implies that $\schroedingerOperator_{\magneticFieldStrength,m}$ has
self-adjoint extensions. If all $\schroedingerOperator_{\magneticFieldStrength,m}$
were essentially self-adjoint, then the same would hold for $\schroedingerOperator_{\magneticFieldStrength}$,
see also~\cite[Theorem VIII.33]{ReedSimon1975} and \cite[X Problem 1.(a)]{ReedSimon1978}.
However, $\schroedingerOperator_{\magneticFieldStrength,0}$ has non-zero
deficiency indices. We summarize the well-known properties of $\schroedingerOperator_{\magneticFieldStrength,m}$
in the following theorem.
\begin{thm}
Let $\schroedingerOperator_{\magneticFieldStrength,m}$ be the differential
operator (\ref{eq:One_dimensional_SO_without_first_Derivative}) with
domain $\smoothCompactlySupportedFunctions{\mathbb{R}^{+}}\subset L^{2}(\mathbb{R}^{+},dr)$.
\begin{enumerate}
\item Each $\schroedingerOperator_{\magneticFieldStrength,m}$ with $m\neq0$
is essentially self-adjoint, and the closure has eigenvalues
\[
\eigenvalue_{n}(\closure{\schroedingerOperator_{\magneticFieldStrength,m}})=-\frac{1}{2\left(n+m+\frac{1}{2}\right)^{2}}+\magneticFieldStrength\multipliedBy m,\qquad n=0,1,2,\ldots
\]
and essential spectrum $\essentialSpectrum{\closure{\schroedingerOperator_{\magneticFieldStrength,m}}}=[\magneticFieldStrength\multipliedBy m,\infty)$.
\item $\schroedingerOperator=\schroedingerOperator_{\magneticFieldStrength,0}$
has a one-dimensional family of self-adjoint extensions $\schroedingerOperator_{\nu}$
parametrized by $\nu\in\mathbb{R}\cup\{\infty\}$ such that
\begin{eqnarray*}
\schroedingerOperator_{\nu} & = & -\frac{1}{2}\frac{d^{2}}{dr^{2}}-\frac{1}{8}\frac{1}{r^{2}}-\frac{1}{r},\\
\domain{\schroedingerOperator_{\nu}} & = & \left\{ u\in L^{2}(\mathbb{R}^{+},dr)\multipliedBy\left|\multipliedBy\begin{array}{c}
u,u'\textrm{ locally absolutely continuous},\\
-\frac{1}{2}u''-\frac{1}{8}\frac{1}{r^{2}}\multipliedBy u-\frac{1}{r}\multipliedBy u\in L^{2}(\mathbb{R}^{+},dr),\quad\nu\multipliedBy u_{0}=u_{1}
\end{array}\right.\right\} ,
\end{eqnarray*}
with boundary values $u_{0}$ and $u_{1}$ defined as 
\[
u_{0}=\lim_{r\searrow0}-\frac{u(r)}{\sqrt{r}\multipliedBy\ln r}\qquad\textrm{and}\qquad u_{1}=\lim_{r\searrow0}\multipliedBy\frac{\left(u(r)+u_{0}\multipliedBy\sqrt{r}\multipliedBy\ln r\right)}{\sqrt{r}}.
\]
The boundary condition $u_{0}=0$ ($\nu=\infty$) gives the Friedrichs
extension $\schroedingerOperator_{\infty}$ with eigenvalues
\begin{equation}
\eigenvalue_{n}(\schroedingerOperator_{\infty})=-\frac{1}{2\left(n+\frac{1}{2}\right)^{2}},\qquad n=0,1,2,\ldots.\label{eq:Bohr_Levels}
\end{equation}
All self-adjoint extensions $\schroedingerOperator_{\nu}$ have the
same essential spectrum $\essentialSpectrum{\schroedingerOperator_{\nu}}=[0,\infty)$.
\end{enumerate}
\end{thm}
\begin{proof}
The statements follow from the analysis of so-called MIC-Kepler systems
in $\mathbb{R}^{3}$ carried out in~\cite[Section 3]{Giri2007},
where we have $\tilde{l}=j+\frac{\delta_{1}+\delta_{2}}{2}=m-\frac{1}{2}$.
In particular, $\schroedingerOperator_{\magneticFieldStrength,m}-\magneticFieldStrength\multipliedBy m$
with $m\neq0$ is essentially self-adjoint~\cite[(3.15)]{Giri2007},
and the closure has eigenvalues~\cite[(3.37)]{Giri2007} 
\[
\eigenvalue_{n}(\closure{\schroedingerOperator_{\magneticFieldStrength,m}-\magneticFieldStrength\multipliedBy m})=-\frac{1}{2(n+m+\frac{1}{2})^{2}},\qquad n=0,1,2,\ldots.
\]
The characterization of the self-adjoint extensions of $\schroedingerOperator=\schroedingerOperator_{\magneticFieldStrength,0}$
follows from~\cite[Theorem D.1]{AlbeverioGesztesyHoegh-KrohnHolden2005}
for $\lambda=\frac{1}{2}$, $\gamma=-2$, $\alpha=0$, $a=1$ and
$W=0.$ The eigenvalues~(\ref{eq:Bohr_Levels}) are given in~\cite[(3.34)]{Giri2007},
and an alternative proof may be found in~\cite[Theorem D.1]{AlbeverioGesztesyHoegh-KrohnHolden2005}.
In contrast to the discrete spectrum, Weyl's essential spectrum theorem~\cite[Theorem XIII.14]{ReedSimon1978}
shows that $\essentialSpectrum{\schroedingerOperator_{\nu}}$ is independent
from the chosen extension~\cite[Section XIII.4, Example 5]{ReedSimon1978}.
The inclusion $[\magneticFieldStrength\multipliedBy m,\infty)\subseteq\essentialSpectrum{\closure{\schroedingerOperator_{\magneticFieldStrength,m}}}$
follows as in~\cite[Theorem 8.3.1]{Davies1995} using explicit quasi-eigenfunctions
and the effective potential $\electricPotential_{eff}(r)=\frac{m^{2}-\frac{1}{4}}{2r^{2}}-\frac{1}{r}$,
which satisfies $\electricPotential_{eff}(r)\to0$ for $r\to\infty$.
A detailed proof of $\inf\essentialSpectrum{\closure{\schroedingerOperator_{\magneticFieldStrength,m}-\magneticFieldStrength\multipliedBy m}}\geq0$
can be found in the appendix of~\cite{Fewster1993}.
\end{proof}
The eigenvalues~(\ref{eq:Bohr_Levels}) are commonly referred to
as Bohr levels. Different values of $\nu\in\mathbb{R}$ lead to different
eigenvalues often denoted as
\[
\eigenvalue_{n}(\schroedingerOperator_{\nu})=-\frac{1}{2\left(n+\frac{1}{2}-\delta\right)^{2}},\qquad n=0,1,2,\ldots
\]
where $\delta$ is called Rydberg correction or quantum defect~\cite{Fewster1993}. 
\begin{cor}
\label{Last_Page_of_MSOaMCV}For any $\magneticFieldStrength\in\mathbb{R}$,
the closure $\closure{\schroedingerOperator_{\magneticFieldStrength}}$
has a one-dimensional family of self-adjoint extensions. The spectrum
of any such extension is the entire real axis.

\newpage{}
\end{cor}
{\footnotesize{\bibliographystyle{amsalpha}
\phantomsection\addcontentsline{toc}{section}{\refname}\bibliography{MSO_and_MCV}
}}{\footnotesize \par}

\noindent \texttt{\small{}}%
\begin{minipage}[t][5mm]{1\columnwidth}%
\noindent \noun{\small{\hfill{}Dartmouth College, Hanover, New Hampshire,
USA}}{\small \par}

\noindent \noun{\small{\hfill{}}}\emph{\small{E-mail address:}}{\small{
}}\texttt{\small{peter.herbrich@dartmouth.edu}}%
\end{minipage}
\end{document}